\documentclass[10pt,oneside,english]{amsart}
\RequirePackage{amsthm,amsmath,mathtools}

\usepackage[T1]{fontenc}
\usepackage[lmargin = 30mm, rmargin = 30mm, bmargin = 25mm, tmargin=25mm]{geometry}
\usepackage{subcaption}
\usepackage{babel,verbatim}
\usepackage{float}
\usepackage{amstext}
\usepackage{enumitem}
\usepackage{amssymb}
\usepackage{mathrsfs}
\usepackage{multirow}
\usepackage{dsfont}
\usepackage{array}
\usepackage[dvipsnames]{xcolor}
\usepackage{amsmath}
\usepackage[toc, page]{appendix}
\usepackage{graphicx}
\usepackage{multicol}
\RequirePackage[colorlinks,citecolor=blue,urlcolor=blue]{hyperref}

\usepackage[backend=biber,style=numeric,citestyle=numeric,maxbibnames=99,sortcites=true]{biblatex}
\renewbibmacro{in:}{}
\appto{\bibsetup}{\sloppy}

\PassOptionsToPackage{subsection=false,section=false}{beamerouterthememiniframes}
\usepackage{tikz}
\usetikzlibrary{shapes.multipart,shapes,arrows,decorations.pathreplacing}
\usepackage{algorithm,algpseudocode}
\usepackage{lineno}
\usepackage{pgfplots}
\usepackage{caption,bbm}

\newcommand{\Oh}{\mathcal{O}}

\newcommand{\D}{\mathrm{d}}
\newcommand{\mS}{\mathcal{S}}
\newcommand{\mL}{\mathcal{L}}
\newcommand{\e}{\mathds{E}}

\newcommand{\R}{\mathbb{R}}
\newcommand{\Complex}{\mathbb{C}}
\newcommand{\N}{\mathbb{N}}
\newcommand{\p}{\mathds{P}}

\newcommand{\U}{\mathrm{U}}
\newcommand{\1}{\mathds{1}}
\newcommand{\wt}[1]{\widetilde{#1}}
\newcommand{\wh}[1]{\widehat{#1}}
\newcommand{\Loc}{L^1_{\mathrm{loc}}}
\newcolumntype{L}{>{\centering\arraybackslash}m{7cm}}
\newcolumntype{T}{>{\centering\arraybackslash}m{4cm}}
\newcolumntype{S}{>{\centering\arraybackslash}m{3.2cm}}
\newcolumntype{M}{>{\centering\arraybackslash}m{3.5cm}}
\newcommand{\s}{\mathfrak{s}}

\newcommand{\ve}{\varepsilon}
\newcommand{\vs}{\varsigma}

\newcommand{\da}{\downarrow}
\newcommand{\ua}{\uparrow}

\newcommand{\Leb}{\text{\normalfont Leb}}

\newcommand{\ov}[1]{\overline{#1}}

\newcommand{\eqd}{\overset{d}{=}}

\newcommand{\nf}[1]{\normalfont{#1}}

\providecommand{\algorithmname}{Algorithm}
\providecommand{\assumptionname}{Assumption}
\providecommand{\examplename}{Example}
\providecommand{\lemmaname}{Lemma}
\providecommand{\propositionname}{Proposition}
\providecommand{\remarkname}{Remark}
\providecommand{\corollaryname}{Corollary}
\providecommand{\theoremname}{Theorem}
\providecommand{\conjecturename}{Conjecture}

\numberwithin{equation}{section}
\numberwithin{figure}{section}
\theoremstyle{plain}

\theoremstyle{plain}
\newtheorem{thm}{\protect\theoremname}[section]
\theoremstyle{remark}
\newtheorem*{rem*}{\protect\remarkname}

\theoremstyle{remark}
\newtheorem{remx}[thm]{\protect\remarkname}
\newenvironment{rem}
    {\pushQED{\qed}\remx}
    {\popQED\endremx}

\theoremstyle{plain}
\newtheorem*{assumption*}{\protect\assumptionname}
\theoremstyle{plain}
\newtheorem{lem}[thm]{\protect\lemmaname}
\theoremstyle{plain}
\newtheorem{cor}[thm]{\protect\corollaryname}
\theoremstyle{plain}
\newtheorem{prop}[thm]{\protect\propositionname}
\theoremstyle{plain}
\newtheorem{conj}[thm]{\protect\conjecturename}
\theoremstyle{definition}
\newtheorem*{ex*}{\protect\examplename}
\newtheorem{exx}[]{\protect\examplename}[section]
\newenvironment{ex}
    {\pushQED{\qed}\exx}
    {\popQED\endexx}

\usepackage{caption}

\makeatletter
\@namedef{subjclassname@2020}{%
  \textup{2020} Mathematics Subject Classification}
\makeatother

\AtBeginDocument{%
   \def\MR#1{}
}

\begin{document}

\title{When is the convex hull of a L\'evy path smooth?}

\author{David Bang, Jorge Gonz\'{a}lez C\'{a}zares \and Aleksandar Mijatovi\'{c}}

\address{Department of Statistics, University of Warwick, \& The Alan Turing Institute, UK}

\email{david.bang@warwick.ac.uk}

\email{jorge.i.gonzalez-cazares@warwick.ac.uk}

\email{a.mijatovic@warwick.ac.uk}

\begin{abstract}
We characterise, in terms of their transition laws, the class of one-dimensional L\'evy processes whose graph has a continuously differentiable (planar) convex hull. We show that this phenomenon is exhibited by a broad class of infinite variation L\'evy processes and depends subtly on the behaviour of the L\'evy measure at zero. We introduce a class of strongly eroded L\'evy processes, whose Dini derivatives vanish at every local minimum of the trajectory for all perturbations with a linear drift, and prove that these are precisely the processes with smooth convex hulls. We study how the smoothness of the convex hull can break and construct examples exhibiting a variety of smooth/non-smooth behaviours. Finally, we conjecture that an infinite variation L\'evy process is either strongly eroded or abrupt, a claim implied by Vigon's point-hitting conjecture. In the finite variation case, we characterise the points of smoothness of the hull in terms of the L\'evy measure.
\end{abstract}



\maketitle

\section{Introduction and main results}
\label{sec:introduction}

Convex hulls of stochastic processes, including Brownian motion and L\'evy processes, have been the focus of many studies for decades, see e.g.~\cite{MR2831081,MR1747095,MR205343,MR714964,MR972777,MR2978134} and the references therein. The boundary of the convex hull of the range of a planar Brownian motion consists of piecewise linear segments but is well-known to be smooth (i.e. continuously differentiable) everywhere~\cite{MR972777}. The convex hull of a graph of a path of a standard Cauchy process  also possesses a smooth 
%continuously %differentiable 
boundary almost surely~\cite{MR1747095}, a fact not easily discerned from the simulation in Figure~\ref{fig:convex_hull} below (but \textit{cf.} discussion following Theorem~\ref{thm:CM_global_smoothness} below). Since the law of the graph of the standard Cauchy process scales linearly in time, it is natural to ask whether smoothness of the hull occurs at all for L\'evy process without a linear scaling property (note that the range of a planar Brownian motion also possesses a temporal scaling property). In this paper we characterise (in terms of transition laws) what turns out to be a rich and interesting class of L\'evy processes whose graphs have smooth convex hulls almost surely. We study its properties by analysing how the smoothness of the hull of a graph may fail for a general L\'evy process (see \href{https://www.youtube.com/watch?v=qPxBqaq2AsQ&list=PLPpwtaET-J4mauZQ6dlcp0i6lp4e3ycJZ}{YouTube}~\cite{Presentation_AM} for a short presentation of our results).

\begin{figure}[ht]
\centering
\includegraphics[width=.4\textwidth]{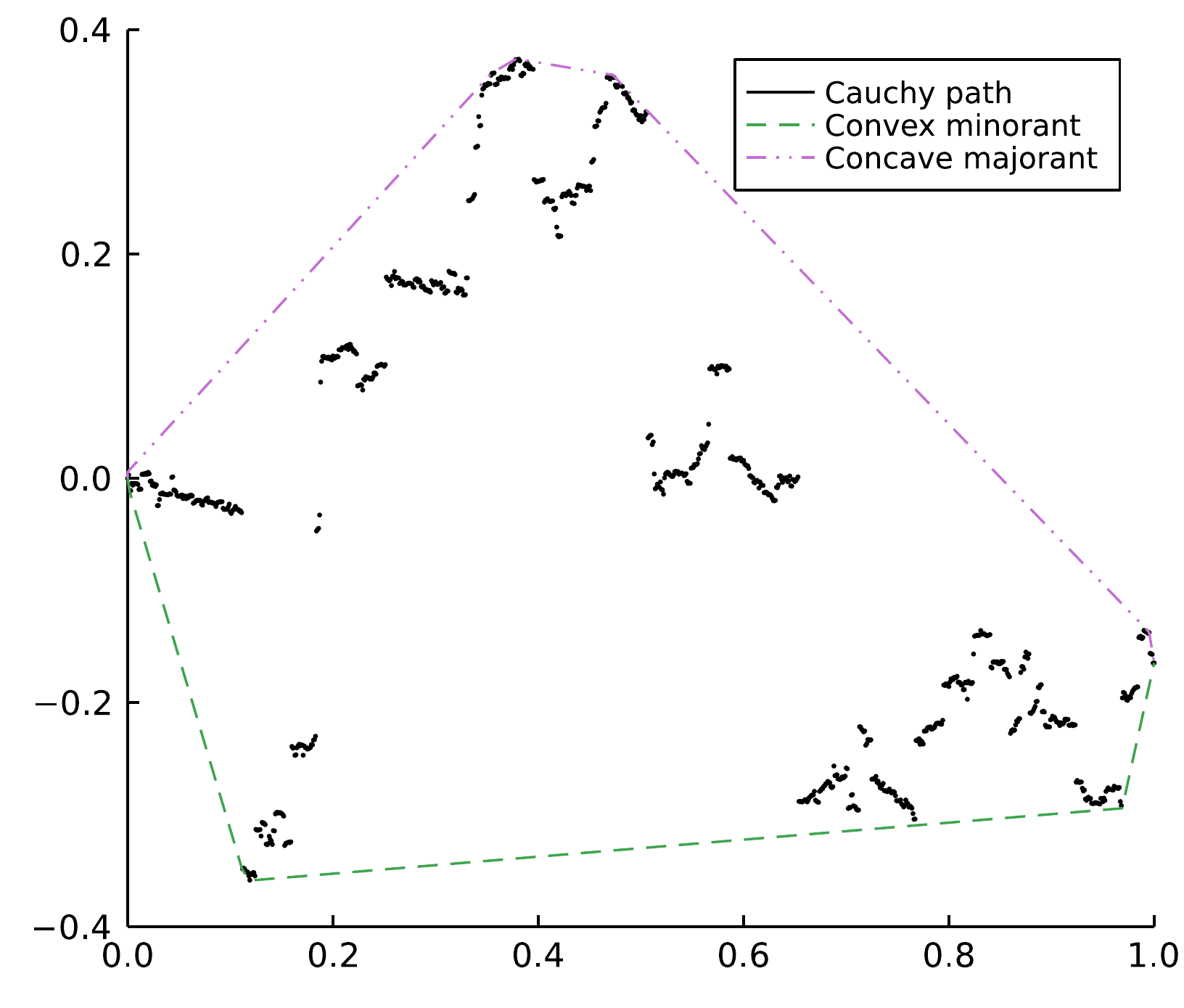}
\caption{\small The path of a Cauchy process and its convex hull.}
\label{fig:convex_hull}
\end{figure}

The boundary of the convex hull of the graph of a L\'evy process $X$ (see~\cite{MR3185174} for background on L\'evy processes) over a finite interval $[0,T]$ is a union of the graphs of the convex minorant and the concave majorant, i.e. the {\color{Green} largest convex} and {\color{Purple} smallest concave} functions dominating the path pointwise from below and above, respectively (see Figure~\ref{fig:convex_hull}).
The minorant and majorant are piecewise linear functions for any L\'evy process $X$. The set of slopes $\mS\subset \R$ of the convex minorant, which has the same law as the set of slopes of the concave majorant (see e.g.~\cite[Thm~11]{CM_Fluctuation_Levy}),
plays a key role in the question of smoothness of the boundary of the hull. 
%Whether there are infinitely many slopes in any  measurable subset of $\R$ is crucial %for determining the smoothness of the hull.
% Throughout the rest of the paper we assume the %L\'evy process $X$ is not compound Poisson with %drift, making its convex minorant $C$ on $[0,T]$ %have infinitely many faces %by~\cite[Thm~12]{CM_Fluctuation_Levy}. Since the %law of $X_t$ is diffuse %by~\cite[Lem.~15.22]{MR1876169}, we can see %from~\cite[Thm~12]{CM_Fluctuation_Levy} that, for %any $x\in\R$ we have $\p(x\in\mS)=0$. In other %words, no specific point is in $\mS$ with positive %probability. Despite this fact, there are some %properties of $\mS$ that are characterised in the 
Unless otherwise stated, $X$ is not compound Poisson with drift (in this case $\mS$ is clearly finite and the smoothness cannot occur), making $X_t$ diffuse for all $t>0$~\cite[Lem.~15.22]{MR1876169} and thus the cardinality of the set of slopes $\mS$ infinite~\cite[Thm~11]{CM_Fluctuation_Levy}.
The following zero-one law characterises the local finiteness of $\mS$ in terms of the increments of $X$. The characterisation holds for all L\'evy processes $X$ with diffuse increments. 
%and does not depend on the time horizon $T\in(0,\infty)$.

\begin{thm}
\label{thm:CM_local_smoothness}
For any measurable set $I\subseteq\R$, the set $\mS\cap I$ is either a.s. finite or a.s. infinite. Moreover, the cardinality $|\mS\cap I|$ of the intersection $\mS\cap I$ is infinite almost surely if and only if 
\begin{equation}
\label{eq:integral_slope}
\int_0^1 \p(X_t/t\in I)\frac{\D t}{t}=\infty.
\end{equation}
\end{thm}

Theorem~\ref{thm:CM_local_smoothness} follows  from a novel zero-one law for stick-breaking processes in Theorem~\ref{thm:0-1_law_formula} and Corollary~\ref{cor:slope_criteria}, established in Section~\ref{sec:0-1lawsbres} below, and the characterisation of the law of convex minorant of a L\'evy process (see e.g.~\cite{CM_Fluctuation_Levy}). Since the set of slopes of the concave majorant of the path of $X$ has the law of $\mS$, Theorem~\ref{thm:CM_local_smoothness} can be used to establish implies the following characterisation of the smoothness of the convex hull.

\begin{thm}
\label{thm:CM_global_smoothness}
The boundary of the convex hull of the graph 
$t\mapsto(t,X_t)$, $t\in[0,T]$, 
of a path of any L\'evy process $X$ 
is continuously differentiable (as a closed curve in $\R^2$) a.s. if and only if~\eqref{eq:integral_slope} holds for all bounded intervals $I$ in $\R$. Moreover, this is equivalent to the set $\mS$ being dense in $\R$ a.s. 
\end{thm}

It is clear that if the set of slopes $\mS$ is not dense in $\R$, the convex hull cannot possess a smooth boundary. Indeed, a gap in $\mS$ (i.e. an open interval contained in the complement $\R\setminus\mS$) results in the jump of the derivative of the convex minorant and concave majorant (see Section~\ref{subsec:infinite_var_proofs} below for the proof of Theorem~\ref{thm:CM_global_smoothness}). Intuitively, as suggested by the simulation in Figure~\ref{fig:convex_hull}, $\mS$ is dense if every contact point of $X$ with the boundary of the hull is both preceded and followed by infinitely many contact points between the path and the boundary.
More generally, Theorem~\ref{thm:CM_local_smoothness} (applied to intervals $I=(a,b)$ with rational  $a<b$) implies that, interestingly, the set $\mL(\mS)$ of the accumulation points (see Appendix~\ref{sec:pw-cf} for definition)
%\footnote{A point $x\in\R$ is an \textit{accumulation} (or \textit{limit}) \textit{point} of $\mS$ if every neighborhood %of $x$ intersects $\mS\setminus\{x\}$.}  
of the random set $\mS$ is almost surely constant for any L\'evy process $X$. Theorem~\ref{thm:CM_global_smoothness} thus states that the convex hull of the path of $X$ has a smooth boundary if and only if $\mL(\mS)=\R$ a.s. 
%A L\'evy process $X$ satisfying this condition is said to be \emph{eroded}.
Note that the criterion in Theorem~\ref{thm:CM_local_smoothness} depends neither on the time horizon $T$ nor (by the L\'evy-It\^o decomposition of $X$) on the behaviour of the L\'evy measure of $X$ on the complement of any neighbourhood of zero, even though the set of slopes $\mS$ does depend on both. 
%Thus the class of L\'evy processes with convex hulls that have smooth boundaries is well-defined and depends only on the fine behaviour of the L\'evy measure around zero.

It was conjectured (without proof) in~\cite[Rem.~3.4.4]{MR3152597} that if the paths of $X$ have infinite variation, then $\mS$ has finitely many points on every interval $(a,b)$ if and only if~\eqref{eq:integral_slope} fails for all $a<b$. This is implied by Theorem~\ref{thm:CM_local_smoothness} above and is furthermore equivalent to $\mL(\mS)=\emptyset$ a.s. Moreover, as we will see below (Proposition~\ref{prop:fin_var_criteria}), since $X$ is not compound Poisson with drift, $\mL(\mS)=\emptyset$ a.s. in fact implies that $X$ must be of infinite variation.

%Vigon introduced the following classes of infinite variation L\'evy processes in~\cite{VigonConjecture,MR1947963}. 
An infinite variation process $X$ is \textit{abrupt} if the following Dini derivatives are infinite at every local minimum $t$ of the path of $X$,
%$\limsup_{\ve\ua0}(X_{t+\ve}-X_{t-})/\ve=0$ and $\liminf_{\ve\downarrow0}(X_{t+\ve}-X_{t})/\ve=0$ (resp. 
$\limsup_{\ve\ua0}(X_{t+\ve}-X_{t-})/\ve=-\infty$ and $\liminf_{\ve\downarrow0}(X_{t+\ve}-X_{t})/\ve=\infty$, where $X_{t-}$ is the left limit of the trajectory at time $t$. The notion of abruptness, introduced by Vigon in his PhD thesis~\cite[Def.~12.1.1]{vigon:tel-00567466} (see also~\cite[Def.~1.1]{MR1947963}), 
captures L\'evy processes that approach and leave very rapidly each local minimum of their trajectory. Interestingly, the main result~\cite[Thm~1.3]{MR1947963} states that an infinite variation process $X$ is abrupt if and only if condition~\eqref{eq:integral_slope} fails for all intervals in $\R$. Since the minimum is a contact point between the path of $X$ and its convex minorant, abrupt L\'evy processes are unlikely to have smooth convex minorants. In fact, the criteria in Theorem~\ref{thm:CM_global_smoothness} and~\cite[Thm~1.3]{MR1947963} imply that a L\'evy process $X$ is abrupt if and only if $\mL(\mS)=\emptyset$ a.s.  An \emph{eroded} L\'evy process $X$ defined in~\cite[Def.~1.2]{VigonConjecture} (see also~\cite[App.~D, p.~10]{vigon:tel-00567466}) has infinite variation and the following Dini derivatives equal to zero at every local minimum $t$ of the path of $X$, $\limsup_{\ve\ua0}(X_{t+\ve}-X_{t-})/\ve=0$ and $\liminf_{\ve\downarrow0}(X_{t+\ve}-X_{t})/\ve=0$. These processes approach and leave their local minima very slowly and are good candidates to possess a smooth convex minorant. However, it follows from~\cite[Thm~1.4]{VigonConjecture} (which can be also be derived from~\cite[Prop~3.6]{MR1947963}) and Theorem~\ref{thm:CM_local_smoothness} that an infinite variation L\'evy process is eroded if and only if $0\in\mL(\mS)$ a.s. is approached continuously by the slopes of the minorant from \textit{both} sides, i.e. $0\in\mL^-(\mS)\cap\mL^+(\mS)$ a.s. (see Appendix~\ref{sec:pw-cf} below for the definition of $\mL^\pm(\mS)$).

It is clear that, if $X$ is abrupt, then $(X_t-rt)_{t\ge 0}$ is also abrupt for any $r\in\R$. However, this invariance may fail for an eroded process. We define a L\'evy process $X$ to be \emph{strongly eroded} if $(X_t-rt)_{t\ge 0}$ is eroded for every $r\in\R$, which is equivalent to $\mL^-(\mS)\cap\mL^+(\mS)=\R$ a.s. Since the interior of $\mL(\mS)$ is contained in $\mL^-(\mS)\cap\mL^+(\mS)\subseteq\mL(\mS)$ by definition, the process $X$ is strongly eroded if and only if $\mL(\mS)=\R$ a.s. or, equivalently, if the boundary of its convex hull is smooth. Since the respective criteria on the law of $X$ in Theorem~\ref{thm:CM_global_smoothness} and~\cite[Thm~1.3]{MR1947963} are \textit{not} complementary, an interesting question, closely related to Vigon's point-hitting conjecture discussed below (see Conjecture~\ref{conj:vigon}), is which (if any) infinite variation processes satisfy~\eqref{eq:integral_slope} for some bounded intervals $I$ but not for others. In particular, are there any eroded processes that are not strongly eroded? (See  Section~\ref{subsec:where_And_how_smoothness_fails} below for further discussion of these questions.) 
 
The class of strongly eroded L\'evy processes, defined in terms of the transition probabilities by the criterion in Theorem~\ref{thm:CM_global_smoothness}, has a rich structure. For example, it contains families of processes with symmetric and asymmetric L\'evy measures, including a standard Cauchy process but excluding all non-standard Cauchy (i.e. weakly 1-stable) processes with asymmetric L\'evy measures, see Section~\ref{sec:examples&specialcases} below. Moreover, a strongly eroded process has no Gaussian component (by Theorem~\ref{thm:zeta_criterion}(ii-a)) and, since it satisfies $\mL(\mS)=\R$ a.s., has paths of infinite variation (by Proposition~\ref{prop:fin_var_criteria}). Its Blumenthal--Getoor index is thus greater or equal to one while the related index $\beta_-$, defined in~\eqref{eq:Pruitt_index} below (cf.~\cite{MR632968}), is less or equal to one (see Proposition~\ref{prop:index_criteria}). More generally, for any strongly eroded L\'evy process $X$, the L\'evy process $X+Y$ is strongly eroded for any L\'evy process $Y$ of finite variation %independent of $X$ 
(see Proposition~\ref{prop:fin_var_perturbation} below). In contrast, if $Y$ and $X$ are both strongly eroded and independent of each other, the L\'evy process $X+Y$ need not (but, of course, could) be strongly eroded, see Example~\ref{ex:addition2} below. The properties of strongly eroded L\'evy processes will be discussed in more detail in the remainder of Section~\ref{sec:introduction} and in Section~\ref{sec:examples&specialcases}.
 
It is natural to attempt to construct the non-random set of limit slopes $\mL(\mS)$ directly from the characteristics of an arbitrary L\'evy process $X$. It turns out that, if $X$ is of finite variation, $\mL(\mS)$ is a singleton given by the natural drift of the process, see Table~\ref{tab:summary} below for an overview of our results. In the infinite variation case, we characterise $\mL(\mS)$ up to Conjecture~\ref{conj:dichotomy} stated below, which is implied by Vigon's point-hitting conjecture~\cite[Conj.~1.6]{VigonConjecture}  (see the discussion of  Conjectures~\ref{conj:dichotomy} and~\ref{conj:vigon} below).
%If the paths of the L\'evy process are of finite %variation, we give a complete answer
%to this question in terms of its characteristic triplet.
%If the process has paths of infinite variation, the %smoothness of the boundary of the convex hull is closely %related to Vigon's ``point-hitting'' %conjecture~\cite[Conjecture.~1.6]{VigonConjecture}, %stated in Conjecture~\ref{conj:vigon} below. 
More precisely, if Vigon's conjecture were true,
our sufficient condition for $X$ to be strongly eroded (i.e. $\mL(\mS)=\R$ a.s.), given in terms of the characteristic exponent of $X$, %for the derivative of the boundary not to be continuous 
would also be necessary, and its complement would imply abruptness (i.e. $\mL(\mS)=\emptyset$~a.s.). Moreover,
via Orey's process in Example~\ref{ex:orey} below, if Conjecture~\ref{conj:vigon} were true, there would exist   
%if the conjecture in~\cite[Conjecture.~1.6]{VigonConjecture} were true, 
a strongly eroded L\'evy processes whose path variation is arbitrarily close to two. This is in contrast to all known examples of strongly eroded L\'evy processes, which turn out to have Blumenthal--Getoor index equal to one (see Section~\ref{sec:examples&specialcases} below). However, if Vigon's conjecture is not true, there would exist an
infinite variation L\'evy process with slopes of the convex minorant accumulating at some deterministic values but not at others. Differently put, in this case the non-random set $\mL(\mS)$ would be a proper closed subset of $\R$, implying that kinks in the boundary of the hull would constitute a proper subset of the contact points between the boundary and the closure (in $\R^2$) of the graph of the path. As it is not easy to imagine a boundary of the hull of the path being smooth in some regions but not in others, our results could perhaps be viewed as further evidence for Vigon's point-hitting conjecture.

\subsection{Where and how  does the continuous differentiability of the boundary fail?}  
\label{subsec:where_And_how_smoothness_fails}
This change of perspective sheds light on where the smoothness features of the boundary discussed above come from. 
%If the conjecture holds, we provide an answer to the %question in terms of the characteristic %exponent of the %L\'evy process. A potential counterexample to Vigon't %conjecture would provide a %process with very interesting %behaviour of the convex hull. We discuss this below what %such an %example would.
Before stating our results in detail, we give an overview in
Table~\ref{tab:summary}. Let $C:[0,T]\to\R$ denote the piecewise linear convex minorant of the path of $X$
on $[0,T]$, see e.g.~\cite[Thm~11]{CM_Fluctuation_Levy}. Its right-derivative %of $C$, denoted by 
$C':(0,T)\to\R$ is a non-decreasing piecewise constant function with image $\{C'(x)|0<x<T\}\supset\mS$. Differently put, for 
every $r\in\mS$ there exists a maximal open interval $I_r\subset(0,T)$, satisfying $C'(I_r)=\{r\}$. 
Note that, in general, $C'$ may but need not be discontinuous at a boundary point of $I_r$. However, as an increasing right-continuous process, its path is completely determined by the set $\mS$ of values on the dense set $\bigcup_{r\in\mS}I_r$ and its discontinuities are in a one-to-one relationship with the gaps of $\mS$.

The second derivative (as a distribution) is given by a positive Radon measure $\D C'$ on $(0,T)$. Since the set of slopes of the concave majorant and the convex minorant have the same law, the second derivative of the concave majorant has the same law as the (negative) Radon measure $-\D C'$. Thus, the derivative of the boundary of the convex hull over the open interval $(0,T)$ is discontinuous at a point if and only if the point is an atom of the measure $\D C'$. Over the set $\{0,T\}\subset[0,T]$, the discontinuity of the boundary occurs if and only if
the derivative $C'$ is either bounded  below or above.

%Table~\ref{tab:summary} summarises the results in this paper, detailed descriptions are given in %Subsections~\ref{subsec:fin_var_fin_time}~\&~\ref{subsec:infin_var_fin_time}. 
\begin{table}[ht]
\begin{tabular}{|S|L|T|} 
	\hline
	L\'evy process $X$ 
	& Derivative $C'$ and the limt set $\mL(\mS)$
	& Measure $\D C'$\\
	\hline
%	Compound Poisson with drift & $C'$ is discontinuous on $V$ & Discrete with a %finite number of atoms \\ \hline
	Finite variation & $C'$ bounded below \textit{and} above; $C'$ discontinuous on boundary $\partial I_r$, $\forall r\in \mS$; $\mL(\mS)=\{\gamma_0\}$, where $\gamma_0=\lim_{t\downarrow 0}X_t/t$ a.s., and $\gamma_0\notin \mS$
	% is natural drift of $X$; 
	%possibly continuous at the unique accumulation point of $V$ and is discontinuous at every %other point of $V$ 
	& atomic; atoms accumulate from left/right or from both sides at a unique (random) accumulation point in $[0,T]$
	%Discrete with unique point that is possibly left/right accumulation of atoms or both 
	\\ \hline
	Infinite variation \& $\s_1$ locally integrable %$\in\Loc(r)$, $\forall r\in \R$  
	& $C'$ discontinuous on  boundary  $\partial I_r$, $\forall r\in \mS$; $-\lim_{t\downarrow 0}C'(t)=\lim_{t\ua T}C'(t)=\infty$; $\mL(\mS)=\emptyset$ & atomic; atoms accumulate only at  $0$ and $T$ \\ \hline
	Infinite variation \&  $\s_1(r)=\infty$,  $\forall r\in \R$  & $C'$ is continuous on $(0,T)$; $-\lim_{t\downarrow 0}C'(t)=\lim_{t\ua T}C'(t)=\infty$; $\mL(\mS)=\R$ & singular continuous \\ \hline
\end{tabular}
\caption{Summary of the regularity results of the convex minorant $C:[0,T]\to\R$ of any L\'evy process $X$ of infinite activity (i.e. not compound Poisson with drift) over time horizon $T$. %Possible behaviours of $C'$ and $\D C'$.
The function $\s_1$, defined in~\eqref{eq:s_q} in terms of the generating triplet of $X$, is either locally integrable or everywhere infinite under Conjecture~\ref{conj:vigon}. By Theorem~\ref{thm:zeta_criterion} below, this conjecture is known to hold for most L\'evy processes.}
\label{tab:summary}
\end{table}

%Theorem~\ref{thm:CM_local_smoothness} characterises the limit sets of $\mS$ in terms of the probability $\p(X_t/t \in I)$. This probability is often not tractable; however, an equivalent condition can be given in terms of 
%Denote the generating triplet  of the L\'evy process $X$ by $(\gamma,\sigma^{2},\nu)$ %(see~\cite[Def.~8.2]{MR3185174})and 
Let $\psi$ be the L\'evy-Khintchine exponent~\cite[Def.~8.2]{MR3185174} of the L\'evy process $X$. %by removing all of its jumps of magnitude at least one. Put differently, the L\'evy measure in $\psi$ is the L\'evy measure of $X$ restricted to the set $(-1,1)\setminus\{0\}$.
Note that $\Re \psi(u)\leq0$ implies $\Re (1/(1+iur-\psi(u)))>0$ for all $r,u\in\R$,
where $\Re z$ is the real part of a complex $z\in\Complex$ and $i^2=-1$. 
%$\psi(u)\coloneqq \log\e e^{iuX_{1}}$, $u\in\R$. 
Hence $0<\s_1(r)\leq\infty$ for any $r\in\R$,
where
%\[
%\psi(u)
%\coloneqq t^{-1}\log\e e^{iuX_{t}}
%=iu\gamma - \frac{1}{2}\sigma^2u^2 + %\int_\R\big(e^{iux}-1-iux\1_{\{|x|<1\}}\big)\nu(\D x),\quad\text{for $u \in \R$, $t\ge 0$.}
%\] 
\begin{equation}
\label{eq:s_q}
\s_1(r)
\coloneqq\frac{1}{2\pi}\int_\R
\Re\frac{1}{1+iur-\psi(u)}\D u,
\qquad r\in\R.
\end{equation} 
%(see also~\cite[Thm~1.5]{VigonConjecture}). 
%Since the criterion in~\eqref{eq:integral_slope}
%does not depend on the large jumps of $X$, 
The identity in 
Theorem~\ref{thm:Vigon} below (first established in~\cite{VigonConjecture} for L\'evy processes with bounded jumps) yields the following equivalence for any real $a<b$:
 %(see also~\cite[Thm~1.5]{VigonConjecture})
\begin{equation}\label{eq:vigon_identity}
\int_a^b \s_1(r)\D r<\infty 
\qquad\text{if and only if}\qquad
\int_0^1 \p(X_t/t\in (a,b))\frac{\D t}{t}<\infty.
\end{equation}
By definition, $\s_1\in \Loc(r)$ 
if and only if $\s_1$ is Lebesgue integrable on a neighbourhood of $r\in\R$.
Thus $r\in\mL(\mS)$  is (by~\eqref{eq:vigon_identity} and Theorem~\ref{thm:CM_local_smoothness}) equivalent to the  condition $\s_1\notin\Loc(r)$, involving only  the characteristic exponent of $X$ and not the law of the increment of $X$. For example, the presence of a non-trivial Brownian component implies that $\mL(\mS)=\emptyset$ and, equivalently, that $\mS$ is locally finite a.s.

The sum of the lengths of the open maximal intervals  $\{I_s:s\in\mS\}$ of constancy of $C'$ equals the time horizon $T$. As suggested by Table~\ref{tab:summary}, the random Radon measure $\D C'$, supported on the complement of $\bigcup_{s\in\mS} I_s$, must therefore be singular a.s. with respect to the Lebesgue measure. 
Thus the Lebesgue decomposition of
$\D C'$
is in general a sum of an atomic and a purely singular continuous components. Moreover, if  Conjecture~\ref{conj:vigon} holds, 
only one of these summands is non-trivial for any L\'evy process. 

The convex minorant of a  compound Poisson process with drift has only finitely many faces, making its derivative necessarily discontinuous at all boundary points of its maximal intervals of constancy.  
If $X$ has infinite activity but finite variation, then $C'$ is bounded on $[0,T]$ and discontinuous at every 
point in $\bigcup_{s\in\mS}\partial I_s$,
but is possibly continuous at the single (random) accumulation point of this set (see Proposition~\ref{prop:fin_var_criteria} below and discussion thereafter for details). If $X$ has too much jump activity or a Brownian component, then $C'$ is discontinuous at every point in  $\bigcup_{s\in\mS}\partial I_s$ and infinite on $\{0,T\}$ (see Proposition~\ref{prop:index_criteria} below and the discussion thereafter for details). Hence, for $C$ to be continuously differentiable on the open interval $(0,T)$ (i.e., for $X$ to be strongly eroded), the process $X$ must have sufficient jump activity to be of infinite variation, but not too much, as its index $\beta_-$~\eqref{eq:Pruitt_index} must be bounded above by one. These features are discussed in more detail in the following subsections (see Table~\ref{tab:dictionary_C'} in Appendix~\ref{sec:pw-cf} for all possible behaviours of the right-derivative of a piecewise linear convex function).

\begin{figure}[ht]
\centering
\includegraphics[width=0.48\textwidth]{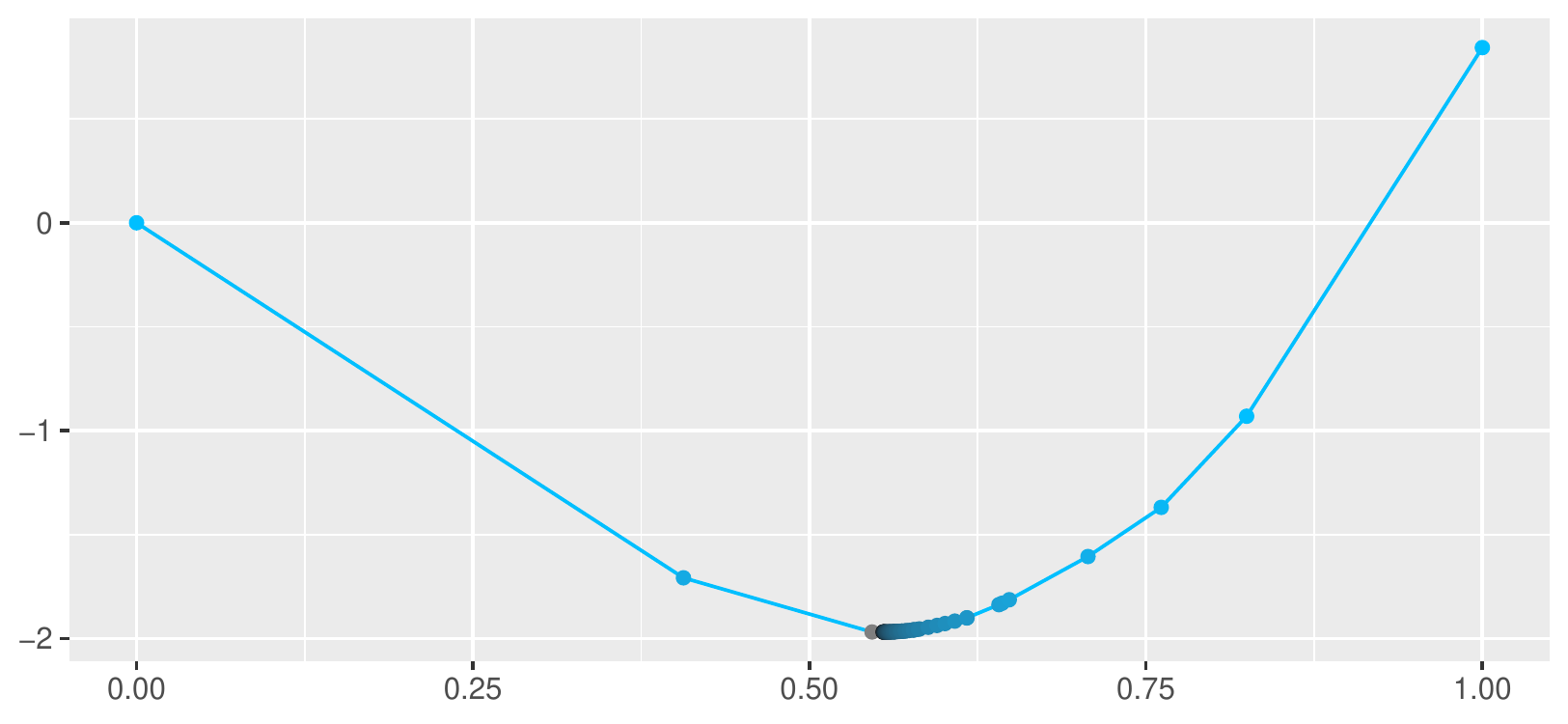}
\includegraphics[width=0.48\textwidth]{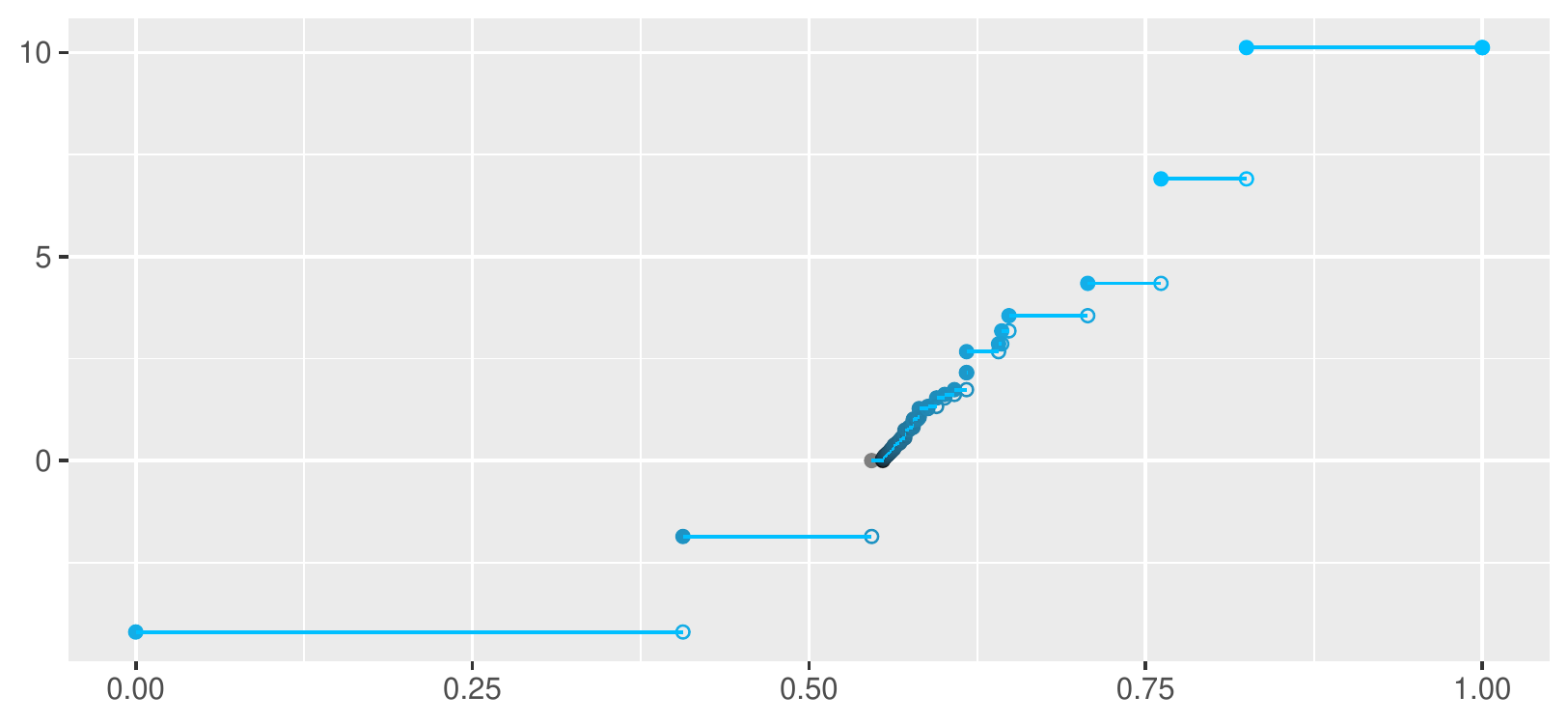}
\includegraphics[width=0.48\textwidth]{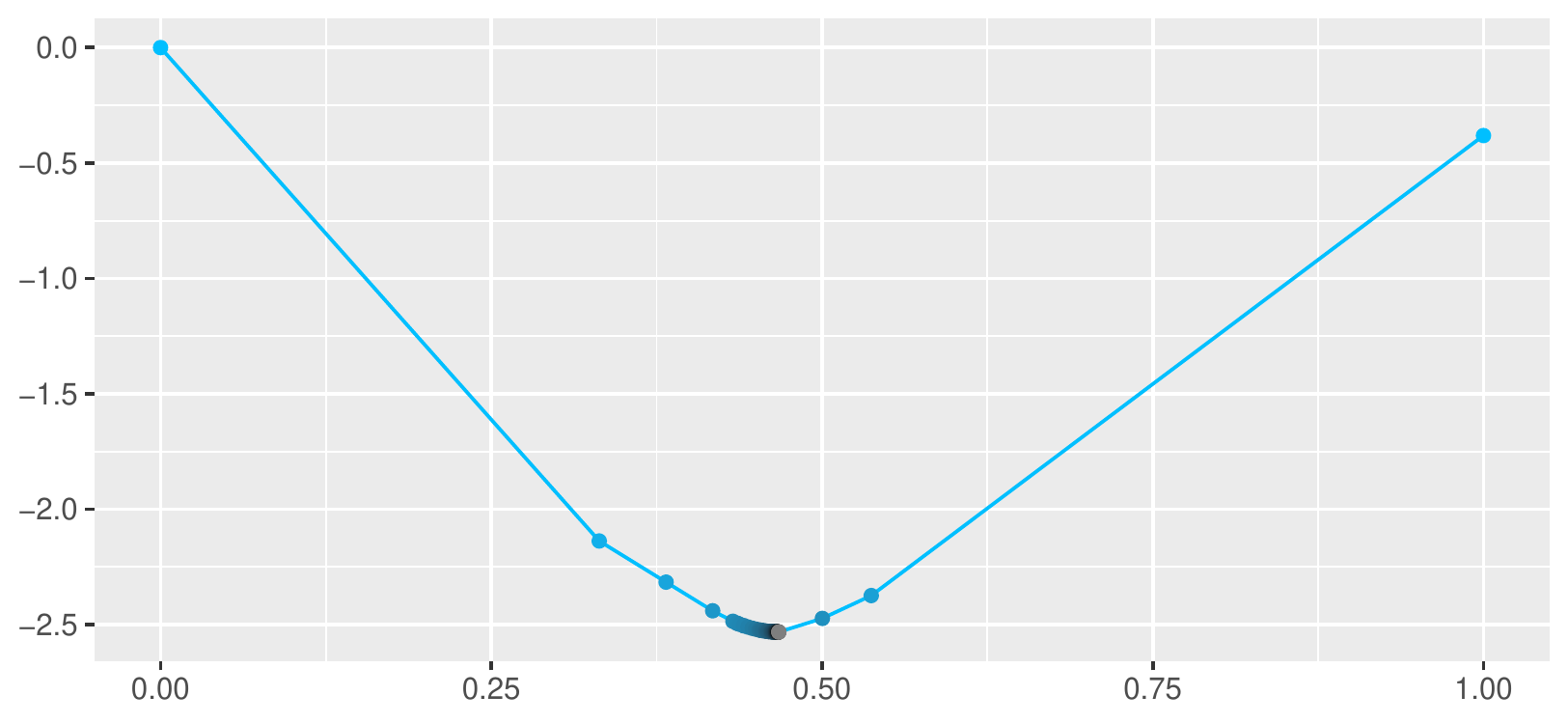}
\includegraphics[width=0.48\textwidth]{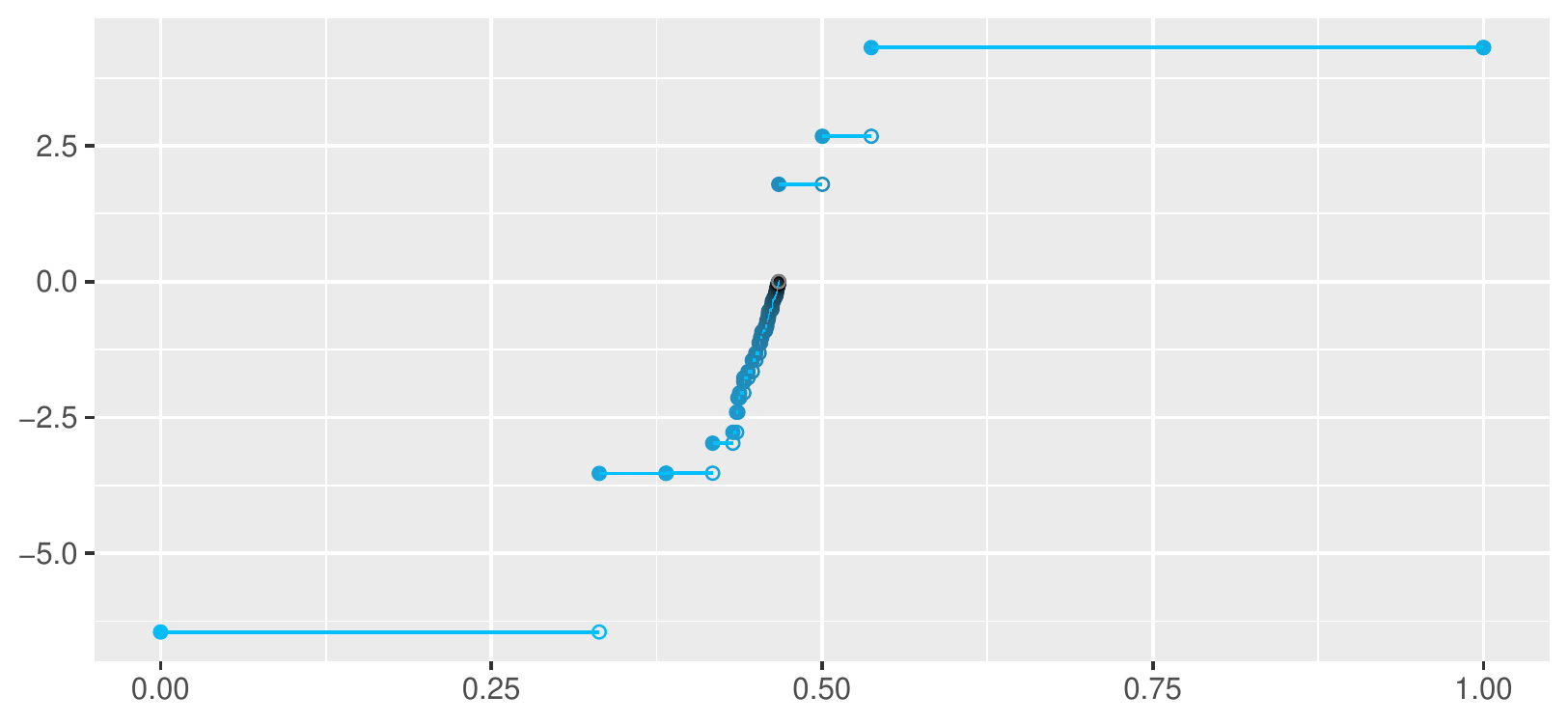}
\includegraphics[width=0.48\textwidth]{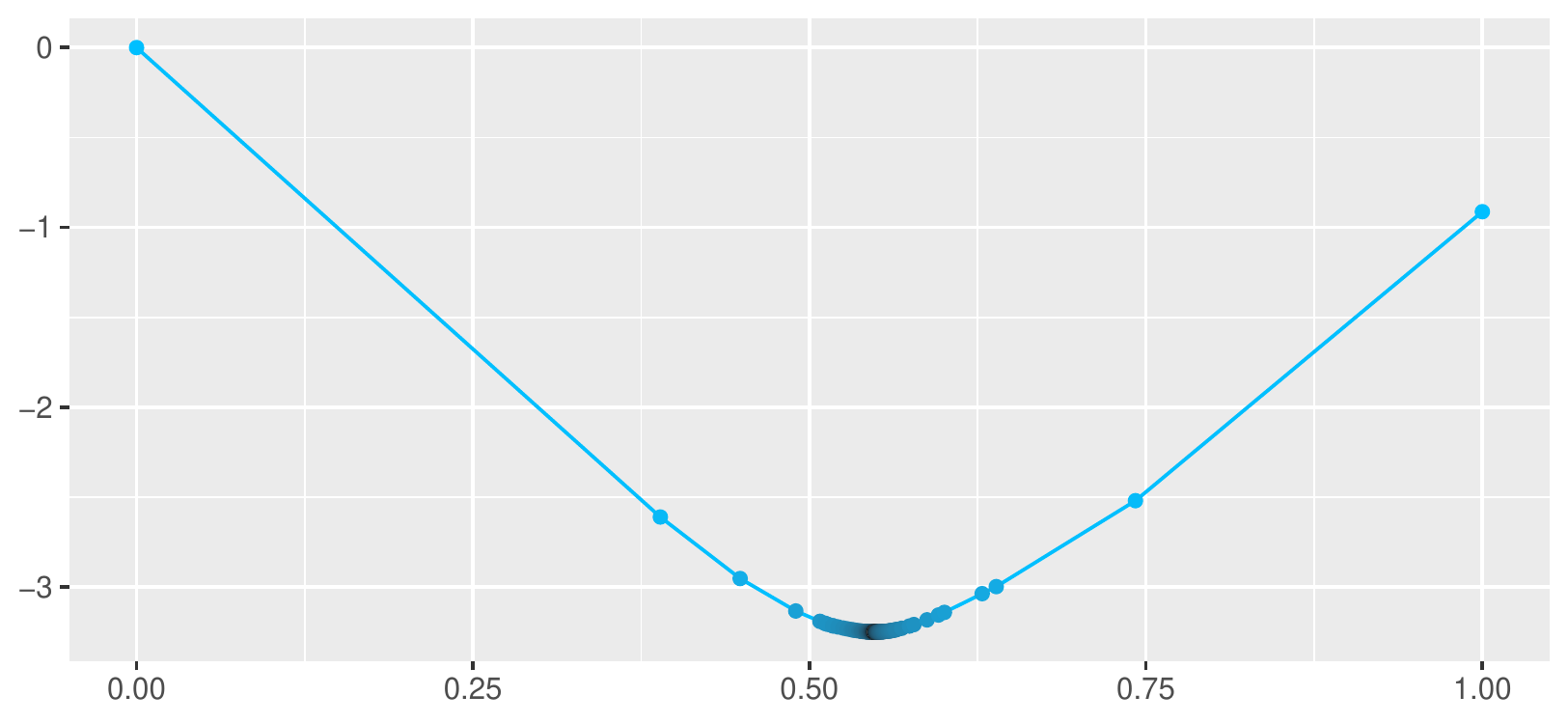}
\includegraphics[width=0.48\textwidth]{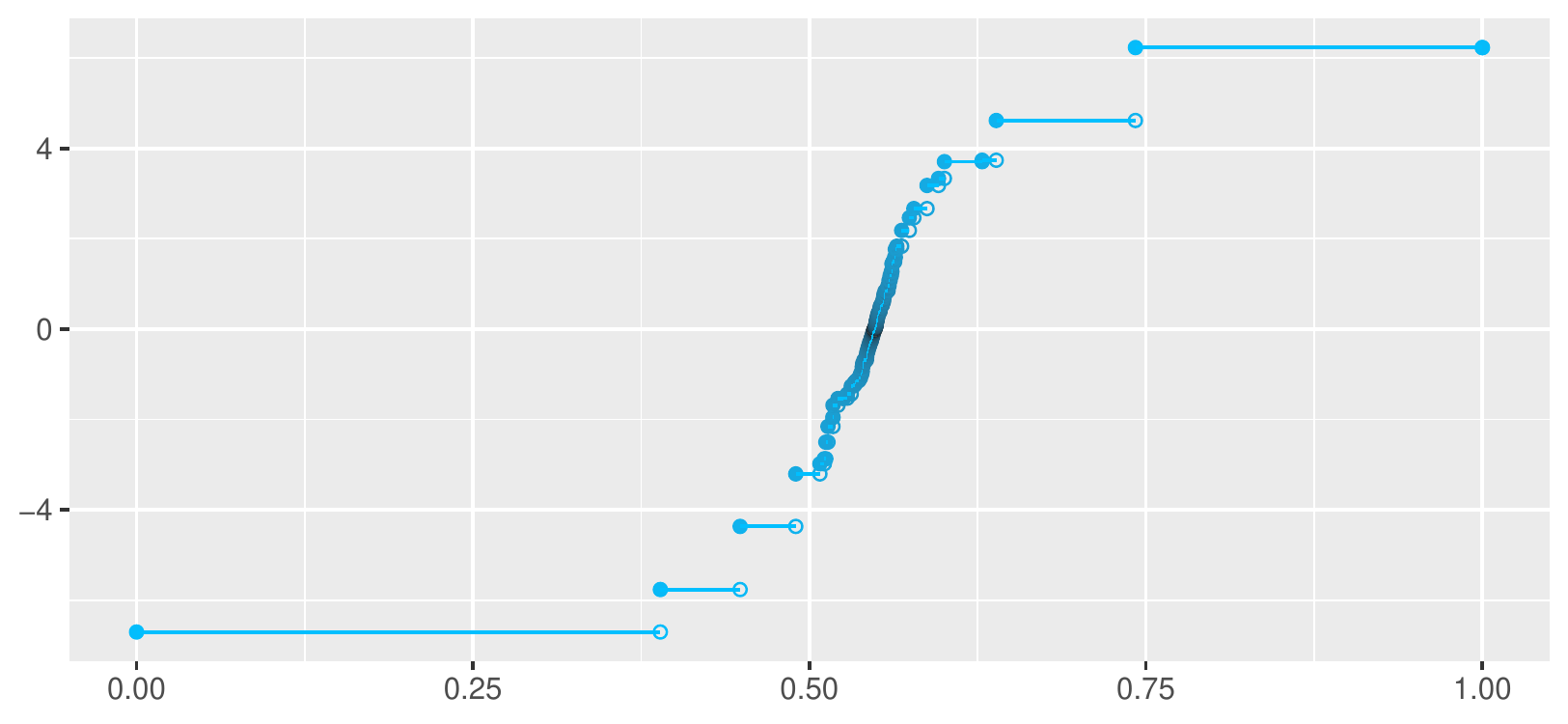}
\caption{\small Behaviour of $C$ and $C'$ for a finite variation infinite activity process $X$. The left panels graph the piecewise linear convex function $C$ (circles mark contact points with the path of $X$). The panels on the right graph the corresponding right-continuous derivative $C'$ (jump-size equals the difference of consecutive slopes). The top (resp. middle) panels correspond to right-accumulation (resp. left-accumulation) in Prop.~\ref{prop:fin_var_criteria}(b) (resp. Prop.~\ref{prop:fin_var_criteria}(c)). Note that in both of these cases, $C'$ is only right-continuous at the accumulation point of jump-times. The bottom panels depict two-sided accumulation 
in Prop.~\ref{prop:fin_var_criteria}(a), making $C'$  continuous at the unique accumulation point of the jump-times.}
\label{fig:point_continuous}
\end{figure}

\subsubsection{Finite variation}
\label{subsec:fin_var_fin_time}
Throughout this subsection we assume $X$ has finite variation but infinite activity. Let~$\gamma_0$ denote the natural drift of $X$ defined in terms of the characteristics in~\eqref{eq:natural_drift} below. Since, by Doeblin's  lemma~\cite[Lem.~15.22]{MR1876169},
it follows that $\p(X_t=\gamma_0 t)=0$ for all $t>0$,
the integrals
\begin{equation}
\label{eq:I_pm}
I_-\coloneqq\int_0^1 \p(X_t<t\gamma_0)\frac{\D t}{t}
\qquad\text{and}\qquad
I_+\coloneqq\int_0^1 \p(X_t>t\gamma_0)\frac{\D t}{t}
\end{equation}
satisfy $I_-+I_+=\infty$, implying that at least one is infinite. %, whenever $X$ has infinite activity. 
Moreover, the integrals $I_\pm$ are given in terms
of the law of a pure-jump L\'evy process $(X_t-\gamma_0t)_{t\geq0}$, uniquely determined by the L\'evy measure of $X$. 
%SATO has near description in terms of triplet %in~\cite[Thm~45.7]{MR3185174}
Let $\mL^+(\mS)$ (resp. $\mL^-(\mS)$)
be the set of right-accumulation 
%\footnote{A point $x\in\R$ is a \emph{right-accumulation}  (or \emph{right-limit}) point of $\mS$ if every neighborhood %of $x$ in $\R$ intersects $\mS\cap(x,\infty)$.} (resp. left-accumulation) points of $\mS$
(see Appendix~\ref{sec:pw-cf} for definition).
Equivalence~\eqref{eq:vigon_identity} and Theorem~\ref{thm:CM_local_smoothness}
imply that $s\in\mL^+(\mS)$ (resp. $s\in\mL^-(\mS)$) if and only if $\s_1\notin\Loc(s+)$ (resp. $\s_1\notin\Loc(s-)$), where  a function $f$  is in $\Loc(s+)$ (resp. $\Loc(s-)$) if and only if $f\cdot \1_{(s,\infty)}\in\Loc(s)$ (resp. $f\cdot\1_{(-\infty,s)}\in\Loc(s)$). In particular, like $\mL(\mS)$, the limit sets $\mL^+(\mS)$ and $\mL^-(\mS)$ are also constant a.s. 

\begin{prop}
\label{prop:fin_var_criteria}
Let $X$ have infinite activity and finite variation. Then the derivative $C'$ (and thus the set of slopes $\mS$) is bounded from below and above. $C'$ is discontinuous on $\bigcup_{r\in\mS}\partial I_r$, where $I_r$ is the maximal interval of constancy of $C'$ corresponding to slope $r$, and 
the limit set of slopes is a singleton $\mL(\mS)=\{\gamma_0\}$ (the natural drift $\gamma_0$ of $X$ is defined in~\eqref{eq:natural_drift}). Time $v\in[0,T]$ at which the process $(X_t-\gamma_0t)_{t\geq0}$ attains its minimum in $[0,T]$ is a.s. unique. If $v>0$, denote the left limit of $C'$ at $v$ by $C'(v-):=\lim_{t\ua v}C'(t)$. Then we have:
\begin{itemize}[leftmargin=2em, nosep]
    \item[{\nf(a)}] if $I_+=I_-=\infty$, then $v\in(0,T)$,  $\mL^-(\mS)=\mL^+(\mS)=\{\gamma_0\}$ and $C'(v-)=C'(v)=\gamma_0$ a.s.;
    %\label{prop:fin_var_criteria(a)}
    \item[{\nf(b)}] if $I_-<\infty$, then $v\in[0,T)$ with $\p(v=0)>0$, $\mL^-(\mS)=\emptyset$, $\mL^+(\mS)=\{\gamma_0\}$, $C'(v)=\gamma_0$ a.s. and, on the event $\{v>0\}$, we have $C'(v-)<\gamma_0$ a.s.;
    %\label{prop:fin_var_criteria(b)}
    \item[{\nf(c)}] if $I_+<\infty$, then $v\in(0,T]$ with $\p(v=T)>0$, $\mL^+(\mS)=\emptyset$, $\mL^-(\mS)=\{\gamma_0\}$, $C'(v-)=\gamma_0$ a.s. and, on the event $\{v< T\}$, we have $C'(v)>\gamma_0$ a.s.
    %\label{prop:fin_var_criteria(c)}
\end{itemize}
\end{prop}

By Rogozin's criterion (see e.g.~\cite[Thm~6]{CM_Fluctuation_Levy}),
the integral conditions in Proposition~\ref{prop:fin_var_criteria}(a) are equivalent to $0$ being regular for both half-lines for the process $(X_t-\gamma_0t)_{t\geq0}$. In particular, for the two conditions to hold concurrently it is necessary (but not sufficient) for $X$ to exhibit infinitely many positive and negative jumps in any finite time interval~\cite[Thm~47.5]{MR3185174} and sufficient (but not necessary) for $X$ to be spectrally symmetric. The other cases are also possible since $I_-<\infty$ (resp. $I_+<\infty$) if $X$ is the difference of two stable subordinators and the positive (resp. negative) jumps have a larger stability index. The following corollary is a simple consequence of  Proposition~\ref{prop:fin_var_criteria}, equivalence~\eqref{eq:vigon_identity} and~\cite[Thm~1]{MR1465812}. It characterises which case in Proposition~\ref{prop:fin_var_criteria} occurs in terms of either the L\'evy measure $\nu$ or the characteristic exponent $\psi$ (via $\s_1$ defined in~\eqref{eq:s_q}) of the process $X$.

\begin{cor}
\label{cor:I_pm}
Let $X$ have infinite activity but finite variation with natural drift $\gamma_0$. Then
\[%begin{align*}
I_\pm=\infty
\enskip\iff\enskip
\s_1\notin\Loc(\gamma_0\pm)
\enskip\iff\enskip
%\nu((-\infty,0))<\infty\enskip\text{or}\enskip
\int_{(-1,1)} \frac{\max\{\pm x,0\}}{\int_0^{\max\{\pm x,0\}} \ov\nu_\mp(y)\D y}\nu(\D x)=\infty,%\\
%&I_-=\infty
%\enskip\iff\enskip
%\s_1\notin\Loc(\gamma_0-)
%\enskip\iff\enskip
%\nu((0,\infty))<\infty\enskip\text{or}\enskip
%\int_{(-1,0)} \frac{|x|}{\int_0^{|x|}\nu((y,\infty))\D y}\nu(\D x)=\infty,
\]%end{align*}
where we define $\ov\nu_+(x):=\nu((x,\infty))$ and $\ov\nu_-(x):=\nu((-\infty,-x))$ for $x>0$ and  $\mp:=-(\pm)$.
\end{cor}

Proposition~\ref{prop:fin_var_criteria} and Corollary~\ref{cor:I_pm} give a complete description 
(in terms of the characteristics of any infinite activity, finite variation L\'evy process $X$) of 
how the continuity of the derivative $C'$ fails, see  Figure~\ref{fig:point_continuous} for all possible behaviours. 
The proof of Proposition~\ref{prop:fin_var_criteria} is based on the criterion in Theorem~\ref{thm:CM_local_smoothness}
and the crucial fact that, as $t\to0$, the quotient $X_t/t$ a.s. stops visiting closed intervals that do not contain  the natural drift $\gamma_0$ (since $\lim_{t\to0}X_t/t=\gamma_0$ a.s.), see Section~\ref{subsec:finite_var_proofs} below for details.
The smoothness of the minorant in the infinite variation case is more intricate precisely because in that case  we do not have a good understanding in general of how frequently the quotient $X_t/t$ visits intervals in $\R$ as $t\to0$.

\subsubsection{Infinite variation}
\label{subsec:infin_var_fin_time}
%Throughout this subsection we assume $X$ has infinite variation. In contrast to the 
The set of slopes $\mS$ is unbounded on both sides
for any $X$ of infinite variation, i.e. $\sup \mS =-\inf \mS=\infty$, and hence $-\lim_{t\downarrow 0}C'(t)=\lim_{t\ua T}C'(t)=\infty$ a.s. (cf. Figure~\ref{fig:all_continuous}). Indeed, by Rogozin's theorem~\cite[Thm~47.1]{MR3185174}, asserting that $X_t/t$ takes arbitrarily large positive and negative values at arbitrarily small times $t$ for any $X$ of infinite variation, it is impossible for the convex minorant to start at $0$ or end at $T$ with a finite slope: $-\infty=\liminf_{t\da 0} X_t/t\ge \inf\mS$, and, by time reversal, $\infty=\limsup_{t\da 0} (X_T-X_{T-t})/t\le\sup\mS$.  
Thus the boundary of the convex hull of the path of $X$ is differentiable over the set $\{0,T\}$. Its smoothness over the open interval $(0,T)$, however, is a much more delicate matter. We now state some results to elucidate this structure. 
%In hopes of simplifying the complicated picture, we provide an important invariance %property and results supporting two heuristics: too much jump activity or asymmetry %break smoothness. 
%\begin{prop}
%\label{prop:inf_var_criteria}
%Suppose $X$ has infinite variation, then we a.s. have $\sup \mS =-\inf \mS=\infty$ and hence $-\lim_{t\downarrow 0}C'(t)=\lim_{t\ua T}C'(t)=\infty$.
%\end{prop}

\begin{figure}[ht]
\centering
\includegraphics[width=0.48\textwidth]{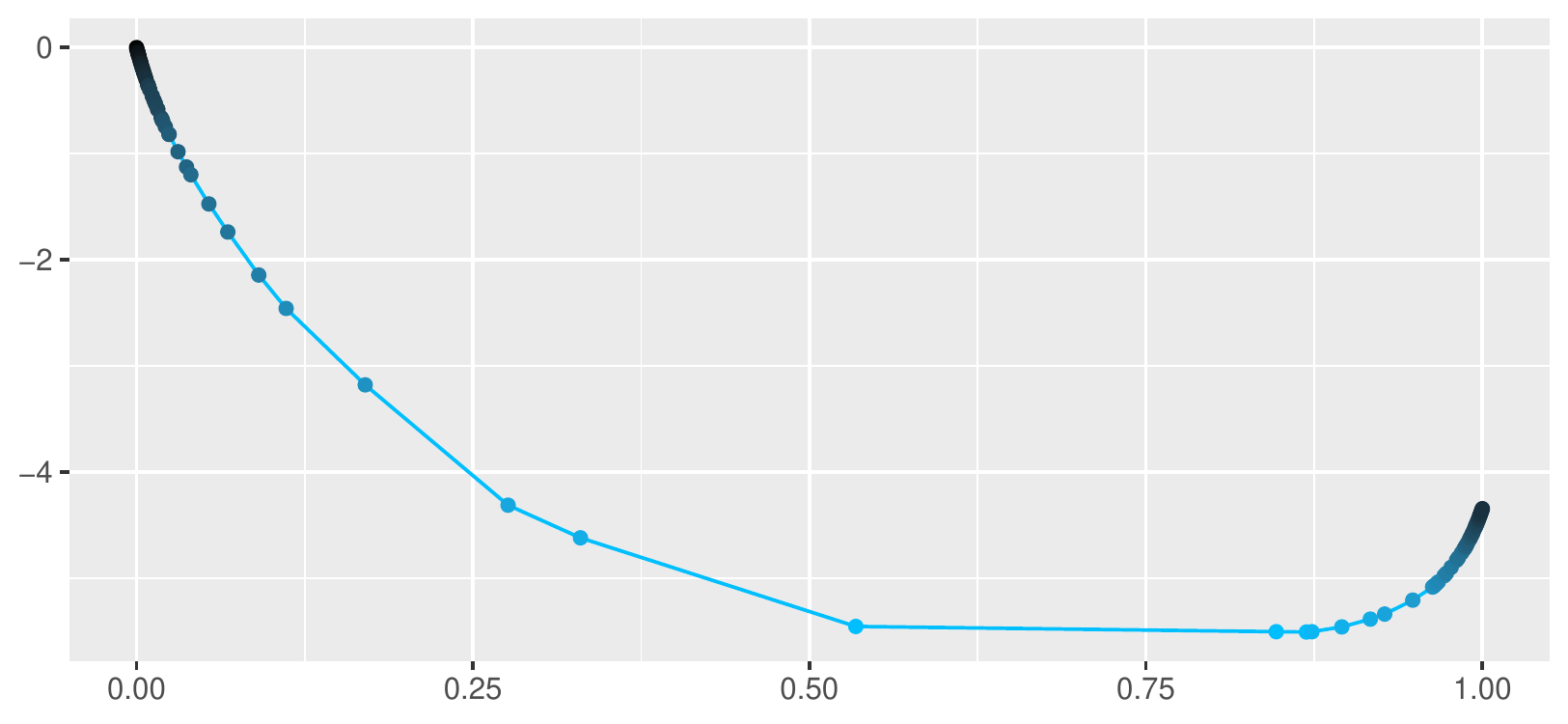}
\includegraphics[width=0.48\textwidth]{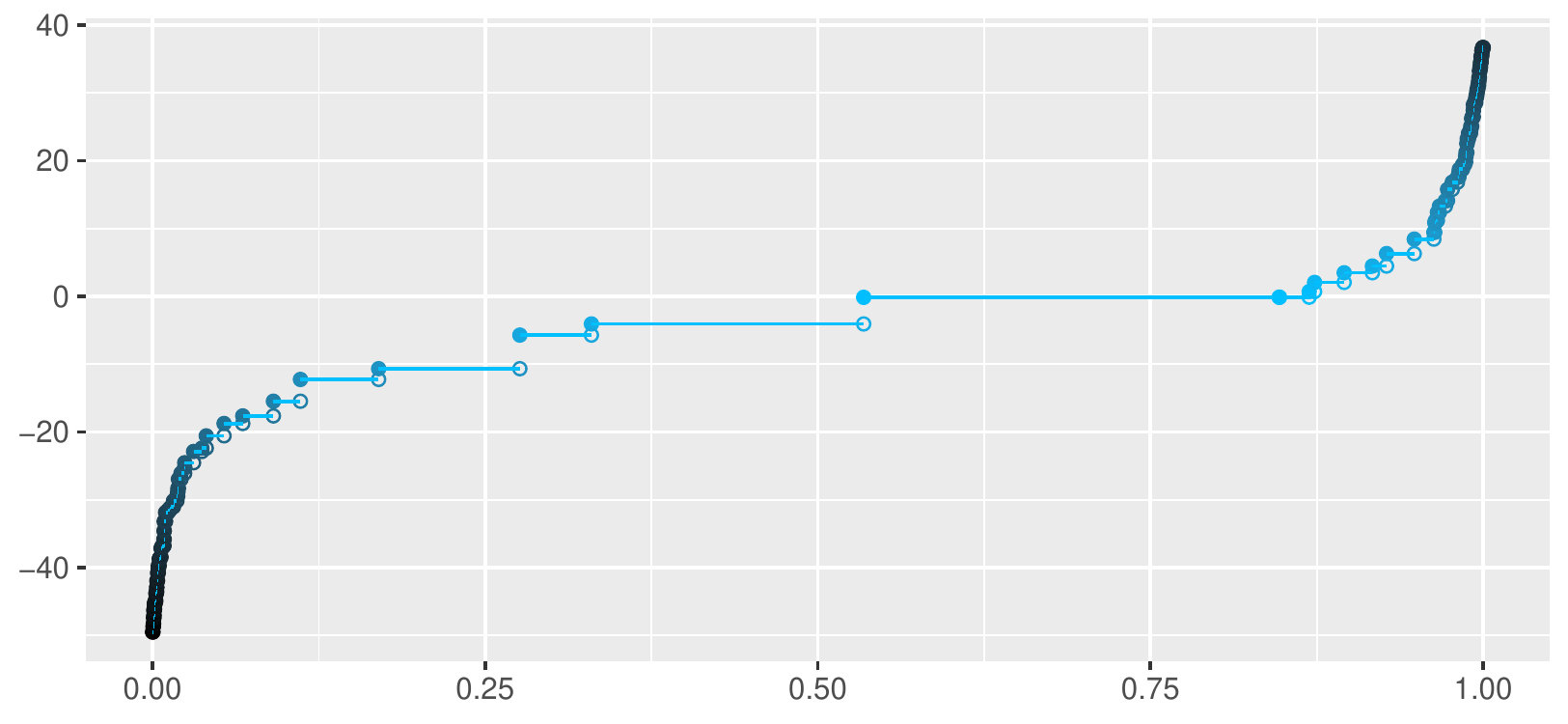}
\includegraphics[width=0.48\textwidth]{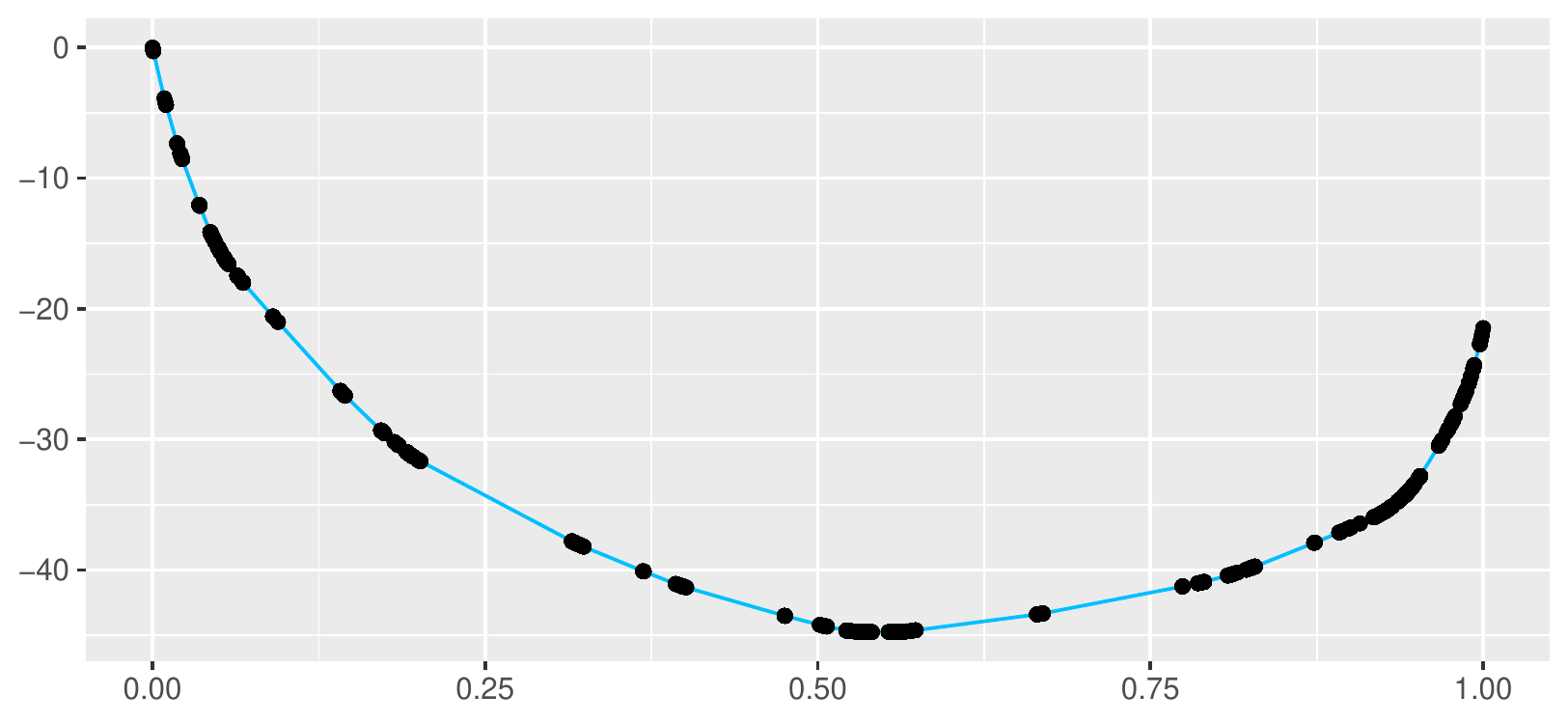}
\includegraphics[width=0.48\textwidth]{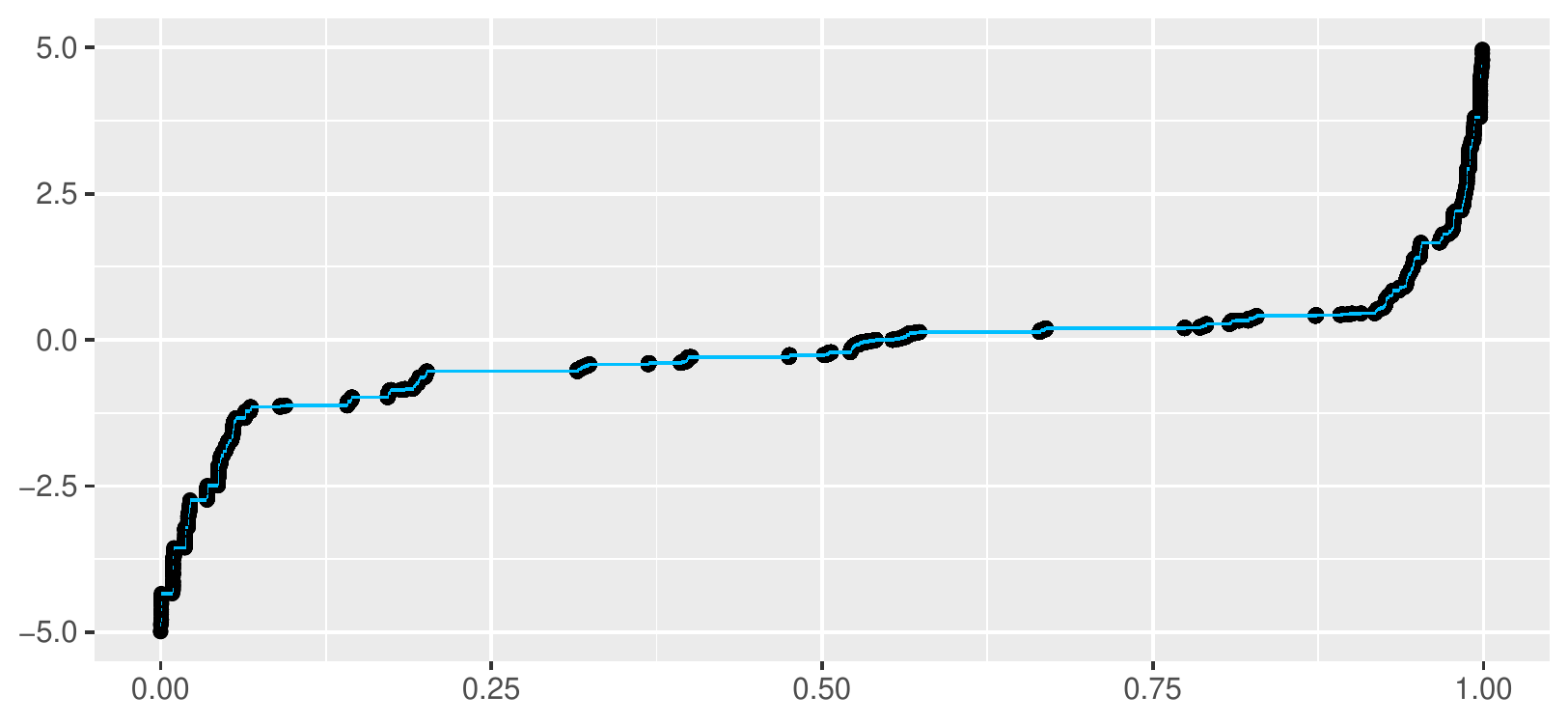}
\caption{\small The left pictures show the graph of a piecewise linear convex function $C$ with circle marks on $\bigcup_{r\in\mS}\partial I_r$. The right pictures show the graph of the derivative $C'$ with visible points of increase marked with a circle. In the top two pictures, there are no points of continuity as $C'$ tends to $\pm\infty$ near times $0$ and $1$. In the bottom two pictures, the function $C$ is continuously differentiable with a singular continuous derivative $C'$.}
\label{fig:all_continuous}
\end{figure}

%Proposition~\ref{prop:inf_var_criteria} 
\vspace{1mm}
\noindent \underline{Strongly eroded L\'evy process, perturbed by a finite variation process, is still strongly eroded.} In fact, for any L\'evy process, such a perturbation shifts the set $\mL(\mS)$ by the natural drift, defined in~\eqref{eq:natural_drift}, of the finite variation process  (given a set $A\subset\R$, define $A+b:=\{a+b\,:\,a\in A\}$ for any $b\in\R$ with convention $\emptyset+b:=\emptyset$). 

\begin{prop}
\label{prop:fin_var_perturbation}
Suppose the L\'evy process $X$ is of the form $X=Y+Z$ for (possibly dependent) L\'{e}vy processes 
$Y$ and $Z$. Let $\mS_X$ and $\mS_Z$ be the sets of slopes of the faces of the convex minorants of $X$ and $Z$, respectively. If $Y$ is of finite variation (possibly finite activity) with natural drift $b$ defined in~\eqref{eq:natural_drift}, then $\mL(\mS_X)=\mL(\mS_Z)+b$.
\end{prop}

The proof of Proposition~\ref{prop:fin_var_perturbation} relies on the a.s. limit $b=\lim_{t\downarrow0}Y_t/t$ and the stick-breaking representation of the convex minorant in~\cite{CM_Fluctuation_Levy}. The main idea is that if $Z_t/t$ frequently visits any neighborhood of a point $x\in\R$ as $t\da 0$, then $X_t/t$ visits the neighborhoods of $x+b$ just as frequently. Crucially, when the visits of $Z_t/t$ to the neighborhoods of $x$ occur, $Y_t/t$ must necessarily be close to $b$ (since the limit $Y_{t_n}/t_n\to b$ holds along any random sequence of times $t_n\da 0$). 

Proposition~\ref{prop:fin_var_perturbation} and its proof may suggest that, if $Y_t/t$ and $Z_t/t$ were to visit all open intervals as $t\da 0$ with such frequency that their respective sets of slopes $\mS_Y$ and $\mS_Z$ are dense, then the same should be true for $X_t/t$. This intuition, however, turns out to be false, as Example~\ref{ex:addition2} below illustrates. Intuitively, the reason for this is that the frequent visits of $Y_t/t$ and $Z_t/t$ to such neighborhoods may be sufficiently rare so that they do not occur simultaneously with sufficient frequency even when $Y$ and $Z$ are independent.

\vspace{1mm}
\noindent \underline{Too much jump activity breaks smoothness.}
Let $\sigma\ge 0$ be the volatility of the Brownian component of $X$ and define the function $\ov\sigma^2(u)\coloneqq \int_{(-u,u)}x^2\nu(\D x)$ for $u>0$. The lower-activity index $\beta_-$ (inspired by Pruitt~\cite{MR632968}) %, see also also~\cite[p.~595]{MR1664705}) 
and the \emph{Blumenthal--Getoor index} $\beta_+$ (introduced in~\cite{MR0123362}) are given by
\begin{equation}\label{eq:Pruitt_index}
\beta_-
%&=\inf\left\{p>0:\liminf_{|u|\to\infty}|u|^{-p}
%	(-\Re\psi(u)-\sigma^{2}u^{2}/2)=0\right\}\\
\coloneqq %\inf\Big\{p>0:\liminf_{u\downarrow 0}u^{p}(\ov{\nu}(u)+u^{-2}\ov{\sigma}^2(u))=0\Big\},
%\inf\Big\{p>0:\liminf_{u\downarrow 0}u^{p}\ov{\nu}(u)=0\Big\},
\inf\Big\{p>0:\liminf_{u\downarrow 0}u^{p-2}\ov{\sigma}^2(u)=0\Big\} \quad\text{and}\quad
\beta_+\coloneqq\inf\left\{p>0\,:\,\int_{(-1,1)}|x|^p\nu(\D x)<\infty\right\},
%&=\inf\left\{p>0:\liminf_{|u|\to\infty}|u|^{-p}
%(-\Re\psi(u)-\sigma^{2}u^{2}/2)=0\right\} \\
%&=\inf\left\{p>0:\liminf_{|u|\to\infty}|u|^{-p}
%|\psi(u)+\sigma^{2}u^{2}/2|=0\right\}\in[0,2],
\end{equation}
respectively. It is easy to see that the inequalities
$0\leq\beta_-\leq\beta_+\leq2$
hold. Since $\beta_-$ (resp. $\beta_+$) presents a lower (resp. upper) bound on the activity of the L\'evy measure $\nu$ at zero, in general we may have $\beta_-<\beta_+$. However, both indices agree if the tails of $\nu$ are regularly varying at $0$ (e.g. if $X_t$ is in the domain of attraction of an $\alpha$-stable law as $t\da 0$).

\begin{prop}\label{prop:index_criteria}
If $\int_1^\infty (1+u^2(\sigma^2+\ov\sigma^2(1/u)))^{-1}\D u<\infty$, then $\mL(\mS)=\emptyset$ a.s. and hence $X$ is abrupt. In particular, this is the case if $\sigma^2>0$ or $\beta_->1$.
\end{prop}

Proposition~\ref{prop:index_criteria} shows that a strongly eroded process necessarily satisfies $\beta_-\le 1\le \beta_+$. This is natural since, in some sense, the running supremum of the process $X$ fluctuates between the functions $t\mapsto t^{1/\beta_+}$ and $t\mapsto t^{1/\beta_-}$ as $t\da0$~\cite[Prop.~47.24]{MR3185174} and the visits of $X_t/t$ to compact intervals determine whether $X$ is strongly eroded. In other words, it is natural that strongly eroded processes require the linear map $t\mapsto t$ to lie between the functions $t\mapsto t^{1/\beta_+}$ and $t\mapsto t^{1/\beta_-}$ as $t\da 0$, which is equivalent to $\beta_-\le 1\le\beta_+$. We remark that, despite the necessary condition on the indices $\beta_-$ and $\beta_+$ allowing a strict inequality, all our examples of strongly eroded processes lie in the boundary case $\beta_-=1=\beta_+$. However, as explained in Example~\ref{ex:orey} below, Conjecture~\ref{conj:vigon} implies that a strict inequality is feasible for certain strongly eroded L\'evy processes. 

\vspace{1mm}
\noindent \underline{Too much asymmetry breaks smoothness.}
Recall that a process creeps \textit{upwards} (resp. \textit{downwards}) if $\p(T_{(x,\infty)}<\infty, X_{T_{(x,\infty)}}=x)>0$ (resp. $\p(T_{(-\infty,x)}<\infty, X_{T_{(-\infty,x)}}=x)>0$) for some $x>0$ (resp. $x<0$), where $T_A\coloneqq\inf\{t>0:X_t\in A\}$ denotes the first hitting time of set $A\subset\R$ (with convention $\inf\emptyset =\infty$), i.e., if the process crosses levels continuously with positive probability. Processes that creep (upward or downward), all of which are abrupt by~\cite[Ex.~1.5]{MR1947963}, tend to have L\'evy measures that are asymmetric on a neighborhood of $0$ (see~\cite{MR1875147} for a characterisation of such processes in terms of $\nu$). For instance, if $\sigma=0$ and $\nu$ is of infinite variation but $\int_{(0,1)}x\nu(\D x)<\infty$ (resp. $\int_{(-1,0)}|x|\nu(\D x)<\infty$), then $X$ creeps upwards (resp. downwards), see~\cite[Eq.~(1.7)]{MR321198} and the subsequent discussion. 

These facts suggest that asymmetry tends to produce abrupt processes. This heuristic is also suggested by the inequality $\Re(1/(1+iur-\psi(u)))\le 1/(1-\Re\psi(u))$, where we note that the characteristic exponent of the symmetrisation $\wh X$ of $X$ (a process with the same law as $X-Y$, where $Y$ is independent of $X$ but shares its law) equals $2\Re\psi$. In particular, under Vigon's point-hitting conjecture (see Conjecture~\ref{conj:vigon} below), \eqref{eq:vigon_identity} gives the following implications: (I) if $\wh X$ is abrupt then $X$ is abrupt and (II) if $X$ is strongly eroded then $\wh X$ is strongly eroded. We complement these observations with the following result, further supporting this heuristic. 

\begin{prop}
\label{prop:asymmetry}
Let $X$ be an infinite variation L\'evy process and suppose there exist constants $c>1$ and $x_0\in(0,\infty]$ such that $\nu([x,y))\ge c\nu((-y,-x])$ for all $0<x<y<x_0$. Then $X$ is abrupt. 
\end{prop}

We stress that the process in Proposition~\ref{prop:asymmetry} is abrupt but need \emph{not} creep. The assumption in Proposition~\ref{prop:asymmetry} requires, on a neighborhood of $0$, the restriction of $\nu$ on the negative half-line to be absolutely continuous with respect to its restriction on the positive half-line with Radon--Nikodym derivative $\varphi(x)=\nu(\D (-x))/\nu(\D x)$ bounded above by $1/c<1$ (equivalently, $\limsup_{x\da 0}\varphi(x)<1$). This condition is nearly optimal, since there exist strongly eroded processes with positive asymmetry $1-\varphi(x)$ that vanishes (arbitrarily) slowly $1-\varphi(x)\downarrow0$ as $x\da 0$, see Example~\ref{ex:low_asymmetry} below. Note that, by using Propositions~\ref{prop:asymmetry} and~\ref{prop:fin_var_perturbation} jointly, we obtain a simple recipe to construct abrupt processes with $\beta_-\le1\le\beta_+$. 

\vspace{1mm}
\noindent \underline{Sufficient conditions for $X$ to be strongly eroded (or abrupt).}
The following theorem, implied by Theorem~\ref{thm:CM_local_smoothness} above and the results in~\cite{MR1947963}, shows that most L\'evy processes of infinite variation are either strongly eroded or abrupt. Moreover, Theorem~\ref{thm:zeta_criterion} offers simple conditions to ascertain whether $X$ is strongly eroded or abrupt. The criteria are (mostly) in terms of the behaviour at infinity of the L\'evy--Khintchine exponent $\psi$ of $X$. More precisely, it is connected to the ratio $\psi(u)/u$ for large $|u|$. This ratio appears naturally since the characteristic exponent of $X_t/t$ is given by $u\mapsto t\psi(u/t)$, whose behaviour for small $t$ is described by the behaviour of $\psi(u)/u$ for large $|u|$.

Let $\Im z$ and $|z|$ denote the imaginary part and the modulus of the complex number $z\in\Complex$. Recall that $\Re\psi$ (resp. $\Im\psi$) is an even (resp. odd) function on $\R$. 

\begin{thm}
\label{thm:zeta_criterion}
Let $X$ be a L\'evy process of infinite variation with $e^{\psi(u)}=\e e^{iuX_1}$, $u\in\R$.
\begin{itemize}[leftmargin=2em, nosep]
\item[\nf(i)] If
$\limsup_{u\to\infty}|\psi(u)/u|<\infty$, then $X$ is strongly eroded.
\item[\nf(ii)] If $\lim_{u\to\infty}|\psi(u)/u|=\infty$, then $X$ is either abrupt or strongly eroded.
\begin{itemize}[leftmargin=1em, nosep]
\item[\nf{(ii-a)}] If $\lim_{|u|\to\infty}|\Re\psi(u)/u|=\infty$, then $X$ is strongly eroded if and only if $\int_1^\infty \Re(1/(1-\psi(u)))\D u=\infty$, 
\item[\nf(ii-b)] If the upper and lower limits of $|\Re\psi(u)/u|$ lie in $(0,\infty)$ and $\lim_{|u|\to\infty}|\Im\psi(u)/u|=\infty$, then $X$ is strongly eroded if and only if $\int_1^\infty u(1+|\Im\psi(u)|^2)^{-1}\D u=\infty$,
\item[\nf(ii-c)] If $\lim_{|u|\to\infty}|\Re\psi(u)/u|=0$ and $\lim_{|u|\to\infty}|\Im\psi(u)/u|=\infty$, then $X$ is strongly eroded if and only if we have $\int_1^\infty (1-\Re \psi(u))(1+|\Im\psi(u)|^2)^{-1}\D u=\infty$.
\end{itemize}
\end{itemize}
\end{thm}

In fact, the proof of Theorem~\ref{thm:zeta_criterion} shows that, under the assumptions of Theorem~\ref{thm:zeta_criterion}, $\s_1$ is either locally bounded (making $X$ strongly eroded) or everywhere infinite (making $X$ abrupt). Cases (i) and (ii) in Theorem~\ref{thm:zeta_criterion} are in some sense generic, but they do not exhaust the class of infinite variation L\'evy processes, see Examples~\ref{ex:fluctuation} and~\ref{ex:orey} below. 
In fact, Example~\ref{ex:fluctuation} defines a strongly eroded L\'evy process, outside of the scope of Theorem~\ref{thm:zeta_criterion}, with the characteristic exponent that fluctuates between linear and superlinear behaviour as $u\to\infty$. However, the class of processes in the union of case (i) and (ii) is closed under addition of independent summands. 
%Vigon first noted in~\cite[Thm~1.7]{VigonConjecture} that in cases (i) and (ii) of Theorem~\ref{thm:zeta_criterion}, $\s_1$ is either locally integrable (and hence $X$ is abrupt) or $\s_1(r)=\infty$ for all $r\in\R$ (and hence $C'$ is continuous). 
Similarly, cases (ii-a), (ii-b) and (ii-c) are not exhaustive within (ii). However, for neither case (i) nor case (ii) to hold, it is necessary that $|\psi(u)|$ fluctuate between linear (or sublinear) and superlinear functions of $|u|$. The fact that  $X$ has infinite variation if and only if $\int_1^\infty u^{-2}|\Re\psi(u)|\D u=\infty$ (by Lemma~\ref{lem:VigonThesis}, see also~\cite[Prop.~1.5.3]{vigon:tel-00567466}) shows that most processes indeed fall within either case (i) or (ii).
%The integral criteria in (ii-a), (ii-b) and (ii-c) simplify the criterion in Theorem~\ref{thm:CM_global_smoothness} to a single integral condition with no varying parameters. 

%We note that, if $\beta_-<1<\beta_+$, then neither case (i) nor case (ii) in Theorem~\ref{thm:zeta_criterion} hold (by Lemma~\ref{lem:LevyBounds} below). For instance, this is the case for Orey's process (see Example~\ref{ex:orey} below). Thus, if Conjecture~\ref{conj:vigon} were false, we would expect it to fail for a L\'evy process for which these strict inequalities hold. Under the assumption that the tails of the L\'evy measure $\nu$ are regularly varying at $0$, we are often able to determine whether $X$ is abrupt or strongly eroded in terms of simpler criteria (see Section~\ref{sec:examples&specialcases} below). Such regular variation implies $\beta_-=1=\beta_+$ but still allows for an ample variety of behaviours.

\vspace{1mm}
\noindent \underline{A conjectural dichotomy.}
Our results may be viewed as further evidence for Vigon's point-hitting conjecture (see Conjecture~\ref{conj:vigon} below), whose origins go back to Vigon's PhD thesis~\cite[p.~10]{vigon:tel-00567466} in 2002, which implies the following dichotomy for infinite variation L\'evy processes: 

\begin{conj}
\label{conj:dichotomy}
Any infinite variation L\'evy process is either abrupt or strongly eroded.
\end{conj}

In order to understand the relationship between Conjecture~\ref{conj:dichotomy} and Vigon's point-hitting conjecture, recall first that the process $X$ hits points if for some $x\in\R$, the hitting time $T_x\coloneqq\inf\{t>0:X_t=x\}$ is finite with positive probability. If $X$ has infinite variation,~\cite[Thm~8]{MR0368175} (see also~\cite[Thm~1]{MR0272059}) implies that $\p(T_x<\infty)>0$ for some $x\in\R$ if and only if $\p(T_x<\infty)>0$ for all $x\in\R$. 
%Its name is due to the fact that $\s_1(r)<\infty$ if and only if $(X_t-rt)_{t\ge 0}$ visits points (see~\cite[Thm~43.3]{MR3185174} and Remark~\ref{rem:Vigon} below).

\begin{conj}[{\cite[Conject.~1.6]{VigonConjecture}}]
\label{conj:vigon}
Let $X$ be an infinite variation process and for any $r\in\R$ define the L\'evy process $X^{(r)}=(X_t-rt)_{t\ge 0}$. Then the following statements are equivalent.
\begin{itemize}[leftmargin=2em, nosep]
\item[\nf{(i)}] There exists some $r\in\R$ such that the process $X^{(r)}$ hits points.
\item[\nf{(ii)}] For all $r\in\R$ the process $X^{(r)}$ hits points.
\item[\nf{(iii)}] The process $X$ is abrupt.
\end{itemize}
\end{conj}

%Denote by $e^q_r(A)$ the $q$-energy of the set $A$~\cite[Def.~42.24]{MR3185174} for the process $X^{(r)}$, then $e^1_r(\{0\})=\s_1(r)$. Hence $\s_1(r)<\infty$ if and only if the $1$-capacity $c^1_r$ of a one-point set is positive~\cite[Def.~43.1]{MR3185174}. 
By~\cite[Thm~7]{MR0368175} (see also~\cite[Thm~2]{MR0272059}), $\s_1(r)<\infty$ is equivalent to $X^{(r)}$ hitting points (recall the definition of $\s_1$ in~\eqref{eq:s_q}). Moreover, by the equivalence in~\eqref{eq:vigon_identity} and~\cite[Thm~1.3]{MR1947963}, $X$ is abrupt if and only if $\s_1$ is locally integrable on $\R$. 
%In short, $\s_1(r)<\infty$ if and only if $X^{(r)}$ hits points, cf.~\cite[Thm~2]{MR0272059}. 
Conjecture~\ref{conj:vigon} thus says that the following three statements are equivalent for any infinite variation L\'evy process $X$: (i)~$\s_1(r)<\infty$ for some $r\in\R$, (ii) $\s_1(r)<\infty$ for all $r\in\R$ and (iii) the function $\s_1$ is locally integrable on $\R$. In particular, under Conjecture~\ref{conj:vigon}, the function $\s_1$ is either everywhere infinite or locally integrable, thus implying Conjecture~\ref{conj:dichotomy} by Theorem~\ref{thm:CM_local_smoothness}, equivalence~\eqref{eq:vigon_identity} and~\cite[Thm~1.3]{MR1947963}. 

The finiteness of $\s_1(r)$, hinging entirely on the integrability at infinity of the positive bounded function $u\mapsto\Re(1/(1+iur-\psi(u)))$, becomes a focal point under Conjecture~\ref{conj:vigon}. For instance, Conjecture~\ref{conj:vigon} holds if $X$ satisfies the assumptions of either Proposition~\ref{prop:asymmetry} or Theorem~\ref{thm:zeta_criterion} (the below proofs of these results establish that $\s_1$ is locally finite on $\R$). The condition $\s_1(0)<\infty$ is equivalent to a number of probabilistic statements about the infinite variation process $X$: 
\begin{itemize}[leftmargin=2em, nosep]
\item the potential measure of $X$ is absolutely continuous with a bounded density (by~\cite[Thm~43.3]{MR3185174}),
\item the point $0$ is regular for itself for the process $X$, i.e. $\p(T_0=0)=1$ (by~\cite[Thms~7 \& 8]{MR0368175}), 
\item the process $X$ possesses a local time field (by~\cite[Thm~V.1]{MR1406564}). 
\end{itemize}
In principle, any of these properties may hold for $X$ but not for some $X^{(r)}$. Conjecture~\ref{conj:vigon}, which asserts that this is not the case,
can thus be equivalently stated by substituting ``hitting of points'' with any of the three properties in the bullet-point list. 

The structure of 
Conjecture~\ref{conj:vigon}, in terms of varying drifts, is natural as
a number of properties of infinite variation processes are known to be invariant under addition of a deterministic drift and, more generally, under a perturbation by an independent finite variation process $Y$. For instance, as Vigon shows in~\cite[Thm~Kaa]{MR1875147}, if the infinite variation process $X$ creeps in either direction, then $X+Y$ also creeps in the same direction. 
%Moreover, by~Proposition~\ref{prop:fin_var_perturbation}, if $X$ is eroded %(resp. abrupt), then $X+Y$ is eroded (resp. abrupt). 
Despite Vigon's extensive body of work in the area over the years (see~\cite{vigon:tel-00567466,MR1875147,MR1947963,VigonConjecture}),
to the best of our knowledge Conjecture~\ref{conj:vigon} remains unsolved.  In Conjecture~\ref{conj:dichotomy} we offer a weaker conjectural dichotomy
and prove that, if it holds for $X$ then it holds for $X+Y$, 
%under a perturbation of $X$ %by a deterministic drift and, more generally,
%by 
where $Y$ is any finite variation process independent of $X$, see Proposition~\ref{prop:fin_var_perturbation}.
As Conjecture~\ref{conj:dichotomy} remains unsolved in spite of our efforts, in conclusion we only remark that it implies the existence of a strongly eroded L\'evy process with high activity of small jumps, i.e. Blumenthal--Getoor index arbitrarily close to two (see Example~\ref{ex:orey} below).

\subsubsection{Infinite time horizon}

Given any L\'evy process $X$ (possibly compound Poisson with drift), define the quantity $l \coloneqq \liminf_{t \to \infty }X_t/t\in [-\infty,\infty]$. The convex minorant $C_\infty$ of $X$ on the time interval $[0,\infty)$ is a.s. finite if and only if  $l\in(-\infty,\infty]$~\cite[Cor.~3]{MR2978134} (recall from Kolmogorov's zero-one law shows that the limit $l$ is a.s. constant); otherwise, $C_\infty$ equals $-\infty$ on $(0,\infty)$. 
%Indeed, this is clear since $l=-\infty$ is equivalent to no linear function $t\mapsto c+mt$ being a lower bound for the path of $X$ on the time interval $[0,\infty)$ and hence, to $C=-\infty$ identically on $(0,\infty)$. 
Denote the positive (resp. negative) part of $x$ by $x^+\coloneqq \max\{x,0\}$ (resp. $x^-\coloneqq \max\{-x,0\}$). When $l\in \R$, $C_\infty$ is also piecewise linear and the slopes of the faces of $C_\infty$ lie on the interval $(-\infty,l)$. Whenever the expectation $\e X_1$ is well defined, i.e., if $\min\{\e X_1^+,\e X_1^-\}<\infty$, the strong law of large numbers~\cite[Thms~36.4 \& 36.5]{MR3185174} implies that $l=\e X_1=\lim_{t\to\infty}X_t/t$ a.s. Otherwise, we have $\e X_1^+=\e X_1^-=\infty$ and~\cite[Thm~15]{MR2320889} (see also~\cite{MR336806}) shows that $l\in \{-\infty,\infty\}$ and 
\begin{equation}
l=\infty \qquad \text{if and only if} \qquad 
\int_{(-\infty,-1)}
    \frac{|x|}{\int_{0}^{|x|}\nu([\max\{1,y\},\infty))\D y}\nu(\D x)<\infty,
\end{equation}
where we denote $x\vee y\coloneqq \max\{x,y\}$. Hence, the convex minorant and concave majorant of $X$ are both a.s. finite if and only if $\e |X_1|<\infty$, and in that case $l=\e X_1=\lim_{t\to\infty}X_t/t$ a.s. 

\begin{prop}
\label{prop:C'_infinity}
Suppose $l\in(-\infty,\infty]$, then we have $\lim_{t\to\infty}C'_\infty(t)=l$ a.s.
\end{prop}

Proposition~\ref{prop:C'_infinity} implies that the set of slopes $\mS_\infty$ of the convex minorant $C_\infty$ satisfies $l\in\mL^-(\mS_\infty)$ and $\mS_\infty\subseteq(-\infty,l)$ a.s. whenever $l\in\R$. This means that $C_\infty$ becomes nearly parallel to the line $t\mapsto lt$ as $t\to\infty$; however, this does not entail any additional continuity for $C'_\infty$ (other than during its intervals of constancy) as it does not occur at a finite time. In particular, Proposition~\ref{prop:C'_infinity} shows that, if $l\in(-\infty,\infty]$, then $\mS_\infty$ is an infinite set even for compound Poisson processes.

\begin{prop}
\label{prop:S'_infinity}
Suppose $l\in(-\infty,\infty]$. Let $\mS$ be the set of slopes of the convex minorant $C$ of $X$ on the time interval $[0,1]$. Then we a.s. have the following equalities
\begin{equation}
\label{eq:S_and_S_infty}
\mL^-(\mS)\cap(-\infty,l)
=\mL^-(\mS_\infty)\cap(-\infty,l)
\quad\text{and}\quad
\mL^+(\mS)\cap(-\infty,l)
=\mL^+(\mS_\infty).
\end{equation}
\end{prop}

As a consequence of Propositions~\ref{prop:C'_infinity} \&~\ref{prop:S'_infinity}, the limit sets $\mL(\mS_\infty)$ and $\mL^\pm(\mS_\infty)$ are all constant a.s. The results in Subsection~\ref{subsec:fin_var_fin_time} \&~\ref{subsec:infin_var_fin_time} together with Propositions~\ref{prop:C'_infinity} \&~\ref{prop:S'_infinity} yield the following.

\begin{cor}
Suppose $X$ has finite variation and $l\in(-\infty,\infty]$. Let $I_{r,\infty}$ be the maximal open interval of constancy of $C'_\infty$ corresponding to slope $r$. Then $C'_\infty$ is discontinuous on $\bigcup_{r\in\mS_\infty}\partial I_{r,\infty}$, lower bounded and $\lim_{t\to\infty}C'_\infty(t)=l$ a.s. Moreover, the following statements hold:\\
If $X$ has finite activity, then 
\begin{itemize}[leftmargin=2em, nosep]
    \item[\nf{(i)}] $C_\infty$ has infinitely many faces with $\mL(\mS_\infty)=\emptyset$ when $l=\infty$ and otherwise $\mL^-(\mS_\infty)=\{l\}$ and $\mL^+(\mS_\infty)=\emptyset$. 
\end{itemize}
If $X$ has infinite activity, then:
\begin{itemize}[leftmargin=2em, nosep]
    \item[\nf{(ii)}] If $l\in (\gamma_0,\infty]$, then the process $t\mapsto X_t-\gamma_0t$ attains its infimum on $[0,\infty)$ at a unique time $v$ and $\mL(\mS_\infty)=\{\gamma_0,l\}$ if $l<\infty$ and otherwise $\mL(\mS_\infty)=\{\gamma_0\}$. Moreover,
    \begin{itemize}[leftmargin=2em, nosep]
    \item[{\nf(ii-a)}] if $I_+=I_-=\infty$, then $v\in(0,\infty)$ and $C'_\infty(v-)=C'_\infty(v)=\gamma_0\in\mL^-(\mS_\infty)\cap\mL^+(\mS_\infty)$ a.s.,
    \item[{\nf(ii-b)}] if $I_-<\infty$, then $v\in[0,\infty)$, $\p(v=0)\in(0,1]$, $C'_\infty(v)=\gamma_0\notin\mL^-(\mS_\infty)$ and, on the event $\{v\ne 0\}$, we have $C_\infty'(v-)<\gamma_0$ a.s.,
    \item[{\nf(ii-c)}] if $I_+<\infty$, then $v\in(0,\infty)$, $C'(v-)=\gamma_0\notin\mL^+(\mS_\infty)$ and we have $C'(v)>\gamma_0$ a.s.
    \end{itemize}
\item[\nf{(iii)}] If $l\in(-\infty,\gamma_0]$, then $\mL^-(\mS_\infty)=\{l\}$ and $\mL^+(\mS_\infty)=\emptyset$.
\end{itemize}
\end{cor}

\begin{cor}
Assume that $X$ has infinite variation and $l\in(-\infty,\infty]$, then $\inf \mS_\infty=-\infty$ a.s. Moreover, $C_\infty'$ is continuous if and only if $\int_0^1 t^{-1} \p(X_t/t\in (a,b))\D t=\infty$ for all $a<b<l$.
\end{cor}

Again, under Conjecture~\ref{conj:dichotomy}, the Lebesgue--Stieltjes measure $\D C'_\infty$ is either purely atomic or purely singular continuous. 

\subsection{Related literature}

The smoothness of the convex hull of planar Brownian motion goes back to Paul L\'evy~\cite{MR0029120}. In~\cite{MR972777}, the authors establish lower bounds on the modulus of continuity of the derivative of the boundary of the convex hull (see~\cite{MR972777} and the references therein). %The authors show, given an increasing function $\eta$ with $\eta(0)=0$ such that $x\eta(x)$ is convex, that the derivative of the convex hull of planar Brownian motion a.s. does not admit $\eta$ as a modulus of continuity if $\int_0^1 x^{-1}\eta(x)\D x<\infty$ (see details in~\cite[Thm~1]{MR972777}). 
These results all concern the spatial convex hull of Brownian motion while we consider the time-space convex hull of a real-valued L\'evy process $X$, i.e. the convex hull of $t\mapsto (t,X_t)$. However, in our context it is also natural to enquire about the modulus of continuity of the convex minorant of $X$, a topic that will be addressed in future work.
%The paper~\cite{MR205343} proves that the curvature of the convex hull of a $d$-dimensional Brownian motion is concentrated on a metrically small set (of $h$--Hausdorff measure $0$ for a suitable class of functions $h$). 

In~\cite{MR1747095}, Bertoin describes the law of the convex minorant of Cauchy process in terms of a gamma process, establishing the continuity of its derivative. The result relies on an explicit description of the right-continuous inverse of the slope process of Cauchy process (see~\cite{MR714964,CM_Fluctuation_Levy} and~\cite[Ch.~XI]{MR1739699} for similar characterisations for other L\'evy processes).  %Similarly,~\cite{MR714964} characterises the law of the multiplicative inverse of the derivative of the all-time convex minorant of Brownian motion as the inverse of a pure-jump process with independent (non-stationary) increments. A similar study was carried out in~\cite[Ch.~XI]{MR1739699} for all L\'evy processes. 
Our approach is instead based on the stick-breaking representation of the convex minorant of L\'evy process first established in~\cite{MR2978134} (see also~\cite{MR2831081,MR2948693,CM_Fluctuation_Levy}). This is an important stepping stone for our results in Section~\ref{sec:0-1lawsbres} below.

The abruptness of a L\'evy process $X$ is closely connected via~\eqref{eq:integral_slope} to the properties of the contact set between $X$ and its $\alpha$-Lipschitz minorant (the largest Lipschitz function with derivative equal to $\pm \alpha$ a.e.) or between $X$ and its convex minorant. This connection also has a geometric interpretation. By~\cite[Thm~3.8]{MR3176366}, the subordinator associated to the contact set between the process $X$ and its $\alpha$-Lipschitz minorant has infinite activity if and only if~\eqref{eq:integral_slope} holds for the interval $I=[-\alpha,\alpha]$. By Theorem~\ref{thm:CM_local_smoothness} this subordinator has infinite activity if and only if $C$ has infinitely many faces whose slope lies on $[-\alpha,\alpha]$. When this occurs, the L\'evy process remains close to the $\alpha$-Lipschitz minorant after touching it~\cite[Rem.~4.4]{MR3176366}. We also observe this behaviour at every contact point between the L\'evy process and its convex minorant when the latter is continuously differentiable. In contrast, an abrupt L\'evy process must leave its convex minorant sharply after every contact point in the same way it leaves its running supremum~\cite{MR1875147} (see also~\cite{VigonConjecture}). This strengthens the contrasting behaviour between abrupt processes and strongly eroded processes, cf. the conjectural dichotomy in Subsection~\ref{subsec:infin_var_fin_time} above. 

\subsection{Organisation of the paper.}

The remainder of the paper is organised as follows. 
Section~\ref{sec:examples&specialcases} illustrates the breath of the class of strongly eroded L\'evy processes. In particular, it shows that even within the class of L\'evy processes with regularly varying L\'evy measure at zero, a wide variety of behaviours is possible.
%In Section~\ref{sec:s} we give equivalent conditions to those of %Theorem~\ref{thm:CM_local_smoothness} to characterise $\mL(\mS)$ in terms of the %generating triplet of $X$. The core result in this section is %Theorem~\ref{thm:Vigon}. Furthermore we describe the connection between $\mS$ and %the small-time behaviour of $X_t/t$. 
Section~\ref{sec:0-1lawsbres} introduces and establishes a zero-one law for the stick-breaking process (see Theorem~\ref{thm:0-1_law_formula} below), which implies Theorem~\ref{thm:CM_local_smoothness}. Theorem~\ref{thm:CM_global_smoothness} and all other results of Section~\ref{sec:introduction} are established in Section~\ref{sec:proofssection}. In Section~\ref{sec:comments} we describe informally, in terms of the path behaviour of the process $X$ as it leaves $0$, what appears to be the main stumbling block for establishing the dichotomy in Conjecture~\ref{conj:dichotomy} and, more strongly, Vigon's point-hitting conjecture (see Conjecture~\ref{conj:vigon} above). Appendix~\ref{sec:pw-cf} describes the analytical behaviour of an arbitrary piecewise linear convex function and its right-derivative. Appendix~\ref{sec:s_q} establishes a formula for the integral of $\s_1(r)$ in terms of the law of the process $X$, which implies the equivalence in~\eqref{eq:vigon_identity} above.
%Section~\ref{sec:proof_Vigon} is dedicated to stating and proving Theorem~\ref{thm:Vigon}, which implies~\eqref{eq:vigon_identity}. 

\section{Is an infinite variation L\'evy process strongly eroded? Examples and counterexamples}
\label{sec:examples&specialcases}

The class of strongly eroded L\'evy processes has a delicate structure, depending crucially on the fine behaviour of the L\'evy measure at zero. 
In this section we present evidence for the following principles for constructing strongly eroded L\'evy processes,  as well as study their limitations. Heuristically, the boundary of the convex hull of the path of an infinite variation L\'evy process $X$ becomes smoother as:
\begin{itemize}
    \item[(I)]  the jump activity decreases (cf. Example~\ref{ex:symm_low_activity});
    \item[(II)] the small jumps become more symmetric (cf. Examples~\ref{ex:symm_low_activity} and~\ref{ex:low_asymmetry});
    \item[(III)] at zero, the L\'evy measure ``approaches'' that of a Cauchy process (cf. Example~\ref{ex:Cauchy-normal-attraction}).
\end{itemize}
However, as we shall see from the examples below, the following features are also demonstrated: (I) a decrease of the straightforward measure of the jump activity, such as the Blumenthal--Getoor index, appears not to be sufficient for $X$ to become strongly eroded, cf. Example~\ref{ex:orey}; (II) there exist both asymmetric strongly eroded processes with one of the tails of the L\'evy measure at zero dominating (i.e.   $|\Im\psi(u)|/\Re\psi(u)\to\infty$ as $|u|\to\infty$, where $\psi$ is the characteristic exponent of $X$), cf. Example~\ref{ex:low_asymmetry}, and symmetric abrupt processes; (III) there exist abrupt processes attracted to Cauchy process, cf. Example~\ref{ex:Cauchy-non-normal-attraction}. We further show that the classes of abrupt processes (i.e., with $\mL(\mS)=\emptyset$) and strongly eroded processes (i.e., with $\mL(\mS)=\R$) are \emph{not} closed under addition of independent summands. Moreover, the sum of a strongly eroded and an independent abrupt process may be either abrupt or strongly eroded, cf. Examples~\ref{ex:addition2} and~\ref{ex:low_asymmetry}. In addition, a subordinated abrupt process of infinite variation may be either abrupt or strongly eroded, while a subordinated strongly eroded process may (but need not) be strongly eroded, cf. Example~\ref{ex:subordination}.

\subsection{Near-Cauchy processes}

We begin with processes attracted to Cauchy process. Already in this class, we will see how easily a minor change in the jump activity of a process turn a strongly eroded process into an abrupt one. In particular, it is clear that information that does not capture the asymmetry of $\psi$ (such as the indices $\beta_+$ and $\beta_-$ defined in~\eqref{eq:Pruitt_index}) will have limitations in determining whether $X$ is strongly eroded or abrupt. 

\begin{ex}[Domain of normal attraction to Cauchy process]
\label{ex:Cauchy-normal-attraction}
Suppose $X_t/t$ converges weakly as $t\downarrow0$ to a Cauchy random variable $S$ with density proportional to $x\mapsto1/(\lambda_1^2+(x-\lambda_2)^2)$ on $\R$ for some $\lambda_1>0$ and $\lambda_2\in\R$. Then $X$ is strongly eroded. Indeed, $\lim_{t\downarrow0}\p(X_t/t\in (a,b))= \p(S\in(a,b))>0$ for any $a<b$ by assumption, so Theorem~\ref{thm:CM_local_smoothness} immediately gives the claim. Such an assumption is satisfied if, for instance, $\nu([x,\infty))x\to c$, $\nu((-\infty,-x])x\to c$ and $\int_{(-1,-x]\cup[x,1)}y\nu(\D y)\to c'$ as $x\downarrow0$ for some $c>0$ and $c'\in\R$ by~\cite[Thm~2]{MR3784492}. 
\end{ex}

\begin{ex}[Domain of non-normal attraction to Cauchy process]
\label{ex:Cauchy-non-normal-attraction}
Assume the characteristic exponent is given by $\psi(u)=iu\gamma +\int_{\R}(e^{iux}-1-iux\1_{\{|x|<1\}})\nu(\D x)$ for $u\in\R$ where $\nu$ is symmetric with $\nu([x,\infty))=x^{-1}\rho(x)$ for a slowly varying function $\rho$ at $0$ with $\lim_{x \downarrow 0}\rho(x)=\infty$. 
%Choose $h(x)=x/\wp_+(x)$, then~\eqref{eq:rv_nu} yields $\nu([x,\infty))h(x)\to 1$ as $x \to 0$. If $h$ starts at $0$ and is continuous and increasing on $(0,\ve)$ for some $\ve>0$, there exists an inverse function $g$ s.t. $h(g(x))=g(h(x))=x$ for $x\in(0,\ve)$ and $g(x) \to 0$ as $x \to 0$. Hence $\ov{\nu}_\pm(g(x))x \to 1$ 
By~\cite[Thm~2(iii)]{MR3784492}, the limit $(\gamma-\int_{\{x\leq |y|<1\}}y\nu(\D y))/(x\nu([x,\infty)))=\gamma/\rho(x)\to 0$ as $x\da 0$ implies that $X_t/h(t)$ converges weakly as $t\downarrow 0$ to a Cauchy random variable $S$ for an appropriate function~$h$ (given in terms of the de Bruijn inverse of $\rho$) with non-constant ratio $h(t)/t$ that is slowly varying at $0$. However, $X$ may be abrupt or strongly eroded. In fact, $\Im\psi(u)=\gamma u$ and $\Re\psi(u)$ is bounded between multiples of $|u|\rho(1/|u|)$ as $|u|\to\infty$ by Lemma~\ref{lem:karamata_Levy} below. Hence~\eqref{eq:vigon_identity} and Theorem~\ref{thm:CM_local_smoothness} show that $X$ is abrupt (resp. strongly eroded) if $\int_0^1(x\rho(x))^{-1}\D x$ is finite (resp. infinite). Intuitively, as shown by the following examples, a sufficiently large $\rho$ may make $X$ sufficiently different from Cauchy process, resulting in an abrupt process $X$. Pick $\rho(x)\coloneqq\log(1/x)^2\1_{(0,1/2)}(x)$, then $X$ has infinite variation (since $\int_0^1 x^{-1}\rho(x)\D x=\infty$) and is abrupt since $\int_0^1(x\rho(x))^{-1}\D x=1/\log(2)<\infty$. If instead $\rho(x)\coloneqq\log(1/x)\1_{(0,1/2)}(x)$, then $X$ is strongly eroded as $\int_0^1(x\rho(x))^{-1}\D x=\infty$. 
\end{ex}

Next we consider $1$-semistable and weakly $1$-stable processes, both of which have relatively simple characteristic exponents.

\begin{ex}[Weakly stable processes]
Let $X$ be a (possibly weakly) $1$-stable process, i.e., with L\'evy measure $\nu(\D x)=|x|^{-2}(c_+\1_{(0,\infty)}(x)+c_-\1_{(-\infty,0)}(x))\D x$ for some $c_\pm\ge 0$ with $c_++c_->0$. If $c_+=c_-$, then $X$ is Cauchy (strictly $1$-stable), has infinite variation and is strongly eroded by Theorem~\ref{thm:zeta_criterion}(i). If $c_+\neq c_-$, then $X$ is weakly $1$-stable, has infinite variation, is not attracted to a Cauchy process as $t\downarrow0$ (see~\cite[Ex.~4.2.2]{MR3784492}) and is abrupt by Proposition~\ref{prop:asymmetry}. Since the symmetrisation of a weakly $1$-stable process is Cauchy, the class of abrupt L\'evy processes is not closed under addition even within the class of weakly stable processes.
\end{ex}

\begin{ex}[Semi-stable processes]
Let $X$ be a $1$-semi-stable process (see~\cite[Def.~13.2]{MR3185174} for definition). Then $X$ has infinite variation and there exists a positive constant $b>1$ such that the L\'evy measure $\nu$ is uniquely defined as a periodic extension of its restriction to $(-b,b)\setminus(-1,1)$. Moreover, by~\cite[Thm~7.4]{MR1819201}, if $X$ is strictly $1$-semi-stable (i.e. if $\int_{(-b,b)\setminus(-1,1)}x\nu(\D x)=0$), then $\s_1(r)=\infty$ for all $r\in\R$, in which case $X$ is strongly eroded by Theorem~\ref{thm:CM_local_smoothness} and~\eqref{eq:vigon_identity}. Otherwise (i.e., if $X$ is not strictly $1$-semi-stable), then $\s_1$ is locally bounded by~\cite[Thm~7.4]{MR1819201}, making $X$ abrupt. In both cases, the tails of the L\'evy measure of $X$ need not be regularly varying at $0$. In particular, this gives examples of strongly eroded processes with possibly asymmetric L\'evy measures that are not regularly varying at $0$. However, all these examples have $\beta_-=\beta_+=1$.
\end{ex}

\subsection{Oscilating characteristic exponent}

The fact that abrupt processes are not closed under addition is obvious since any strongly eroded process is the sum of two spectrally one-sided processes, both of which creep and are thus abrupt. In contrast, proving that strongly eroded processes are not closed under addition requires us to look at processes with oscillating characteristic function in the sense that $|\psi(u)/u|$ as a finite lower limit and an infinite upper limit as $|u|\to\infty$. 

\begin{ex}[Strongly eroded with mild oscillation]
\label{ex:fluctuation}
Define the function $\rho:(0,e^{-e})\to(0,\infty)$ given by $\rho(1/x)\coloneqq (1+\sin(\log\log x)^2\cdot\log x\cdot (\log\log x)^2)\1_{(e^e,\infty)}(x)$. We claim that $\rho$ is slowly varying at~$0$. By Karamata's representation theorem~\cite[Thm~1.3.1]{MR1015093}, it suffices to show that $h(x)\coloneqq\log\rho(e^{-x})$ (i.e. $\rho(x)=e^{h(\log(1/x))}$) satisfies $h'(x) \to 0$ as $x \to \infty$. This is easy to see in our case since we have $h(x)=\log(1+\sin(\log x)^2\cdot x\cdot(\log x)^2)$, establishing the slow variation of $\rho$. 

%To see this, it suffices to show that $|p(\lambda x)-p(x)|/p(x)\to 0$ as $x\da 0$ for any $\lambda>1$. Note from the triangle inequality that
%\[
%\frac{|p_3(\lambda x)-p_3(x)|}{p_3(x)}
%\le\frac{\log\lambda}{p_3(x)}
%+\frac{|\sin^2(\log\log (\lambda x))-\sin^2(\log\log x)|p_1(\lambda x)}{p_3(x)}
%+\frac{\sin^2(\log\log x)|p_1(\lambda x)-p_1(x)|}{p_3(x)}.
%\]
%We show that all three terms on the right tend to $0$ as $x\da 0$. The first term obviously tends to $0$ since $p_3(x)\to\infty$ as $x\to0$. Since $x\mapsto \sin^2(x)$, $x\in\R$, and $x\mapsto\log(1+x)$, $x>0$, are Lipschitz, then
%\[
%\frac{|\sin^2(\log\log (\lambda x))-\sin^2(\log\log x)|p_1(\lambda x)}{p_3(x)}
%\le \log\Big(1+\frac{\log\lambda}{\log x}\Big)\frac{p_1(\lambda x)}{\log x}
%\le \log\lambda\cdot\frac{p_1(\lambda x)}{\log^2 x}\to 0,
%\enskip\text{as }x\to \infty.
%\]
%Finally, we have $\sin^2(\log\log x)|p_1(\lambda x)-p_1(x)|/p_3(x)
%\le|p_1(\lambda x)-p_1(x)|/p_1(x)\to 0$, as $x\to \infty$.

Let $X$ be symmetric with $\nu(\D x)/\D x=x^{-2}\rho(x)\1_{(0,e^{-e})}(x)$ and $\sigma=0$, implying that $\Im\psi(u)=0$ and $|\Re\psi(u)|$ is bounded between multiples of $|u|\rho(1/|u|)$ as $|u|\to\infty$ (see Lemma~\ref{lem:karamata_Levy} below). Thus, for some $c>0$ and all sufficiently large $|u|$, we have
\[
\Re\frac{1}{1-iur-\psi(u)}
=\frac{1-\Re\psi(u)}{u^2r^2 + (1-\Re\psi(u))^2}
\ge\frac{c}{|u|(1+\sin(\log\log|u|)^2\cdot\log|u|\cdot(\log\log|u|)^2)}.
\]
We claim that $X$ is strongly eroded. To see this, note that $\sin(k\pi+u)^2\le u^2$ for any $k\in\N$ and $u\in\R$, implying that $\sin(\log\log u)^2\cdot(\log\log u)^2\le 1$ for $u\in[\exp(e^{k\pi-1/(k\pi)}),\exp(e^{k\pi})]$. Hence, we have
\[
\int_{e^e}^\infty\frac{\D u}{u(1+\sin(\log\log u)^2\cdot\log u\cdot(\log\log u)^2)}
\ge\sum_{k=1}^\infty
    \int_{\exp(e^{k\pi-1/(k\pi)})}^{\exp(e^{k\pi})}\frac{\D u}{2u\log u}\\
=\sum_{k=1}^\infty\frac{1}{2k\pi}=\infty,
\]
proving that $X$ is strongly eroded. Note that Theorem~\ref{thm:zeta_criterion} is inapplicable as $\liminf_{u\to\infty}|\psi(u)/u|<\infty$ and $\limsup_{u\to\infty}|\psi(u)/u|=\infty$.
\end{ex}

\begin{ex}[Eroded processes are not closed under addition]
\label{ex:addition2}
Define the function  $\varrho:(0,e^{-e})\to(0,\infty)$ given by $\varrho(1/x)=(1+\cos(\log\log x)^2\cdot\log x\cdot(\log\log x)^2)\1_{(0,e^{e})}(x)$. A similar argument to the one made in Example~\ref{ex:fluctuation} above shows $\varrho$ is slowly varying at $0$ and yields another strongly eroded process. However, the sum $X$ of such a process and the one from Example~\ref{ex:fluctuation} above is symmetric and with L\'evy measure $\nu$ satisfying $\nu([x,\infty))=x^{-1}(2+\log(1/x)(\log\log(1/x))^2)\1_{(0,e^{-e})}(x)$. Lemma~\ref{lem:karamata_Levy} then implies that for some $c>0$ and all sufficiently large $|u|$ we have
\[
\Re\frac{1}{1-iur-\psi(u)}
\le \frac{1}{1-\Re\psi(u)}
\le \frac{c}{|u|\cdot\log|u|\cdot(\log\log|u|)^2}.
\]
Since $\int_{e^e}^\infty(u\cdot\log u\cdot(\log\log u)^2)^{-1}\D u=1<\infty$, the process $X$ is abrupt by Theorem~\ref{thm:CM_local_smoothness} and~\eqref{eq:vigon_identity}.
\end{ex}

The fluctuations present in the previous examples are tame enough for us to determine decisively that $\s_1$ is identically infinite and hence not locally integrable. In the following example we find a symmetric process for which we may show that $\s_1(0)=\infty$ but for which, as a consequence of the heavy oscillations of its characteristic function, it is incredibly hard to find whether $\s_1(r)$ is finite or not for \emph{any} given $r\ne 0$. The oscillations of the characteristic exponent in particular satisfy both $\liminf_{|u|\to\infty}|\psi(u)|=0$ (hence $\beta_-=0$) and $\limsup_{|u|\to\infty}|\psi(u)|/|u|^\alpha\ge 4$ for a constant $\alpha\in(1,2)$ (which in fact agrees with $\beta_+$) that may be taken arbitrarily close to $2$. In particular, this symmetric process is a prime candidate for one of the two interesting possibilities: (I) a counter-example to Conjecture~\ref{conj:vigon}, as $\s_1(0)=\infty$ but $\s_1(r)$ is possibly finite for some $r\ne 0$ (which may possibly result in a non-eroded, non-abrupt process) or (II) a strongly eroded process with path variation $\beta_+=\alpha\in(1,2)$ arbitrarily close to that of a Brownian motion. When the process is asymmetric, however, it is abrupt.

\begin{ex}[Can an eroded L\'evy process have path variation close to that of Brownian motion?]
\label{ex:orey}
%\cite[Ex.~41.23]{MR3185174}
We recall the definition of Orey's process~{\cite{MR226701}}.
Fix any $\alpha\in(1,2)$, $c_\pm\ge0$ with $c_++c_->0$ and integer $\eta>2/(2-\alpha)$. Set $\sigma=0$, $\gamma=0$ and $\nu=\sum_{n\in\N}a_n^{-\alpha}(c_+\delta_{a_n}+c_-\delta_{-a_n})$ for $a_n=2^{-\eta^n}$. Then we have $\int_{(-1,1)} |x|\nu(\D x)=(c_++c_-)\sum_{n\in\N}a_n^{1-\alpha}=\infty$, making the associated L\'evy process of infinite variation. Since $\cos(2k\pi)=1$ for every integer $k$ and $1-\cos(x)\le x^2$ for $x\in[0,1]$, 
\begin{align*}
-\Re\psi(2\pi/a_n)
&=(c_++c_-)\sum_{k=1}^\infty a_k^{-\alpha}(1-\cos(2\pi a_k/a_n))
=(c_++c_-)\sum_{k=1}^\infty 2^{\alpha\eta^k}\big(1-\cos\big(2\pi 2^{\eta^n-\eta^k}\big)\big)\\
&\le 4(c_++c_-)\pi^2\sum_{k=n+1}^\infty 2^{2\eta^n-(2-\alpha)\eta^k}
=4(c_++c_-)\pi^2\sum_{k=1}^\infty 2^{-\eta^n((2-\alpha)\eta^k-2)}
\xrightarrow[n\to\infty]{} 0,
\end{align*}
where the limit follows by the monotone convergence theorem and the inequality $\eta>2/(2-\alpha)$. Thus, $\liminf_{u\to\infty}|\Re\psi(u)|=0$. Similarly,
$-a_n^\alpha\Re\psi(\pi/a_n)
=(c_++c_-)a_n^\alpha\sum_{k=1}^\infty a_k^{-\alpha}(1-\cos(\pi a_k/a_n))
%=2a_n\sum_{k=1}^\infty a_k^{-\alpha}\big(1-\cos\big(\pi 2^{\eta^n-\eta^k}\big)\big)
\ge 4$, implying $\limsup_{u \to \infty} |\psi(u)|/|u|^\alpha\ge 4$. 

Suppose $c_+=c_-$. Then $\Im\psi=0$ and, since $u\mapsto 1/(1-\psi(u))$ is the characteristic function of $X_{\mathrm{e}_1}$, where $\mathrm{e}_1$ is a unit-mean exponential time independent of $X$, and $\limsup_{|u|\to\infty}1/(1-\psi(u))=1>0$, the Riemann--Lebesgue lemma implies that $X_{\mathrm{e}_1}$ is singular continuous. In particular, $u\mapsto1/(1-\psi(u))$ is not integrable on $\R$, giving $\s_1(0)=\infty$. If instead $c_+\ne c_-$, then $X$ is abrupt by Proposition~\ref{prop:asymmetry}. 

Furthermore, we note that $\beta_-=0$ and $\beta_+=\alpha$, with the strong oscillation of $\psi$ resulting in a large gap between these indices. To see that $\beta_+=\alpha$, note that $\int_{(-1,1)}|x|^p\nu(\D x)=2\sum_{n=1}^\infty 2^{(\alpha-p)\eta^n}$, where the sum diverges for all $p\le \alpha$ and converges for all $p>\alpha$. Since $\ov\sigma^2(u)=2\sum_{n\in\N, \, a_n<u} a_n^{2-\eta}$, the lower limit of $u^{-2}\ov\sigma^2(u)$ is attained along the sequence $a_n$ and 
\[
\liminf_{n\to\infty}a_n^{-2}\ov\sigma^2(a_n)
=\liminf_{n\to\infty}2\sum_{k=n+1}^\infty 2^{2\eta^n-(2-\alpha)\eta^k}
= \liminf_{n\to\infty}2\sum_{k=1}^\infty 2^{-\eta^n((2-\alpha)\eta^k-2)}=0,
\]
by the monotone convergence theorem. 
\end{ex}

\subsection{L\'evy measure with regularly varying tails}
\label{subsec:reg_var}
%A more explicit criteria to determine when $X$ is abrupt or eroded is often available when the tails of $\nu$ are regularly varying. 
We begin with some estimates on the characteristic exponent $\psi(u)=-\sigma^2u^2/2 + iu\gamma +\int_{\R}(e^{iux}-1-iux\1_{\{|x|<1\}})\nu(\D x)$ for $u\in\R$ in terms of commonly used functions of the L\'evy measure for the proof of Proposition~\ref{prop:index_criteria}. Recall that $\ov\sigma^2(x)=\int_{(-x,x)}y^2\nu(\D y)$ for $x>0$ and define the functions $\ov\nu(x)\coloneqq\nu(\R\setminus(-x,x))$, $x\in(0,1)$, and
\begin{equation}
\label{eq:ov_nu_sigma}
\ov{\gamma}(x)\coloneqq\int_{(-1,1)\setminus(-x,x)}y\nu(\D y),
\quad \text{for }x>0.
\end{equation}

\begin{lem}\label{lem:LevyBounds}
%{\normalfont(a)} For any $\alpha>0$ and all $u\in(0,1]$,
%\begin{equation}\label{eq:Bounds_sigma_nu}
%\overline{\sigma}^{2}(u)\leq|u|^{2-\alpha}I_{\alpha},
%\quad\overline{\nu}(u)\leq\overline{\nu}(1)+I_{\alpha}u^{-\alpha}.
%\end{equation}
{\normalfont(a)} For any $u\ne 0$, 
$\tfrac{1}{3}u^2\ov{\sigma}^2(|u|^{-1})
\leq-\Re\psi(u)-\tfrac{1}{2}u^2\sigma^2
\leq 2\ov{\nu}(2|u|^{-1})+
\tfrac{1}{2}u^2\ov{\sigma}^2\big(2|u|^{-1}\big)$.\\
{\normalfont(b)} For any $|u|> 1$, we have 
$|\Im\psi(u)+ (\ov{\gamma}(|u|^{-1})-\gamma)u|
\leq \tfrac{1}{3}u^2 \ov{\sigma}^2(|u|^{-1}) + \ov{\nu}(|u|^{-1})$.
\end{lem}

\begin{proof} %[Proof of Lemma~\ref{lem:LevyBounds}]
%(a) Multiplying the integrand by $(|u|/|x|)^{2-\alpha}>1$ and then completing the integration set to $(-1,1)$ gives the first inequality. Similarly, the second one follows from
%\[
%\ov{\nu}(u)-\ov{\nu}(1)=\int_{(-1,1)\backslash(-u,u)}\nu(\D x)
%\leq u^{-\alpha}\int_{(-1,1)\backslash(-u,u)}|x|^\alpha\nu(\D x)
%\leq I_\alpha u^{-\alpha}.
%\]
(a) Note that 
$\frac{1}{3}x^2 \1_{\{|x|<1\}}
\le1-\cos(x)\le\frac{1}{2}\min\{x^2,4\}$ 
for all $x\in\R$. Integrating then gives
\[
\frac{1}{3}\int_\R(ux)^2 \1_{\{|ux|<1\}}\nu(\D x)
\leq \int_\R(1-\cos(ux))\nu(\D x)
\leq\frac{1}{2}\int_\R \min\{(ux)^2,4\}\nu(\D x),
\]
implying the inequality in (a).\\
(b) %The first two bounds follow along the same lines as before, using the inequalities, respectively, $|\sin(ux)|\le|ux|^\alpha\1_{\{|x|<1\}}+\1_{\{|x|\ge 1\}}$ and $|\sin(ux)-ux\1_{\{|x|<1\}}|\le\1_{\{|x|<1\}}|ux|^\alpha +\1_{\{|x|\ge 1\}}$. For the third one, 
First note that
\[
\int_\R (\sin(ux)-ux\1_{\{|x|<1\}})\nu(\D x)
=-u\ov{\gamma}\big(|u|^{-1}\big)
+\int_\R (\sin(ux)-ux\1_{\{|ux|<1\}})\nu(\D x).
\]
Hence, integrating $|\sin(ux)-ux\1_{\{|ux|<1\}}|\le \frac{1}{3}\1_{\{|ux|<1\}}u^2x^2+\1_{\{|ux|\ge 1\}}$ gives the result.
\end{proof}

Throughout, we use the notation $f(x) \sim g(x)$ as $x \to a$ if $f(x)/g(x) \to 1$ as $x \to a$, $f(x)=\Oh(g(x))$ as $x\to a$ if $\limsup_{x\to a}f(x)/g(x)<\infty$ and $f(x)\approx g(x)$ if both $g(x)=\Oh(f(x))$ and $f(x)=\Oh(g(x))$. Assume throughout the remainder of this section that, for some $\alpha\in[0,2]$, the functions
\begin{equation}
\label{eq:rv_nu_density}
%\ov\nu_-(x):=
\wp_-(x)\coloneqq x^\alpha\nu((-\infty,-x])
\quad\text{and}\quad
%\ov\nu_+(x):=
\wp_+(x)\coloneqq x^{\alpha}\nu([x,\infty))
\quad\text{for }x\in(0,1),
\end{equation}
are slowly varying at $0$ (see definition in~\cite[Sec.~1.2]{MR1015093}).
The infinite variation of $X$ requires either $\alpha\ge 1$ or $\sigma^2>0$. However, if either $\alpha>1$ or $\sigma^2>0$, then $X$ is abrupt by Proposition~\ref{prop:index_criteria}. Thus, without loss of generality we assume $\alpha=1$ and $\sigma=0$ throughout the remainder of this section. Moreover, since we may modify arbitrarily the L\'evy measure of $X$ away from $0$ without changing $\mL(\mS)$ (by Proposition~\ref{prop:fin_var_perturbation}), we may assume that $\nu$ is supported on $(-1,1)$.

The following result controls the real and imaginary parts of the characteristic exponent $\psi$. This is important in determining whether $X$ is abrupt or strongly eroded because they feature in the integrand in the definition of $\s_1(r)$. 
%is of the form $(1-\Re\psi(u))/((1-\Re\psi(u))^2+(ur+\Im\psi(u))^2)$. 
The proof of Lemma~\ref{lem:karamata_Levy} is based on Karamata's theorem and the elementary estimates from Lemma~\ref{lem:LevyBounds}. Define the functions $\wp(x)\coloneqq \wp_+(x) + \wp_-(x)$ and $\wt\wp_\pm(x)\coloneqq \int_x^1 t^{-1}\wp_\pm(t)\D t$ for $x\in(0,1)$. Note that the infinite variation of $X$ is equivalent to $\int_0^1 x^{-1}\wp(x)\D x=\infty$, which is further equivalent to $\lim_{x\da 0}(\wt\wp_+(x)+\wt\wp_-(x))=\infty$. Moreover, we see that
the functions $\wt\wp_\pm$ are slowly varying at $0$ with $\lim_{x\da 0}\wt\wp_\pm(x)/\wp_\pm(x)=\infty$  by~\cite[Prop.~1.5.9a]{MR1015093}. 

\begin{lem}
\label{lem:karamata_Levy}
Suppose $\sigma=0$ and the L\'evy measure $\nu$ is supported on $(-1,1)$ and satisfies~\eqref{eq:rv_nu_density} for $\alpha=1$. Then $\ov\gamma$ in~\eqref{eq:ov_nu_sigma} satisfies $\ov\gamma(x)=\wp_+(x)+\wt\wp_+(x)-(\wp_-(x)+\wt\wp_-(x))$ for all $x\in(0,1)$ and
\begin{equation}
\label{eq:RV-psi}
\Re\psi(u)\approx |u|\wp\big(|u|^{-1}\big)
\quad\text{and}\quad
\Im\psi(u)=\big(\gamma - \ov\gamma\big(|u|^{-1}\big)\big)u
+\Oh\big(|u|\wp\big(|u|^{-1}\big)\big)
\quad\text{as }|u|\to\infty.
\end{equation}
\end{lem}

\begin{proof}
The formula for $\ov\gamma$ follows by applying Fubini's theorem. Similarly, Fubini's theorem yields 
\[
\ov\sigma^2(x)
=\int_{(-x,x)}y^2\nu(\D y)
=\int_0^x 2y(\ov\nu(y)-\ov\nu(x))\D y
=\int_0^x 2y\ov\nu(y)\D y - x^2\ov\nu(x).
\]
Note that $\ov\nu(x)=x^{-1}\wp(x)$ for all $x\in(0,1)$. Karamata's theorem~\cite[Thm~1.5.11]{MR1015093} then shows that $\int_0^x 2y\ov\nu(y)\D y\sim 2x\wp(x)$ while $x^2\ov\nu(x)=x\wp(x)$, implying that $\ov\sigma^2(x)\sim x\wp(x)$ as $x\da 0$. Then the estimates in~\eqref{eq:RV-psi} follow from Lemma~\ref{lem:LevyBounds}.
\end{proof}

Lemma~\ref{lem:karamata_Levy} provides sufficient control over the characteristic exponent of $X$ in two regimes: if $\nu$ is near-symmetric (i.e. $\ov\gamma(x)=\Oh(1)$ as $x\da 0$, see definition in~\eqref{eq:ov_nu_sigma}), or if $\nu$ is skewed (i.e. $\lim_{x\da 0}x^{-1}\ov\gamma(x)/\ov\nu(x)\in\{-\infty,\infty\}$, which is equivalent to $\lim_{x\da 0}(\wt\wp_+(x)-\wt\wp_-(x))/\wp(x)\in\{-\infty,\infty\}$ by Lemma~\ref{lem:karamata_Levy}), motivating the two ensuing subsections. 
In the remainder of this section we freely apply equivalence~\eqref{eq:vigon_identity} and Theorems~\ref{thm:CM_global_smoothness} and~\ref{thm:zeta_criterion}.

\subsubsection{Near-symmetric L\'evy measure} 
Suppose $\ov\gamma$ in~\eqref{eq:ov_nu_sigma} satisfies $\ov{\gamma}(x)=\Oh(1)$ as $x\da 0$ (e.g., $\nu$ symmetric). Lemma~\ref{lem:karamata_Levy} gives $|\Im\psi(u)/u|=\Oh(1\vee\wp(|u|^{-1}))$ and thus $|\Im\psi(u)|=\Oh(|u|\vee\Re\psi(u))$ as $|u|\to\infty$. Thus, the integrand in the definition of $\s_1(r)$ is asymptotically bounded between multiples of $\wp(1/|u|)/(u(1+\wp(1/|u|)^2))$ as $|u|\to\infty$ for all $r\ne\gamma$.

%Under the following two (non-exhaustive) cases, we are able to determine when $X$ is abrupt or strongly eroded.

\begin{ex}[Near-symmetric with high activity]
\label{ex:symm_high_activity}
Suppose $\liminf_{x\da0}\wp(x)=\infty$. Then we have $\liminf_{|u|\to\infty}|\psi(u)/u|=\infty$ and Theorem~\ref{thm:zeta_criterion}(ii-a) shows that $X$ is either eroded or abrupt. Moreover, by Lemma~\ref{lem:karamata_Levy}, we have
\[
\Re\frac{1}{1+iur-\psi(u)}
=\frac{1-\Re\psi(u)}{(1-\Re\psi(u))^2+(ur-\Im\psi(u))^2}
\approx\frac{1}{\Re\psi(u)}
\approx\frac{1}{|u|\wp(|u|^{-1})}
\quad\text{as}\enskip|u|\to\infty.
\]
Therefore $X$ is strongly eroded if and only if $x\mapsto 1/(x\wp(x))$ is not integrable at $0$.
\end{ex}

\begin{ex}[Near-symmetric with low activity]
\label{ex:symm_low_activity}
Suppose $\limsup_{x\da0}\wp(x)<\infty$. Then Lemma~\ref{lem:karamata_Levy} implies that $\limsup_{|u|\to\infty}|\psi(u)/u|<\infty$ so Theorem~\ref{thm:zeta_criterion}(i) shows that $X$ is strongly eroded.
\end{ex}
%This can also be deduced directly. Indeed, we have $|\Re\psi(u)|=\Oh(|u|)$ and thus, for some $c>0$,
%\[
%\Re\frac{1}{1+iur-\psi(u)}
%=\frac{1-\Re\psi(u)}{(1-\Re\psi(u))^2+(ur-\Im\psi(u))^2}
%\ge\frac{c\Re\psi(u)}{u^2}
%\approx\frac{\wp(|u|^{-1})}{|u|}
%\quad\text{as}\enskip |u|\to\infty.
%\]
%Its integral is infinite since $X$ has infinite variation (i.e. $\int_0^1x^{-1}\wp(x)\D x=\infty$). Thus, $X$ is eroded.
%\emph{Case 3} ($\limsup_{x\to0}\wp(x)=\infty>\liminf_{x\to0}\wp(x)$). In this case behaviour may be too wild so one would need a measure of how often $\wp$ approaches its tail maximum and tail minimum. A simple example is $\wp(x)=\exp(-(\log x)^{1/3}\cos((\log x)^{1/3}))$.\\

\subsubsection{Skewed L\'evy measure}
Suppose the function $\ov\gamma$ in~\eqref{eq:ov_nu_sigma} satisfies $\lim_{x\da 0}x^{-1}\ov\gamma(x)/\ov\nu(x)\in\{-\infty,\infty\}$ (equivalently, $\lim_{x\da 0}(\wt\wp_+(x)-\wt\wp_-(x))/\wp(x)\in\{-\infty,\infty\}$). This is the case if, for instance, we have $\liminf_{x\da0}\wp_+(x)/\wp_-(x)>1$ or $\limsup_{x\da0}\wp_+(x)/\wp_-(x)<1$ by~\cite[Prop.~1.5.9a]{MR1015093}. Moreover, either of these inequalities essentially imply that $X$ is abrupt. More precisely, if $x\mapsto\nu([x,1))=x^{-1}\wp_+(x)$ and $x\mapsto\nu((-1,-x])=x^{-1}\wp_-(x)$ are eventually differentiable with a monotone derivative as $x\da 0$, then the monotone density theorem~\cite[Thm~1.7.2]{MR1015093} shows that the respective derivatives are asymptotically equivalent to $-x^{-2}\wp_+(x)$ and $-x^{-2}\wp_-(x)$. Hence, the Radon-Nikodym derivative $\nu(\D x)/\nu(\D(-x))$ is asymptotically equivalent to $\wp_+(x)/\wp_-(x)$ as $x\da 0$. Thus, either of the limits $\liminf_{x\da0}\wp_+(x)/\wp_-(x)>1$ or $\limsup_{x\da0}\wp_+(x)/\wp_-(x)<1$ imply that $X$ is abrupt by Proposition~\ref{prop:asymmetry}. 

The following example shows that the condition in Proposition~\ref{prop:asymmetry} is close to being sharp. It constructs strongly eroded processes whose asymmetry, quantified by the quotient $(\wp_+(x)-\wp_-(x))/\wp_-(x)>0$, converges (arbitrarily) slowly to $0$ as $x\da0$. Clearly, the roles of $\wp_+$ and $\wp_-$ could be reversed without affecting these conclusions. 
%constructs both strongly eroded and abrupt processes with $\wp_+$ being arbitrarily close to $\wp_-$, further showing how delicate the distinction may be. 
Moreover, since Cauchy process is strongly eroded but spectrally one-sided infinite variation processes are abrupt, the following example also shows that the sum of an abrupt process and an independent strongly eroded process may result in an abrupt or a strongly eroded process.

\begin{ex}[Low asymmetry]
\label{ex:low_asymmetry}
Suppose $\wp_+(x)=(p(x)+q(x))\1_{(0,\ve)}(x)$ and $\wp_-(x)=p(x)\1_{(0,\ve)}(x)$ for positive slowly varying functions $p$ and $q$ defined on $(0,\ve)$. Define recursively $\log^{(1)}(x)=\log x$ and $\log^{(n+1)}(x)=\log(\log^{(n)}(x))$ for $n\in\N$. Fix $n\in\N\cup\{0\}$, define the functions $p(x)=1/\prod_{k=1}^n\log^{(k)}(1/x)$
and $q(x)=p(x)/\log^{(n+1)}(1/x)$ (where an empty product equals $1$ by convention) and choose $\ve$ sufficiently small to ensure $p$ and $q$ are both positive and $x\mapsto x^{-1}\wp_\pm(x)$ are monotone on $(0,\ve)$. Then Lemma~\ref{lem:karamata_Levy} gives $|\Re\psi(u)|\approx |u|p(1/|u|)$ and $|\Im\psi(u)|\approx |u|\log^{(n+2)}(|u|)$ as $|u|\to\infty$. Since $|u|^{-1}p(1/|u|)/(\log^{(n+2)}(|u|))^2$ is not integrable at infinity, then $\s_q(r)=\infty$ for all $r\in\R$, making $X$ strongly eroded with slowly varying asymmetry $\wp_+(x)/\wp_-(x)-1=q(x)/p(x)=1/\log^{(n+1)}(1/x)\to 0$ as $x\downarrow 0$ and with $|\Im\psi(u)/\Re\psi(u)|\approx \log^{(n+2)}(|u|)/p(1/|u|)\to\infty$ as $|u|\to\infty$. 

A similar analysis shows that the choice $q(x)=p(x)/\log^{(n)}(1/x)$ instead leads to an abrupt process. We point out that, in either case, $p$ cannot be much smaller since the infinite variation of $X$ requires the function $u^{-2}|\Re\psi(u)|\approx |u|^{-1}p(1/|u|)$ to be non-integrable at infinity by Lemma~\ref{lem:VigonThesis}.
\end{ex}

\subsection{Subordination}
\label{sec:subordination}

Let $X$ be an infinite variation L\'evy process and $Y$ be an independent driftless subordinator with Fourier--Laplace exponent $\phi(u)\coloneqq\log\e[\exp(uY_1)]$ for any $u\in\Complex$ with $\Re u\le 0$. Then, for any $c\ge 0$, the subordinated process $Z=(X_{ct+Y_t})_{t\ge 0}$ is L\'evy with characteristic exponent given by $u\mapsto \phi(\psi(u))+c\psi(u)$. The following example shows that subordinating an abrupt processes can result in either an abrupt or a strongly eroded process $Z$ and that subordinated strongly eroded processes can still be strongly eroded. 

\begin{ex}[Subordinating abrupt and strongly eroded processes]
\label{ex:subordination}
(a) Suppose $X$ is a Brownian motion, $Y$ is an $\alpha$-stable subordinator (with $\alpha\in(0,1)$) and $c=0$. Then $Z$ is a symmetric $2\alpha$-stable process, making it abrupt for $\alpha>1/2$, strongly eroded for $\alpha=1/2$ and of finite variation for $\alpha<1/2$.\\
(b) Suppose $c>0$ and $(X_{Y_t})_{t\ge 0}$ is of finite variation. Then $Z$ can be decomposed as the sum of two independent processes, one with the law of $(X_{ct})_{t\ge 0}$ and the other with the law of $(X_{Y_t})_{t\ge 0}$. Thus, Proposition~\ref{prop:fin_var_perturbation} implies $\mL(\mS_X)=c\mL(\mS_Z)$, where $\mS_X$ and $\mS_Z$ are the set of slopes of the convex minorants of $X$ and $Z$, respectively, with convention $cA:=\{ca:a\in A\}$ and $c\emptyset:=\emptyset$.\\
(c) Suppose $\lim_{|u|\to\infty}|\psi(u)/u|=\infty$ and $c>0$. Then $Z$ satisfies the conditions of Theorem~\ref{thm:zeta_criterion}(ii), making the process either strongly eroded or abrupt.\\
(d) Suppose $\limsup_{|u|\to\infty}|\psi(u)/u|<\infty$ and $Z$ is of infinite variation. Then, for some $C\ge 0$ and all $z\in\Complex$, we have $|\phi(z)|\le C|z|$ (see, e.g.~\cite[Sec.~1.2, p.~7]{MR1746300}) implying that $Z$ satisfies the assumptions of Theorem~\ref{thm:zeta_criterion}(i), making it strongly eroded. In particular, this is the case if $X$ is symmetric Cauchy of unit scale (i.e. the law of $X_1$ has parameters $(\lambda_1,\lambda_2)=(1,0)$ as in Example~\ref{ex:Cauchy-normal-attraction} above), $c=0$ and $Y$ has L\'evy measure $\nu_Y(\D t) = t^{-2}(\log (1/t))^{-2}\1_{(0,1)}(t)\D t$. Indeed, it suffices to verify that $Z$ has infinite variation. By~\cite[Thm~30.1]{MR3185174}, the L\'evy measure $\nu_Z$ of $Z$ is given by $\nu_Z(\D x)=\int_0^1\p(X_t\in\D x)\nu_Y(\D t)$, thus Fubini's theorem yields
\begin{align*}
\int_{(-1,1)}|x|\nu_Z(\D x)
&=\int_{(-1,1)}|x|\int_0^1 \p(X_t\in\D x)\frac{\D t}{t^2(\log (1/t))^{2}}\\
&=\int_0^1\int_0^1 \frac{x\D x}{t^2+x^2}\frac{2\D t}{\pi t(\log (1/t))^{2}}
=\int_0^1\frac{\log(1+t^{-2})}{\pi t(\log (1/t))^{2}}\D t
\ge\int_0^1\frac{\D t}{\pi t\log(1/t)}=\infty.\qedhere
\end{align*}
\end{ex}

\section{Zero-one law for stick breaking and the slopes of the minorant}
\label{sec:0-1lawsbres}

For $T>0$, let $( \lambda_n)_{n\in\N}$ be a uniform stick-breaking process on $[0,T]$, defined recursively in terms of an i.i.d. $\U(0,1)$ sequence $(U_n)_{n\in\N}$ as follows: $L_0\coloneqq T$, $ L_n\coloneqq U_nL_{n-1}$ and $\lambda_n\coloneqq L_{n-1}- L_n$ for $n\in\N$. 
Let $(V_n)_{n\in\N}$
be an
i.i.d. sequence of $\U(0,1)$ random variables, independent of the stick-breaking process
$(\lambda_n)_{n\in\N}$.
%The process $(L_n)_{n\in\N\cup\{0\}}$ is known as the %stick-remainder process. 
Denote $\R_+:=[0,\infty)$ and recall that a measurable function $f$ is in $\Loc(0+)$ if for some $\ve>0$ 
it satisfies $\int_0^\ve |f(t)|\D t<\infty$. 
%valid for any non-negative and measurable function $f$. 
The following zero-one law, which does not involve the L\'evy process $X$, is key for the analysis of the regularity of the convex hull of $X$.

\begin{thm}
\label{thm:0-1_law_formula}
Let $\varphi:\R_+\times[0,1]\to\R_+$ be measurable and bounded. Define  $\Sigma_T\coloneqq\sum_{n=1}^\infty\varphi(\lambda_n,V_n)$ and the function
$\phi:t\mapsto\int_0^1\varphi(t,u)\D u$. Then 
$\Sigma_T$ is either a.s. finite or a.s. infinite, characterised by
%$\e\Sigma_T=\int_0^Tt^{-1}\phi(t)\D t$ and 
\begin{gather}
\label{eq:0-1-Law}
\Sigma_T<\infty\enskip\text{a.s.}
\quad\iff\quad
%\int_0^1\Phi(t)\frac{\D t}{t}=\infty
%\quad\iff\quad
t\mapsto t^{-1}\phi(t)\in\Loc(0+).
\end{gather}
Moreover, the mean of $\Sigma_T$ is given by $\e\Sigma_T=\int_0^Tt^{-1}\phi(t)\D t$.
\end{thm}

Note that, by~\eqref{eq:0-1-Law} in Theorem~\ref{thm:0-1_law_formula}, $\Sigma_T$ is either a.s. finite for all $T>0$ or a.s. infinite for all $T>0$. 
Furthermore, the proof of Theorem~\ref{thm:0-1_law_formula} implies that 
$\Sigma_T<\infty$ a.s. if and only if
$t\mapsto t^{-1}\Phi(t)\in\Loc(0+$, 
where
$\Phi(t):=t^{-1}\int_0^t\phi(s)\D s$.

%\subsection{Proof of %Theorem~\ref{thm:0-1_law_formula}}
\begin{proof} 
%Define $\Phi:t\mapsto t^{-1}\int_0^t\phi(s)\D s$. 
Proving that $\Sigma_T$ is either a.s. finite or a.s. infinite, according to~\eqref{eq:0-1-Law}, requires three steps. First, we show that the events $\{\Sigma_T=\infty\}$ and $\{\wt\Sigma_T=\infty\}$ agree a.s., where  $\wt\Sigma_T\coloneqq\sum_{n=1}^\infty\Phi(L_n)$. 
%and recall $\Sigma_T=\sum_{n=1}^\infty\varphi(\lambda_%n,V_n)$. 
Second, we use the Poisson process embedded in the stick remainders  $(L_n)_{n\in\N}$ to establish that $\p(\wt\Sigma_T=\infty)=1$ if $\int_0^1t^{-1}\Phi(t)\D t=\infty$ and otherwise $\p(\wt\Sigma_T=\infty)=0$.
Third, we use the Poisson point process, given by the stick-breaking process on an independent exponential time horizon, to establish~\eqref{eq:0-1-Law}.

Define the filtration $(\mathcal{F}_n)_{n\in\N \cup\{0\}}$ by $\mathcal{F}_{0}\coloneqq\{\emptyset,\Omega\}$ and $\mathcal{F}_{n}\coloneqq\sigma((U_k,V_k);k\leq n)$ for $n\ge1$. Note that the conditional distribution of $\lambda_n$, given $\mathcal{F}_{n-1}$, is uniform on the interval  $(0,L_{n-1})$, implying 
\[
\e[\varphi(\lambda_n,V_n)|\mathcal{F}_{n-1}]
=\e[\phi(\lambda_n)|\mathcal{F}_{n-1}]
=L_{n-1}^{-1}\int_{0}^{L_{n-1}}\phi(s)\D s
=\Phi(L_{n-1}),
\quad n\in\N.
\]
Hence, the process 
$(M_n)_{n\in\N\cup\{0\}}$, given by $M_0\coloneqq0$ and $M_n\coloneqq\sum_{k=1}^n(\varphi(\lambda_k,V_k)-\Phi(L_{k-1}))$ for $n\in\N$, is a $(\mathcal{F}_n)$-martingale with bounded increments (recall that $\varphi$ is bounded). By~\cite[Prop.~7.19]{MR1876169}, the event $A:=\{M_n\text{ converges to a finite limit}\}$ satisfies the following equality  
\begin{equation}
\label{eq:equiv_events}
A=\left\{\sup_{n\in\N }M_n<\infty\right\}
=\left\{\inf_{n\in\N }M_n>-\infty\right\}
\quad\text{a.s.}
\end{equation}
On $A$ we have $\Sigma_T=\wt\Sigma_T+\lim_{n\to\infty}M_n$, implying that $\Sigma_T=\infty$ if and only if $\wt\Sigma_T=\infty$. On the complement of $A$, by~\eqref{eq:equiv_events}, we must have $\sup_{n\in\N}M_n=-\inf_{n\in\N}M_n=\infty$, implying $\Sigma_T=\wt\Sigma_T=\infty$. Thus, the events $\{\Sigma_T=\infty\}$ and $\{\wt\Sigma_T=\infty\}$ agree a.s. 

Note that $(-\log(U_n))_{n\geq1}$ are iid exponential random variables with unit mean. Hence, the process $(N(t))_{t\ge0}$, given by $N(t)\coloneqq\sum_{n=1}^\infty\1_{\{-\log(T^{-1}L_n)\leq t\}}$, is a standard Poisson process. Denote by $N(\D x)$ the corresponding Poisson point process on $(0,\infty)$. Since $\wt\Sigma_T=\int_{(0,\infty)}\Phi(Te^{-x})N(\D x)$, Campbell's formula~\cite[Lem.~12.2]{MR1876169} yields the Laplace transform of $\wt\Sigma_T$:
\[
\log\e[\exp(-q\wt\Sigma_T)]
=-\int_0^\infty\big(1-e^{-q\Phi(Te^{-x})}\big)\D x
=-\int_{0}^{T}\big(1-e^{-q\Phi(t)}\big)\frac{\D t}{t},
\quad q\geq0.
\]
Since $\varphi$ is bounded, there exists $q_0>0$ such that $0\leq \Phi(t)\leq 1/q_0$ for all $t>0$. The inequalities $x/2\le 1-e^{-x}\leq x$, valid  for $x\in[0,1]$, imply the following for all $q\in(0,q_0]$:
\begin{align}
\label{eq:integral_bounds_Phi}
(q/2)\int_{0}^{T}\Phi(t)\frac{\D t}{t} 
\leq \int_{0}^{T}\big(1-e^{-q\Phi(t)}\big)\frac{\D t}{t}
\leq
q\int_{0}^{T}\Phi(t)\frac{\D t}{t}.
\end{align}
The monotone convergence theorem implies
\begin{equation*}
%\label{eq:0-1-Phi}
\p(\wt\Sigma_T<\infty)
=\lim_{q\downarrow0}\e\big[e^{-q\wt\Sigma_T}\big]
=\lim_{q\downarrow0}\exp\bigg(-\int_{0}^{T}
\big(1-e^{-q\Phi(t)}\big)\frac{\D t}{t}\bigg)
=\begin{cases}
1, & t^{-1}\Phi(t)\in\Loc(0+),\\
0, & t^{-1}\Phi(t)\not\in\Loc(0+).
\end{cases}
\end{equation*}
Since $\{\Sigma_T=\infty\}=\{\wt\Sigma_T=\infty\}$ a.s., the first claim in the theorem follows.

In order to prove the equivalence in~\eqref{eq:0-1-Law}, note first that whether the function $t\mapsto t^{-1}\Phi(t)$ is in $\Loc(0+)$
does not depend on $T$. Hence the random variable $\wt\Sigma_T$ (and thus $\Sigma_T$) is either finite a.s. for all $T>0$ or infinite a.s. for all $T>0$. 
%which proves that $\Sigma$ is either a.s. finite or a.s. infinite, since $\p(\Sigma<\infty)=\p(\wt\Sigma<\infty)$. 
%In order to prove~\eqref{eq:0-1-Law}, it remains to %show that $\Sigma_T<\infty$ a.s. implies %$\int_0^1t^{-1}\phi(t)\D t<\infty$. 
%Assume $\Sigma_T<\infty$ a.s. for some $T>0$ and hence %for all $T>0$. %We may extend the definition of $\phi$ %to $\R_+\times[0,1]$ by setting it to 0 on %$(T,\infty)\times[0,1]$. 
Let $E$ be independent of $(\lambda,V)$ and exponentially distributed with unit mean.
Thus 
$\{\Sigma_T=\infty\}=\{\Sigma=\infty\}$ almost surely,
where 
$\Sigma:=\sum_{n\in\N}\varphi(E\lambda_n/T,V_n)$.
It is hence sufficient to prove the equivalence between
$\p(\Sigma<\infty)=1$ and  
$t^{-1}\phi(t)\in\Loc(0+)$.

Since $(E\lambda_n/T)_{n\in\N}$ is a stick-breaking process on the random interval $[0,E]$,~\cite[App.~A]{CM_Fluctuation_Levy}
and the marking theorem imply that 
$\Xi=\sum_{n=1}^\infty \delta_{(E\lambda_n/T,V_n)}$
%it follows that  %$\sum_{n\in\N}\varphi(E\lambda_n/T,V_n)<\infty$ a.s. %and $\Xi=\sum_{n=1}^\infty %\delta_{(E\lambda_n/T,V_n)}$ 
is a Poisson point process %(see~\cite[App.~A]{CM_Fluctuation_Levy}). %Moreover, 
with mean measure 
\begin{equation*}
%\label{eq:CM_measure_exponential}
\mu(\D t,\D u):=t^{-1}e^{-t} \D t\D u,\quad(t,x)\in\R_+\times[0,1].
\end{equation*}
Moreover, as
$\Sigma=\int_{\R_+\times[0,1]}\varphi(t,u)\Xi(\D t, \D u)$,
Campbell's formula~\cite[Lem.~12.2]{MR1876169}  implies
\begin{align}
\label{eq:CM_measure_exponential}
-\log\e\big[e^{- q\Sigma}\big] 
 =  \int_{\R_+\times[0,1]}\big(1-e^{- q\varphi(t,u)}\big)
\mu(\D t, \D u)
= \int_{\R_+\times[0,1]}\big(1-e^{- q\varphi(t,u)}\big)
e^{-t}\frac{\D t}{t}\D u.
\end{align}
%The integrand is integrable on $[1/q,\infty)\times[0,1]$ since it is bounded by $qe^{-qt}$. Hence, the integral in display is finite if and only if the integrand is integrable on $[0,1/q]\times[0,1]$. 
There exists $q_1>0$ such that for all $q\in(0,q_1]$
we have 
$0\leq \varphi(t,u)\le 1/q$ for all $(t,u)\in\R_+\times[0,1]$. 
%Since, by the assumption, $\Sigma<\infty$ a.s. the %integral is finite for all $q>0$.
%, since $\p(\Sigma_T<\infty)=\lim_{q\downarrow %0}\e[e^{-q \Sigma_T}]$ by dominated convergence. 
The elementary inequalities that implied~\eqref{eq:integral_bounds_Phi}
yield the following for all $q\in(0,q_1]$:
%for all sufficiently small $q>0$, satisfying %$\varphi(t,u)\le 1/q$ for all %$(t,u)\in\R_+\times[0,1]$, 
%we have
\begin{align}
\label{eq:bounds_on_integrals_CF}
(q/2)\int_0^\infty \phi(t)e^{-qt}\frac{\D t}{t}\leq
    \int_{\R_+\times[0,1]}\big(1-e^{- q\varphi(t,u)}\big) e^{-qt}\frac{\D t}{t}\D u
    \leq q\int_0^\infty \phi(t)e^{-qt}\frac{\D t}{t}.
\end{align}
The monotone convergence theorem 
yields
$\p(\Sigma<\infty)
=\lim_{q\downarrow0}\e[e^{-q\Sigma}]$.
Since
$\int_0^\infty t^{-1}\phi(t)e^{-t}\D t<\infty$
if and only if 
$t^{-1}\phi(t)\in\Loc(0+)$, 
\eqref{eq:CM_measure_exponential}-\eqref{eq:bounds_on_integrals_CF}
imply that 
$\p(\Sigma<\infty)=1$
if and only if $t^{-1}\phi(t)\in\Loc(0+)$, establishing~\eqref{eq:0-1-Law}.
%\begin{align*}
%\int_0^{1/q} \phi(t)\frac{\D %t}{t}=\int_{[0,1/q]\times[0,1]}\varphi(t,u)\frac{\D %t}{t}\D u\le \frac{2e^q}{q^2}\int_{[0,1/q]\times[0,1]}%\big(1-e^{- q\varphi(t,u)}\big)qe^{-qt}\frac{\D %t}{t}\D u
%<\infty,
%\end{align*} 
%since $qx/2\leq 1-e^{-qx}$ and $e^{-q}\le e^{-qx}$ for %$x\in[0,1/q]$. Since $\phi(t)/t$ is bounded on %$[1/q,1]$, the equivalence in~\eqref{eq:0-1-Law} %follows. 

Recall the identity
$\sum_{n=1}^\infty \e[f( \lambda_{n})]
=\int_0^T t^{-1}f(t)\D t$
for any measurable $f:[0,T]\to\R_+$, see e.g.~\cite[Eq.~(12)]{bang2021asymptotic}.
%The formula for the mean of %$\Sigma_T$ follows %from 
%The formula for $\e \Sigma_T$ %follows by conditioning 
Since the stick-breaking process
$(\lambda_n)_{n\in\N}$
and the iid sequence
$(V_n)_{n\in\N}$ are independent, the following holds
$\e\Sigma_T
=\sum_{n=1}^\infty\e\left[\varphi(\lambda_n,V_n)\right]
=\sum_{n=1}^\infty\e\left[\phi(\lambda_{n})\right]
=\int_0^T \phi(t)t^{-1}\D t$.
\end{proof}

Recall that $X$ is a L\'evy process (see \cite[Def.~1.6, Ch.~1]{MR3185174}), assumed not to be compound Poisson with drift. Let $F(t,x)\coloneqq \p(X_t\le x)$ for all $t\ge0$ and $x\in\R$. Let $u\mapsto F^{-1}(t,u)$ be the right-inverse of $x\mapsto F(t,x)$ for every $t>0$ and note that $X_t\eqd F^{-1}(t,U)$ for any $U\sim\U(0,1)$. 
The convex minorant of the path of $X$ on the interval $[0,T]$
is piecewise linear, with the set of length-height
pairs $\{(\ell_n,\xi_n)\,:\,n\in\N\}$ of the maximal faces having the same law as the following random set 
%faces of the convex minorant $C$ of a L\'evy process %$X$ on the interval $[0,T]$ states that 
\begin{equation}
\label{eq:Levy-SB}
\{(\ell_n,\xi_n)\,:\,n\in\N\}
\eqd 
\{(\lambda_n,F^{-1}(\lambda_n,V_n))\,:\,n\in\N\},
\end{equation}
see e.g.~\cite[Thm~11]{CM_Fluctuation_Levy}. 
This identity in law and Theorem~\ref{thm:0-1_law_formula} yield the following result.

\begin{cor}
\label{cor:slope_criteria}
Pick a measurable function $f:(0,\infty)\times\R\to \R$ and a measurable set $I\subseteq\R$. Then we have  %$A_n\coloneqq \{f(\ell_n,\xi_n)\in I\}$ for all %$n\in\N$. Then 
$\p(\{f(\ell_n,\xi_n)\in I\}\enskip\text{i.o.})\in\{0,1\}$ and 
\begin{equation}
\label{eq:slopes}
\p(\{f(\ell_n,\xi_n)\in I\}\enskip\text{i.o.})=1
\quad\iff\quad
%\enskip\iff\enskip
\int_0^1\p(f(t,X_t)\in I)\frac{\D t}{t}=\infty.
\end{equation}
\end{cor}

Note that Corollary~\ref{cor:slope_criteria} (with $f(t,x)=x/t$) directly implies Theorem~\ref{thm:CM_local_smoothness} above.

\begin{proof}
%Let $\mathcal{F}$ be as in Theorem~\ref{thm:0-1_law_formula} and note
%that $A_n\in\mathcal{F}_n$ for all $n\in\N$. 
Define the function $\varphi:(s,u)\mapsto\1_{\{f(s,F^{-1}(s,u))\in I\}}$. For any $U\sim\U(0,1)$ we have  
$$\phi(s)=\int_0^1\varphi(s,u)\D u= \p(f(s,F^{-1}(s,U))\in I)=\p(f(s,X_s)\in I)\quad\text{for all $s>0$.}$$
An application of Theorem~\ref{thm:0-1_law_formula} and~\eqref{eq:Levy-SB}  imply the corollary. 
\end{proof}

\begin{rem}%\label{rem:0-1_io}
\textbf{(i)} The equality in law in~\eqref{eq:Levy-SB}
follows from the representation theorem for convex minorants of L\'evy processes in~\cite[Thm~11]{CM_Fluctuation_Levy} because
$X$ is assumed not to be compound Poisson with drift.
Under this assumption, the law of $X_t$ (for any $t>0$) is diffuse (i.e. non-atomic) by Doeblin's lemma~\cite[Lem.~15.22]{MR1876169}, implying that no two linear segments in the piecewise linear convex function defined  in~\cite[Thm~11]{CM_Fluctuation_Levy}
have the same slope. Thus all these linear segments are maximal faces. The identity in law in~\eqref{eq:Levy-SB} is essentially the content of~\cite[Thm~1]{MR2978134}. Since the proof of~\cite[Thm~1]{MR2978134} is highly non-trivial and moreover relies on deep results in the fluctuation theory of L\'evy processes, we chose the route above based on~\cite[Thm~11]{CM_Fluctuation_Levy}, whose proof is short and elementary, requiring only the definition of a L\'evy process. 

\noindent \textbf{(ii)}
Corollary~\ref{cor:slope_criteria} can be used to determine the limit points of the random countable set $\mS_f\coloneqq \{f(\ell_n,\xi_n):n\in\N\}$. Indeed, the sets $\mL(\mS_f)$ and $\mL^\pm(\mS_f)$ are determined by the following countable family of events  $\{|\mS_f\cap(a,b)|=\infty\}=\{f(\ell_n,\xi_n)\in(a,b)\text{ \emph{i.o.}}\}$, where $a<b$ range over the rational numbers.
%and $\{f(\ell_n,\xi_n)\in(a,b)\enskip\text{\emph{i.o.}}\}$ agree a.s. %Moreover, $\mL(\mS_f)$ can be %determined by the countable set of %events $\{|\mS_f\cap(a,b)|=\infty\}%$ (and their complements) for %rational $a<b$. 
By Corollary~\ref{cor:slope_criteria}, the indicator of any such event is almost surely constant, making the limit sets $\mL(\mS_f)$ and $\mL^\pm(\mS_f)$ also constant almost surely. In particular, $\mL(\mS_f)$ and $\mL^\pm(\mS_f)$ are independent of $\mS_f$ itself and are not affected under conditioning on an event of positive probability. In par, we may modify the L\'{e}vy measure of $X$ by adding or removing a finite amount of mass anywhere on $\R\setminus\{0\}$ without altering $\mL(\mS_f)$ or $\mL^\pm(\mS_f)$. 
\end{rem}

\section{Continuous differentiability of the boundary of the convex hull -- proofs}
\label{sec:proofssection}

This section is dedicated to proving the results stated in Section~\ref{sec:introduction}. Let $\psi$ be the L\'evy--Khintchine exponent~\cite[Thm~8.1 \& Def.~8.2]{MR3185174} of the L\'evy process $X$, defined by $\psi(u)=t^{-1}\log\e \exp(iuX_t)$, for $t>0$, $u\in\R$, and satisfying
$\psi(u)=-\sigma^2u^2/2 + iu\gamma +\int_{\R}(e^{iux}-1-iux\1_{\{|x|<1\}})\nu(\D x)$
for $u\in\R$ for constants $\sigma\ge0$, $\gamma\in\R$ and L\'evy measure $\nu$ on $\R$ with $\int_{\R} \min\{1,x^2\}\nu(\D x)<\infty$ and $\nu(\{0\})=0$. The vector $(\gamma,\sigma^2,\nu)$ is known as the generating triplet of $X$ corresponding to the cutoff function $x\mapsto\1_{(-1,1)}(x)$, see~\cite[Def.~8.2]{MR3185174}.

%\subsection{Proofs of Lemma~\ref{lem:SB-P-trivial}, Propositions~\ref{prop:S_h} \&~\ref{prop:0-1_law_formula}, Corollary~\ref{cor:slope_criteria} and Corollary~\ref{cor:Khintchine}}

\subsection{Finite variation -- proofs}
\label{subsec:finite_var_proofs}
The L\'evy process $X$ has paths of finite variation if and only if $\sigma=0$ and $\int_{(-1,1)}|x|\nu(\D x)<\infty$. In this case one defines the natural drift $\gamma_0\in\R$ of $X$ by  
\begin{equation}
\label{eq:natural_drift}
\gamma_0 \coloneqq  \gamma-\int_{(-1,1)}x\nu(\D x).
\end{equation} 

\begin{proof}[Proof of Proposition~\ref{prop:fin_var_criteria}]
Since $X$ has finite variation, by~\cite[Thm~43.20]{MR3185174}
yields $\lim_{t\downarrow 0}X_t/t=\gamma_0$ a.s.
Hence the positive (resp. negative) half-line is not regular  for the process $(X_t-ct)_{t\geq0}$ if  $c>\gamma_0$ (resp. $c<\gamma_0$). Rogozin's criterion 
(see e.g.~\cite[Thm~6]{CM_Fluctuation_Levy})
then yields $\int_0^1t^{-1}\p(X_t-(\gamma_0+\ve)t>0)\D t<\infty$ and  $\int_0^1t^{-1}\p(X_t-(\gamma_0-\ve)t<0)\D t<\infty$ for all $\ve>0$. 
By Theorem~\ref{thm:CM_local_smoothness} we get that, for every $\ve>0$,  the set $\mS\cap (\R\setminus(\gamma_0-\ve,\gamma_0+\ve))$ is finite a.s.
Moreover,
since $X$ is of infinite activity, by~\cite[Thm~11]{CM_Fluctuation_Levy} and
Doeblin's  lemma~\cite[Lem.~15.22]{MR1876169}, the cardinality of the set of slopes $\mS$ is infinite.  Thus we conclude that $\mS$ is bounded a.s. Furthermore, $r\in\R\setminus\{\gamma_0\}$ cannot be in the limit set $\mL(\mS)$ a.s. and  $\gamma_0\in \mL(\mS)$ a.s. In particular, the set of slopes $\mS$ consists of isolated points with a single accumulation point $\mL(\mS)=\{\gamma_0\}$ satisfying  $\gamma_0\notin  \mS$ a.s. Hence $C'$, whose image is contained in the closure of $\mS$, is bounded and discontinuous on $\bigcup_{r\in\mS}\partial I_r$.

Without loss of generality we may assume that $X$ has right-continuous paths with left limits.
In particular, denote $X_{t-}:=\lim_{s\ua t}X_s$ if $t>0$ and $X_{0-}:=X_0$ otherwise.
Let  
$v$ be the last time in $[0,T]$ the process $(X_t-\gamma_0t)_{t\geq0}$ attains its minimum,
i.e.  $v$ is the greatest time in $[0,T]$ satisfying
$\min\{X_v,X_{v-}\}-\gamma_0 v=\inf_{t\in[0,T]}(X_t-\gamma_0t)$.
Since $t\mapsto C(t)-\gamma_0 t$ is the convex minorant on $[0,T]$ of $t\mapsto X_t-\gamma_0 t$, 
if the latter function attained its minimum at two or more times with positive probability, the former function, which is piecewise linear and convex, would have a face of slope zero with positive probability.
Since the increments of $X$ are diffuse by Doeblin's  lemma~\cite[Lem.~15.22]{MR1876169}, this contradicts the formula for the slopes in~\cite[Thm~11]{CM_Fluctuation_Levy}. Moreover, $v$ is the a.s. unique time at which the convex function $t\mapsto C(t)-\gamma_0 t$ on $[0,T]$ attains its minimum. %Since $\gamma_0\notin\mS$, we have $C'>\gamma_0$ on~$(v,T]$ and $C'<\gamma_0$ on $[0,v)$, implying that $v$ is the unique time at which $t\mapsto C(t)-\gamma_0t$ (and hence $t\mapsto X_t-\gamma_0t$) attains its minimum on $[0,T]$.

The probability $\p(v=0)=\p(0=\inf_{t\in[0,T]}(X_t-\gamma_0t))$ (resp. $\p(v=T)=\p(0=\inf_{t\in[0,T]}(\gamma_0t+X_T-X_{T-t}))= \p(0=\inf_{t\in[0,T]}(\gamma_0t-X_t))$) is positive if zero is not regular for $(X_t-\gamma_0t)_{t\geq0}$ (resp. $(-X_t+\gamma_0t)_{t\geq0}$) for the half-line $(-\infty,0)$, which is by Rogozin's criterion 
(see e.g.~\cite[Thm~6]{CM_Fluctuation_Levy}) equivalent to $I_-<\infty$ (resp. $I_+<\infty$).
In particular, $v\in(0,T)$ a.s. is equivalent to $I_+=I_-=\infty$. 
We proved above that $\gamma_0$ is the only limit point of $\mS$ a.s. Thus, by definition of 
$\mL^-(\mS)$ (resp. $\mL^+(\mS)$) in the paragraph containing~\eqref{eq:I_pm} above, $\gamma_0$ is a left (resp. right) limit point of $\mS$ if and only if the set $\mS\cap(-\infty,\gamma_0)$  (resp. $\mS\cap(\gamma_0,\infty)$) has infinitely many elements a.s., which is by Theorem~\ref{thm:CM_local_smoothness} equivalent to $I_-=\infty$ (resp. $I_+=\infty$).
Thus $I_\pm<\infty$ implies  $\mL^\pm(\mS)=\emptyset$ and $\mL^\mp(\mS)=\{\gamma_0\}$, where $\mp 1=-(\pm 1)$. Furthermore, if 
$\gamma_0$ is in $\mL^+(\mS)$ (resp. $\mL^-(\mS)$), then 
$C'(v)=\inf\{s>\gamma_0: s\in\mS\}=\gamma_0$
(resp. $C'(v-)=\sup\{s<\gamma_0:s\in\mS\}=\gamma_0$), where the infimum (resp. supremum) is necessarily taken over a non-empty set. If $\mL^+(\mS)=\emptyset$ (resp. $\mL^-(\mS)=\emptyset$), on the event $v\in(0,T)$ we have
$C'(v)=\inf\{s>\gamma_0: s\in\mS\}>\gamma_0$
(resp. $C'(v-)=\sup\{s<\gamma_0:s\in\mS\}<\gamma_0$), where the infimum (resp. supremum) is necessarily taken over a finite non-empty set. This concludes the proof of the proposition. 
\end{proof}

\begin{proof}[Proof of Corollary~\ref{cor:I_pm}]
The equivalence $I_\pm=\infty\iff\s_1\notin\Loc(\gamma_0\pm)$
follows from the equivalence in~\eqref{eq:vigon_identity}. 
We now prove that $I_+=\infty$ is equivalent to the integral condition in the corollary (the equivalence involving $I_-=\infty$ follows from this one by considering $-X$). We consider two cases:\\
%By~\cite[Thm~6.5(iii)]{MR2250061} (see also~\cite[Thm~1]{MR1465812}), $I_+=\infty$ 
\noindent {\bf(I)} $\nu((-\infty,0))<\infty$: then $\p(X_t>t\gamma_0)\ge \exp(-t\nu((-\infty,0)))$, implying $I_+=\infty$ by definition. The function $\varpi(x)\coloneqq x/\int_0^x \nu((-\infty,-y))\D y$, $x>0$, is bounded below by the positive constant $1/\nu((-\infty,0))\in(0,\infty]$, implying $\int_{(0,1)}\varpi(x)\nu(\D x)=\infty$ as $X$ is of infinite activity (thus $\nu((0,\infty))=\infty$). 

\noindent {\bf(II)} $\nu((-\infty,0))=\infty$: then $\varpi(x)$ is finite for $x>0$ and, by~\cite[Thm~1]{MR1465812}, $I_+=\infty$ is equivalent to 
%the condition on the L\'evy measure $\nu$ follows from~\cite[Thm~1]{MR1465812} and integration by parts. 
%Indeed, the cases $\nu((0,\infty))<\infty$ and 
$\int_{(0,1)}\nu((x,\infty))\D\varpi(x)=\infty$. Note further that $1/\varpi(x)=x^{-1}\int_0^x\nu((-\infty,-y))\D y$ is a non-increasing function since it is the average over the interval $(0,x)$ of the non-increasing function $y\mapsto\nu((-\infty,-y))$. This makes the Radon measure $\D\varpi(x)$ well defined. 
The function $\varpi$ is continuous on $(0,\infty)$ and, since $1/\varpi(x)\ge \nu((-\infty,-x))\to\infty$ as $x\da 0$, we have $\lim_{x\da 0}\varpi(x)=0$. Fubini's theorem yields
\begin{align*}
\int_{(0,1)} \nu((x,\infty))\D\varpi(x)
&=\int_{(0,1)}\int_{(x,\infty)} \nu(\D y)\D\varpi(x)
=\int_{(0,\infty)}\int_{(0,1\vee y)} \D\varpi(x)\nu(\D y)\\
&=\int_{(0,\infty)}\varpi(1\vee y)\nu(\D y)
=\nu([1,\infty)) \varpi(1)
+\int_{(0,1)} \varpi(y)\nu(\D y).
\end{align*}
Thus $\int_{(0,1)} \varpi(y)\nu(\D y)=\infty$ is equivalent to $\int_{(0,1)} \nu((x,\infty))\D\varpi(x)=\infty\iff I_+=\infty$~\cite[Thm~1]{MR1465812}. 
\end{proof}

\subsection{Infinite variation -- proofs}
\label{subsec:infinite_var_proofs}

%\begin{proof}[Proof of Proposition~\ref{prop:inf_var_criteria}]
%Since $X$ has infinite variation, $0$ is regular for both half-lines for $t\mapsto X_t-ct$ for any $c\in\R$~\cite[Thm~47.1]{MR3185174}. Hence, Rogozin's criterion~\cite[Thm~48.1]{MR3185174} and Corollary~\ref{cor:slope_criteria} imply that $|\mS\cap(c,\infty)|=\infty$ and $|\mS\cap(-\infty,c)|=\infty$ a.s. for any $c\in\R$, giving the result.
%\end{proof}

%\begin{proof}[Proof of Theorem~\ref{thm:CM_local_smoothness}]
%Corollary~\ref{cor:slope_criteria} gives the claim by choosing $f(t,x)=x/t$. 
%\end{proof}

\begin{proof}[Proof of Theorem~\ref{thm:CM_global_smoothness}]
First note that the smoothness of the boundary of the convex hull of $X$ requires $X$ to have infinite variation by Proposition~\ref{prop:fin_var_criteria}. Similarly, if $X$ is of finite variation, then~\eqref{eq:integral_slope} fails for any compact interval $I$ with $\gamma\notin I$. Thus, both conditions in Theorem~\ref{thm:CM_global_smoothness} require $X$ to have infinite variation, which we assume in the remainder of this proof.

Since $X$ is of infinite variation, the set of slopes $\mS$ is unbounded below and above by Rogozin's theorem as explained in the first paragraph of Subsection~\ref{subsec:infin_var_fin_time}. This makes the boundary of the convex hull of $X$ smooth at times $0$ and $T$. It remains to prove that the convex minorant $C$ of $X$ is continuously differentiable if and only if the condition~\eqref{eq:integral_slope} holds for all bounded intervals $I$. Recall that the right-derivative $C'$ is right-continuous by definition, and thus, its image equals $\mL^+(\mS)\cup\mS$ (see Table~\ref{tab:dictionary_C'} for all possible behaviours of the right-derivative of a piecewise linear convex function). 

Suppose the boundary of the convex hull of $X$ is smooth a.s., making $C'$ continuous a.s. By the intermediate value theorem, since $C'$ is unbounded from below and above, its image $\mL^+(\mS)\cup\mS$ must equal $\R$. Since $\mS$ is countable, $\mL^+(\mS)\cup\mS=\R$ a.s. implies $\mL^+(\mS)=\R$ a.s. Since $\mL^+(\mS)\subset\mL(\mS)$, we have $\mL(\mS)=\R$, implying that $\mS$ is dense in $\R$ a.s. and thus
condition~\eqref{eq:integral_slope} holds for all bounded intervals $I$. 

Now assume~\eqref{eq:integral_slope} holds for all bounded intervals $I$. Note that $\mL^+(\mS)\cap\mL^-(\mS)$ contains the interior of $\mL(\mS)$, so the condition $\mL(\mS)=\R$ implies $\mL^+(\mS)=\R$. Since $C'$ is right-continuous and non-decreasing with image $\mL^+(\mS)\cup\mS=\R$, it must be continuous, completing the proof.
\end{proof}

\begin{proof}[Proof of Proposition~\ref{prop:fin_var_perturbation}]
%We will prove that $x+b\in \mL(\mS_X)\iff x\in \mL(\mS_Z)$ (recall that $ \mL(\mS_X)$ and $\mL(\mS_Z)$ are both a.s. constant). 
%By scaling, we may assume without loss of generality that $r=1$. 
Recall that $Y$ and $Z$ are possibly dependent L\'evy processes, $X=Y+Z$ and $Y$ is of finite variation with natural drift $b$. Let $( \lambda_n)_{n\in\N}$ be an independent uniform stick-breaking process on $[0,T]$, defined recursively in terms of an i.i.d. $\U(0,1)$ sequence $(U_n)_{n\in\N}$ as follows: $L_0\coloneqq T$, $ L_n\coloneqq U_nL_{n-1}$ and $\lambda_n\coloneqq L_{n-1}- L_n$ for $n\in\N$. For $n\ge1$ define $\zeta_n^X\coloneqq X_{L_n}-X_{L_{n+1}}$, $\zeta_n^Y\coloneqq Y_{L_n}-Y_{L_{n+1}}$ and $\zeta_n^Z\coloneqq Z_{L_n}-Z_{L_{n+1}}$. By~\cite[Thm~11]{CM_Fluctuation_Levy}, the convex minorant of $X$ (resp. $Y$; $Z$) has the same law as the unique piecewise linear convex function with faces $((\lambda_n,\zeta^X_n)\,:\,n\in\N)$ (resp. $((\lambda_n,\zeta^Y_n)\,:\,n\in\N)$; $((\lambda_n,\zeta^Z_n)\,:\,n\in\N)$). In particular, the sets of slopes $\mS_X$, $\mS_Y$ and $\mS_Z$ have the same law as the sets $\{\zeta^X_n/\lambda_n:n\in\N\}$, $\{\zeta^Y_n/\lambda_n:n\in\N\}$ and $\{\zeta^Z_n/\lambda_n:n\in\N\}$, respectively, and hence share their limit sets. These limit sets must be constant a.s. by Theorem~\ref{thm:CM_local_smoothness}. In particular, $\mL(\{\zeta^Y_n/\lambda_n:n\in\N\})=\{b\}$ a.s. by Proposition~\ref{prop:fin_var_criteria}.
The result now follows from the fact that, for any deterministic sequences $(y_n)_{n\in\N}$ and $(z_n)_{n\in\N}$ with $\lim_{n\to\infty}y_n=b$, we have $\mL(\{y_n+z_n:n\in\N\})=\mL(\{z_n:n\in\N\})+b$. 
%Since $\mS_Y:=\{\xi^Y_n/\ell^Y_n\,:\,n\in\N\}$ is bounded and $\mL(\mS_Y)=\{b\}$ by Proposition~\ref{prop:fin_var_criteria}, then we a.s. have $\lim_{n\to\infty}\xi^Y_n/\ell^Y_n=b$ and thus $\lim_{n\to\infty}\zeta^Y_n/\lambda_n=b$ a.s. (and the same is true for any, even random, subsequence). If $x\in\mL(\mS_X)$ then $x\in\mL(\{\zeta_n^X/\lambda_n\,:\,n\in\N\})$ a.s. Hence, there is a random subsequence $(n_k)_{k\in\N}\subseteq\N$ such that $x=\lim_{k\to\infty}\zeta^X_{n_k}/\lambda_{n_k}$ a.s., but then 
%\[
%x=\lim_{k\to\infty}\big(\zeta^Y_{n_k}/\lambda_{n_k}
%+\zeta^Z_{n_k}/\lambda_{n_k}\big)
%=b+\lim_{k\to\infty}\zeta^Z_{n_k}/\lambda_{n_k}
%\quad\text{a.s.}
%\] 
%Therefore $x-b\in\mL(\{\zeta_n^Z/\lambda_n\,:\,n\in\N\})$ a.s. and hence $x-b\in\mL(\mS_Z)$. In other words, $x\in\mL(\mS_X)$ implies $x-b\in\mL(\mS_Z)$. An analogous argument shows the converse: $x\in\mL(\mS_Z)\implies x+b\in\mL(\mS_X)$, completing the proof.
\end{proof}

\begin{proof}[Proof of Proposition~\ref{prop:index_criteria}] 
%First note that the diffuseness of the law of $X$ implies that $\mS$ is a.s. infinite. Moreover, the centering of the heights in the definition of $\mS$ is equivalent to centering $X$ when it is of finite variation (see e.g.~\cite[Thm~43.20]{MR3185174} and Proposition~\ref{prop:fin_var_criteria} above).
Our assumption implies that $\s_1$ is finite and uniformly bounded. Indeed, by Lemma~\ref{lem:LevyBounds}(a), we obtain
\[
2\pi\s_1(r)=
\int_{\R}\Re\frac{1}{1+iru-\psi(u)}\D u
%=\frac{1-\Re\psi(u)}{(ru-\Im\psi(u))^2+(1-\Re\psi(u))^2}
\le\int_{\R}\frac{1}{1-\Re\psi(u)}\D u
\le\int_{\R}\frac{3}{1+u^2(\sigma^2+\ov\sigma^2(1/|u|))}\D u<\infty.
\]
Thus, Theorem~\ref{thm:CM_local_smoothness} and~\eqref{eq:vigon_identity} imply $\mL(\mS)=\emptyset$.

It remains to show that the assumption holds if $\sigma^2>0$ or $\beta_->1$. If $\sigma^2>0$ (resp. $\beta_->1$) fix some $\alpha\in(1,2)$ (resp. $\alpha\in(1,\beta_-)$) and note that, by the definition of $\beta_-$, there exists some $K>0$ such that $ u^2(\sigma^2+\ov\sigma^2(|u|^{-1}))\geq K|u|^{\alpha}$ for all $u\in\R\setminus(-1,1)$. Hence, we have 
\[
\int_1^\infty\frac{\D u}{1+u^2(\sigma^2+\ov\sigma^2(1/u))}
\le\int_1^\infty\frac{\D u}{1+Ku^{\alpha}}<\infty.
\qedhere
\] 
\end{proof}

\begin{proof}[Proof of Proposition~\ref{prop:asymmetry}]
By Proposition~\ref{prop:fin_var_perturbation}, we may assume without loss of generality that $x_0=1$, $\nu(\R\setminus(-1,1))=0$ and $\gamma=0$. Decompose $X=Y+Z$ where the L\'evy processes $Y$ and $Z$ are independent of each other and have generating triplets $(0,0,\nu|_{(0,1)})$ and $(0,0,\nu|_{(-1,0)})$, respectively. Let $\psi_Y$ and $\psi_Z$ be the characteristic exponents of $Y$ and $Z$, respectively. Note that $\psi=\psi_Y+\psi_Z$ and recall that the functions $\Re\psi$, $\Re\psi_Y$ and $\Re\psi_Z$ are even while $\Im\psi$, $\Im\psi_Y$ and $\Im\psi_Z$ are odd. The idea is to bound the function $\s_1(r)$ for $X$ uniformly over compact sets by the corresponding function for $Y$ (note that $Y$ is of infinite variation and creeps, making it abrupt by~\cite[Ex.~1.5]{MR1947963}). 

The assumption implies that $\int_{(0,1)}f(x)\nu(\D x)\ge c\int_{(-1,0)}f(-x)\nu(\D x)\ge 0$ for any measurable function $f:(0,1)\to[0,\infty)$. Thus, the following inequalities hold for all $u>0$:
\begin{equation}
\label{eq:asymmetry1}
|\Re\psi_Y(u)|\le|\Re\psi(u)|\le (1+1/c)|\Re\psi_Y(u)|,
\quad\text{and}\quad
\Im\psi(-u)\ge (1-1/c)\Im\psi_Y(-u)>0.
\end{equation}
Fubini's theorem and the infinite variation of $Y$ imply  $\int_{(0,1)}\nu([x,1))\D x
%=\int_{(0,1)}\int_{(0,y]}\D x\nu(\D y)
=\int_{(0,1)}y\nu(\D y)=\infty$. Moreover, for any $u>0$, Fubini's theorem yields
\begin{align*}
\frac{\Im\psi_Y(-u)}{u}
&=\frac{1}{u}\int_{(0,1)}(ux-\sin(ux))\nu(\D x)
\ge\frac{1}{u}\int_{(1/u,1)}(ux-1)\nu(\D x)\\
&=\int_{(1/u,1)}\int_{(1/u,x]}\D y\nu(\D x)
=\int_{(1/u,1)}\nu([y,1))\D y
\xrightarrow[u\to\infty]{}\infty.
\end{align*}
Fix any $R>0$ and let $M>0$ satisfy $\tfrac{1}{2}(1-1/c)|\Im\psi_Y(u)|\ge |u|R$ for all $|u|\ge M$. Then, by~\eqref{eq:asymmetry1}, for all $|u|\ge M$ and $|r|\le R$, we have
\[
|ur-\Im\psi(u)|
\ge |\Im\psi(u)|-|ur|
\ge (1-1/c)|\Im\psi_Y(u)|-|ur|
\ge \tfrac{1}{2}(1-1/c)|\Im\psi_Y(u)|.
\]
Hence~\eqref{eq:asymmetry1} gives $|1+iur-\psi(u)|\ge \tfrac{1}{2}(1-1/c)|1-\psi_Y(u)|$ for all $|u|\ge M$ and $|r|\le R$. Another application of~\eqref{eq:asymmetry1} then gives, for all $u\in\R$ and $r\in[-R,R]$, 
\[
\frac{1-\Re\psi(u)}{|1-iur-\psi(u)|^2}
\le 
\1_{\{|u|<M\}}
+
\1_{\{|u|\ge M\}}\frac{1+1/c}{(1-1/c)^2/4}
    \cdot\frac{1-\Re\psi_Y(u)}{|1-\psi_Y(u)|^2}.
\]
Since $Y$ is abrupt, the right-hand side of the display above is integrable over $(u,r)\in\R\times[-R,R]$ by Theorem~\ref{thm:CM_local_smoothness} and~\eqref{eq:vigon_identity}. Thus, $\s_1(r)$ is finite, uniformly bounded and integrable on $r\in[-R,R]$. Since $R>0$ was arbitrary, $X$ is abrupt.
\end{proof}

\begin{proof}[Proof of Theorem~\ref{thm:zeta_criterion}]
For the proofs of Parts (i) and (ii), we adopt the arguments given in~\cite{VigonConjecture}.\\
Part (i). Assume that there exists some $k\in (0,\infty)$ such that $(1+|\psi(u)|)/u\leq k$ for all $u\ge 1$. Recall that $\Re\psi(u)\leq 0$, and note that 
\begin{align*}
\int_1^\infty \Re \frac{1}{1+iur-\psi(u)}\D u 
&=\int_1^\infty \frac{1-\Re\psi(u)}{|1+iur-\psi(u)|^2}\D u \\
&\geq \int_1^\infty \frac{|\Re\psi(u)|}{(1+|iur|+|\psi(u)|)^2}\D u
%&=\int_1^\infty \frac{|\Re\psi(u)|}{(r+(q+|\psi(u)|)/|u|)^2u^2}\D u
\geq \frac{1}{(k+r)^2}\int_1^\infty \frac{|\Re\psi(u)|}{u^2}\D u.
\end{align*} 
Since $X$ has infinite variation the right hand side is always infinite by~\cite[Prop.~1.5.3]{vigon:tel-00567466}. Hence $\s_1(r)=\infty$ for all $r$, implying the claim. 

Part (ii). Suppose that $\liminf_{u \to \infty}|\psi(u)/u|=\infty$. It suffices to show that, if $\s_1(r_0)<\infty$ for some $r_0\in\R$, then $\sup_{r\in [r_0-R,r_0+R]}\s_1(r)<\infty$ for any $R>0$. Indeed, this would imply that either $\s_1(r)=\infty$ for all $r$ (making $X$ strongly eroded) or $\s_1(r)$ is bounded uniformly on compact sets (making $X$ abrupt). Suppose $\s_1(r_0)<\infty$ for some $r_0$ and fix $R>0$. By assumption, there exists some $M\ge 1$ such that $|\psi(u)/u|\ge 2R+3(|r_0|+1)$ for all $|u|\ge M$. Thus, for $|u|\ge M$ and $|r-r_0|\le R$, we have the inequalities $|\psi(u)|\ge 3|1+ir_0u|+2R|u|\ge 2|1+iru|+|1+ir_0u|$, and hence
\[
|1+ir_0u-\psi(u)|
\le |1+ir_0u|+|\psi(u)|
\le 2(|\psi(u)|-|1+iur|)
\le 2|1+iur-\psi(u)|.
\]
Thus, for any $r \in [r_0-R,r_0+R]$, we have
\[
\int_0^\infty \Re\frac{1}{1+iru-\psi(u)}\D u
%&=\int_1^M \Re\frac{1}{1+iru-\psi(u)}\D u+\int_M^\infty \Re\frac{1}{1+iru-\psi(u)}\D u\\
\leq \int_0^M \Re\frac{1}{1+iru-\psi(u)}\D u
+2\int_M^\infty \Re\frac{1}{1+ir_0u-\psi(u)}\D u,
\]
implying $\s_1(r)\le M/\pi + 2\s_1(r_0)$ since $\Re(1/(1-\psi(u)))\le 1$ for all $r,u\in\R$. 
%Since this is true for all $R>0$, the function above is finite everywhere and thus locally integrable. By dominated convergence it can be shown that $\s_1$ is continuous.

It remains to prove parts (ii-a)--(ii-c). By Part~(ii), $\s_1(r)$ is either everywhere finite and locally integrable, or $\s_1(r)=\infty$ for all $r$. Thus, in the remainder of the proof it suffices to check if $\s_1(0)<\infty$. Since $\Re(1/(1-\psi(u)))\le 1$  is locally integrable, $\Im\psi$ is an odd function and $\Re\psi$ is an even function, the finiteness of $\s_1(0)$ depends only on that of the following integral: 
\begin{align}
\label{eq:s_qintegrable_reduced}
    \int_1^\infty \frac{1-\Re \psi(u)}{(1-\Re\psi(u))^2+\Im\psi(u)^2}\D u.
\end{align}

Part (ii-a). Assume that $\lim_{|u| \to \infty}|\Re\psi(u)/u|=\infty$.  Since $\Re\psi(u)$ is super-linear we know that the denominator of the integrand in~\eqref{eq:s_qintegrable_reduced} is asymptotically equivalent to $|\Re\psi(u)|^2+|\Im\psi(u)|^2=|\psi(u)|^2$. Similarly the numerator of the integrand in~\eqref{eq:s_qintegrable_reduced} is asymptotically equivalent to $|\Re\psi(u)|$. Hence the integral in~\eqref{eq:s_qintegrable_reduced} is infinite if and only if $\int_1^\infty \Re(1/(1-\psi(u)))\D u=\infty$.

Part (ii-b). Assume now that the upper and lower limits of $|\Re\psi(u)/u|$ as $|u|\to \infty$ lie in $(0,\infty)$ and that $\lim_{|u|\to \infty}|\Im\psi(u)/u|=\infty$. In this case the denominator of the integrand in~\eqref{eq:s_qintegrable_reduced} is asymptotically equivalent to $|\Im\psi(u)|^2$ and the numerator of~\eqref{eq:s_qintegrable_reduced} is asymptotically sandwiched between multiples of $|u|$. Hence the integral in~\eqref{eq:s_qintegrable_reduced} is infinite if and only if $\int_1^\infty u/(1+|\Im\psi(u)|^2)^{-1}\D u=\infty$. 

Part (ii-c) Assume now $\lim_{|u|\to \infty}|\Re\psi(u)/u|=0$ and that $\lim_{|u|\to \infty}|\Im\psi(u)/u|=\infty$. In this case the denominator of the integrand in~\eqref{eq:s_qintegrable_reduced} is asymptotically equivalent to $|\Im\psi(u)|^2$. Hence the integral in~\eqref{eq:s_qintegrable_reduced} is infinite if and only if $\int_1^\infty (1-\Re \psi(u))(1+|\Im\psi(u)|^2)^{-1}\D u=\infty$. 
\end{proof}

\subsection{Infinite time horizon -- proofs} 
\label{subsec:proofs_infinity_horizon}

\begin{proof}[Proof of Proposition~\ref{prop:C'_infinity}]
Let $\Xi=\sum_{n\in\N}\delta_{(\ell_n,\xi_n)}$ be a Poisson point process with mean measure given by $\mu(\D t,\D x) =\1_{\{x/t<l\}}t^{-1}\p(X_t\in\D x)\D t$. By~\cite[Cor.~3]{MR2978134} and the convexity of $C_\infty$, the result will follow if we show that $\Xi(\{(t,x)\,:\, c\le x/t<l,\, t\ge 1\})=\infty$ a.s. for any $c<l$. Since $\Xi$ is Poisson, it suffices to show that its mean is infinite. To that end, we will prove that $\mu(\{(t,x)\,:\, x/t<c,\, t\ge 1\})<\infty$ and $\mu(\{(t,x)\,:\, x/t<l,\, t\ge 1\})=\infty$. 

Fix any $c<l$ and define the L\'evy process $X^{(c)}=(X^{(c)}_t)_{t\ge 0}\coloneqq (X_t-ct)_{t\ge 0}$. Since $\lim_{t\to\infty}X_t^{(c)}=\infty$ a.s. (by definition of $l$), Rogozin's criterion~\cite[Thm~48.1]{MR3185174} yields 
\[
\mu(\{(t,x) \,:\, x/t<c,\, t\ge 1\})
=\int_1^\infty \p\big(X_t^{(c)}<0\big)\frac{\D t}{t}<\infty.
\]
It remains to establish that $\mu(\{(t,x)\, :\, x/t<l,\, t\ge 1\})=\infty$. If we assume that $l=\infty$ then we have $\mu(\{(t,x) \,:\, x/t<l,\, t\ge 1\})=\int_1^\infty t^{-1}\D t=\infty$. Assume instead that $l<\infty$ and let $X^{(l)}$ be as before with $c=l$. In this case $\e X_1^{(l)}=0$, making $X^{(l)}$ recurrent by~\cite[Rem.~37.9]{MR3185174}, so the event $\{\liminf_{t\to\infty}X_t^{(l)}=\infty\}$ has probability $0$. Hence, Rogozin's criterion~\cite[Thm~48.1]{MR3185174} yields 
\[
\mu(\{(t,x) \,:\, x/t<l,\, t\ge 1\})
=\int_1^\infty \p\big(X_t^{(l)}<0\big)\frac{\D t}{t}=\infty.
\qedhere
\]
\end{proof}

\begin{proof}[Proof of Proposition~\ref{prop:S'_infinity}]
Fix any $s<l$ and pick an arbitrary  $c\in(s,l)$. By Proposition~\ref{prop:C'_infinity}, the random time $\tau\coloneqq\inf\{t>0:C'_\infty(t)>c\}$ is finite a.s. Note that $X_t\ge C'_\infty(t\wedge\tau)+c(t\vee\tau-\tau)$ for all $t\ge 0$, where $x\wedge y\coloneqq \min\{x,y\}$. Since the latter is a convex function, the maximality of convex minorants implies that the convex minorant $\wt C$ of $X$ on the time interval $[0,\tau+1]$ is equal to $C_\infty$ on the interval $[0,\tau]$ and that all the faces of either $\wt C$ or $C_\infty$ with slope smaller than $c$ lie on the interval $[0,\tau]$. Since the set of slopes $\wt\mS$ of the faces of $\wt C$ have the same left and right accumulation points as $\mS$ a.s.,~\eqref{eq:S_and_S_infty} follows.
\end{proof}

\section{Concluding remarks}
\label{sec:comments}

The probabilistic arguments used in the proofs of Theorem~\ref{thm:CM_local_smoothness} (see Section~\ref{sec:0-1lawsbres}) and Proposition~\ref{prop:fin_var_perturbation} (see Section~\ref{subsec:infinite_var_proofs}) strongly suggest that ``frequent'' visits of the process $(X_t/t)_{t\in(0,1]}$ to bounded intervals as $t\da 0$ play a major role in $X$ being strongly eroded. The \emph{time spent} during such visits, and not the number of visits, appears to be the key quantity for the following reasons.\\ 
(I) The integral $\int_0^1\p(X_t/t\in I)t^{-1}\D t$ in Theorem~\ref{thm:CM_local_smoothness}, which needs to be infinite if $X$ is to be strongly eroded, is equal to the mean of the (weighted) occupation measure $\mathcal{T}(I)\coloneqq\int_0^1\1_I(X_t/t)t^{-1}\D t$ of the interval $I$ corresponding to the process $(X_t/t)_{t\in(0,1]}$.\\
(II) For any abrupt process $X$, the process $(X_t/t)_{t\in(0,\ve]}$ visits every bounded interval infinitely many times for every $\ve>0$. Indeed, since $\s_1$ is locally integrable, $\s_1(r)$ is finite for a.e. $r\in\R$. Moreover, if $\s_1(r)<\infty$, then $0$ is regular for itself for the process $(X_t-rt)_{t\ge 0}$ and hence $(X_t/t)_{t\in(0,\ve]}$ visits $r$ infinitely often for every $\ve>0$. These visits, however, are brief since $X$ is abrupt and thus $\e\mathcal{T}(I)<\infty$ for all bounded intervals $I$.

In the finite variation case, our ability to obtain a complete picture of how and where smoothness of the derivative $C'$ fails is due to the fact that, for every open interval $I$, the process $(X_t/t)_{t\in(0,\ve]}$ spends all of the (resp. no) time in $I$ for all sufficiently small $\ve>0$ if the limit $\lim_{t\da0}X_t/t$ lies inside (resp. outside) of $I$. In order to establish Conjectures~\ref{conj:dichotomy} and~\ref{conj:vigon}, we would need a better understanding (in the infinite variation case) of how much time the process $(X_t/t)_{t\in(0,1]}$ spends on any bounded interval. Such a result would allow us to apply Theorem~\ref{thm:0-1_law_formula} above to obtain the conjectured dichotomy.
However, a result of this type appears to be delicate because the jumps of  $(X_t/t)_{t\in(0,\ve]}$ visit all bounded intervals infinitely many times for all $\ve>0$ whenever the positive and negative jumps of $X$ both have infinite variation. (Recall that if $\int_{(-1,0)}|x|\nu(\D x)<\infty$ or $\int_{(0,1)}x\nu(\D x)<\infty$, then the process $X$ creeps and is therefore abrupt.) Indeed, let $\Delta_t\coloneqq X_t-X_{t-}$ denote the jump of $X$ at time $t>0$ and let $\Xi=\sum_{\Delta_t\ne 0}\delta_{(t,\Delta_t)}$ be the Poisson measure on $(0,\infty)\times(\R\setminus\{0\})$ of the jumps of $X$ with mean measure $\Leb\otimes\nu$. For any $\ve>0$, the Poisson variable $\mathscr{T}(I)\coloneqq\sum_{t\in(0,\ve]}\1_I(\Delta_t/t)=\int_{(0,\ve]\times(\R\setminus\{0\})}\1_I(x/t)\Xi(\D t,\D x)$ is a.s. infinite for any interval $I=[a,b)\subset(0,\infty)$ since its mean is infinite: by Campbell's formula,
\[
\e\mathscr{T}([a,b))
=\int_{(0,\infty)}\int_0^\ve
    \1_{[a,b)}(x/t)\D t\nu(\D x)
%=\int_{(0,\infty)}\int_0^1
%    \1_{(x/b,x/a]}(t)\D t\nu(\D x)
\ge\big(\tfrac{1}{a}-\tfrac{1}{b}\big)\int_{(0,a\ve]}x\nu(\D x)
=\infty.
\]

Theorem~\ref{thm:CM_local_smoothness} can be rephrased as follows: for a given interval $I$, we have $|\mS\cap I|=\infty$ a.s. if and only if $\e\mathcal{T}(I)=\infty$ where we recall $\mathcal{T}(I)=\int_0^1\1_I(X_t/t)t^{-1}\D t$. In light of Theorem~\ref{thm:0-1_law_formula}, it is natural to speculate that something stronger is true, namely, $\e\mathcal{T}(I)=\infty$ if and only if $\mathcal{T}(I)=\infty$ a.s. We make the final observation that this occupation measure equals the total time $X_t/t$ spends in $I$ under an exponential change of variable: $\mathcal{T}(I)=\int_0^\infty\1_I(X_{e^{-u}}/e^{-u})\D u$. This emphasis on the time spent by $X_t/t$ over exponentially small times is in line with the geometric decay of the length of the sticks in the stick-breaking representation for the convex minorant $C$, the main tool in proving Theorem~\ref{thm:CM_local_smoothness}.

%\bibliographystyle{abbrv}
%\bibliography{Referencer}

\appendix
\section{Vertices, slopes and derivatives of piecewise linear convex functions}
\label{sec:pw-cf}

A point $x\in\R$ is an \textit{accumulation} (or \textit{limit}) \textit{point} of 
a set $\mathcal{A}\subset \R$
if every neighborhood of $x$ in $\R$ intersects $\mathcal{A}\setminus\{x\}$. Denote by $\mL(\mathcal{A})$ the set of all accumulation points in $\R$
of the set $\mathcal{A}$.
A point $x\in\R$ is a \emph{right-accumulation}  (or \emph{right-limit}) \emph{point} of $\mathcal{A}$ if every neighborhood of $x$ in $\R$ intersects $\mathcal{A}\cap(x,\infty)$. Denote by $\mL^+(\mathcal{A})$ the set of all right-accumulation points in $\R$
of the set $\mathcal{A}$. A set of \emph{left-accumulation}  (or \emph{left-limit}) \emph{points} of $\mathcal{A}$, denoted by $\mL^-(\mathcal{A})$, is defined analogously. Note that 
$\mL(\mathcal{A})=\mL^+(\mathcal{A})\cup\mL^-(\mathcal{A})$ with the intersection $\mL^+(\mathcal{A})\cap\mL^-(\mathcal{A})$ consisting of points in $\R$ that are limits
of a strictly decreasing and a strictly increasing sequence of elements in $\mathcal{A}$. Moreover, 
the closure $\ov{\mathcal{A}}$ of $\mathcal{A}$ in $\R$  equals $\mathcal{A}\cup\mL(\mathcal{A})$. 

Let $C:[0,T]\to\R$ be a piecewise linear convex function on a bounded interval. Let $\mS$ be the set of the slopes of the linear segments of $C$ and
$\{I_r:r\in\mS\}$ the family of maximal open intervals of constancy of the right-continuous derivative $C'$ of $C$. Denote  by $V:=\bigcup_{r\in\mS} \partial I_r$
a subset of $[0,T]$ consisting of all the boundary points of the intervals of constancy of $C'$. Both $\mS$ and $V$ are countable sets. Table~\ref{tab:dictionary_C'} describes all possible behaviours of the derivative $C'$.

\begin{table}[ht]
\begin{tabular}{|T|M|L|} 
	\hline
	Times $V$ 
	& Slopes $\mS$
	& Derivative $C'$  \\
	\hline
	 $v\in [0,T] \setminus \ov{V}$ 
	 & $C'(v) \in \mS$ 
	 & $C'$ is constant on a neighbourhood of $v$ \\
	 \hline 
	 $v \in V\setminus \mL(V)$ 
	 & $C'(v) \in \mS\setminus \mL(\mS)$ 
	 & $C'$ is equal to a constant on $(v-\ve,v)$ and a different constant on $[v,v+\ve)$ for some $\ve>0$ \\
	 \hline
	 \multirow{2}{3cm}{\centering \phantom{mmmmmmmmm} $v \in \mL^+(V)\setminus \mL^-(V)$ (thus $v \in V\cap[0,T)$)}
	 & $C'(v) \in \mL^+(\mS)\cap \mS$ 
	 & $C'$ is continuous at $v$; $C'(v)<C'(v+\delta)$ for all $\delta\in(0,T-v)$; $C'$ is constant on $(v-\ve,v]$ for some $\ve>0$ \\ \cline{2-3}
	 & $C'(v) \in \mL^+(\mS)\setminus \mS$ 
	 & $C'(v)<C'(v+\delta)$ for all $\delta\in(0,T-v)$;
	 if $v>0$, $C'(v)>C'(v-)$ and  $C'$ constant on $(v-\ve,v)$ for some $\ve>0$;    \\ \hline
	 \multirow{2}{3cm}{\centering 
	 \phantom{mmmmmmmmm}
	 $v \in \mL^-(V)\setminus \mL^+(V)$ (thus $v \in V\cap(0,T]$)} 
	 & $C'(v-) \in \mL^-(\mS)\cap \mS$ 
	 & $C'$ is continuous at $v$, $C'(v-\ve)<C'(v-)$ for any $\ve>0$ and $C'$ is constant on $[v,v+\ve)$ for some $\ve>0$ \\ \cline{2-3}
	 & $C'(v-) \in \mL^-(\mS)\setminus\mS$ 
	 & $C'(v-\ve)<C'(v-)$ for any $\ve>0$ and, if $v \ne T$, $C'(v)>C'(v-)$ with $C'$ constant on $[v,v+\ve)$ for some $\ve>0$ \\ \hline
	  \multirow{2}{3cm}{\centering
	 $v \in \mL^-(V) \cap \mL^+(V)$ (thus $v \notin V$)} 
	 & $C'(v) \in \mL^+(\mS) \setminus \mL^-(\mS)$ (and $C'(v) \notin \mS$) 
	 & $C'$ is discontinuous at $v$ with $C'(v-)<C'(v)<C'(v+\ve)$ for any $\ve>0$\\ \cline{2-3}
	 & $C'(v) \in \mL^-(\mS) \cap \mL^+(\mS)$ (and $C'(v) \notin \mS$) 
	 & $C'$ is continuous at $v$ with $C'(v-\ve)<C'(v)<C'(v+\ve)$ for any $\ve>0$\\\hline
\end{tabular}
\caption{The behaviours of $C'$ and of the sets of times $V$ and slopes $\mS$ of 
an arbitrary piecewise linear function $C:[0,T]\to\R$. The table exhausts all possibilities.
Recall that $C'$ is non-decreasing and right-continuous on $(0,T)$, 
defined by its limits on $\{0,T\}$,
$C'(0):=\lim_{t\downarrow0}C'(t)\in[-\infty,\infty)$
and 
$C'(T):=\lim_{t\ua0}C'(t)\in(-\infty,\infty]$,
and has left-limits $C'(v-):=\lim_{t\ua v} C'(t)$ for all $v\in(0,T]$.}
\label{tab:dictionary_C'}
\end{table}

\section{Characteristic exponent, infinite variation and the integrability of \texorpdfstring{$\s_p(r)$}{s}}
\label{sec:s_q}

\subsection{Integrability of \texorpdfstring{$\s_p(r)$}{s}}

Throughout the remainder of the section we assume $X$ to have infinite activity. Recall that $\psi$ is the characteristic exponent of $X$, satisfying $\exp(\psi(u))=\log\e[\exp(iuX_1)]$ for $u\in\R$. Define for any $p>0$ and $r\in\R$, 
\begin{equation*}%\label{eq:vigon_density}
\s_p(r) 
:= \frac{1}{2\pi}
\int_\R\Re\frac{1}{p+iur-\psi(u)}\D u.
%\qquad
%\z_p(r) 
%:= \int_\R\bigg|\frac{1}{p+iur-\psi(u)}\bigg|\D u.
\end{equation*} 
The relation $\s_p(r)\in(0,\infty]$ holds since, by $\Re\psi(u)\le 0$, the integrand in the definition of $\s_p(r)$ is positive for all $u\in\R$. Define for any $q>p>0$ the measures $\mu_p$ and $\mu_{p,q}$ given by 
\[
\mu_p(A)
\coloneqq\int_0^\infty\p(X_t/t\in A)e^{-pt}\frac{\D t}{t},
\qquad
\mu_{p,q}(A)
\coloneqq\int_0^\infty\p(X_t/t\in A)
    (e^{-pt}-e^{-qt})\frac{\D t}{t},
\]
for any measurable $A\subset\R$. We note here that both measures are diffuse since the law of $X_t/t$ is diffuse by Doeblin's lemma~\cite[Lem.~15.22]{MR1876169}. Moreover, $\mu_{p,q}(\R)<\infty$ for any finite $q>p>0$ since $t\mapsto (e^{-pt}-e^{-qt})/t$ is integrable on $(0,\infty)$, while clearly $\mu_p(\R)=\infty$
for any L\'evy process $X$. 
In fact, by Theorem~\ref{thm:CM_global_smoothness}, $X$ is strongly eroded if and only if $\mu_p(I)=\infty$ for all bounded intervals $I$ in $\R$.

\begin{rem}\label{rem:Vigon}
\nf{(i)} The measure $\mu_p(A)$ is equal to the expected value of $|\mS\cap A|$ when the time horizon $T$ is an independent exponential random variable with mean $1/p$.\\
\nf{(ii)} For any $q>0$ and $z,w\in\Complex$ with $\Re z,\Re w\ge 0$ and $q\ge |z-w|$ we have $\Re(1/(q+z))\le 8\Re(1/(q+w))$. Indeed, the inequality is
equivalent to $(q+\Re z)|q+w|^2
%=(q+\Re z)|q+w|^2
\le% 2c(q+\Re w)|q+z|^2=
2(q+\Re w)\cdot (2|q+z|)^2$, which follows from $|q+w|
=|q+z + (w-z)|
\le |q+z| + q
\le 2|q+z|$ and 
$q+\Re z
=q+\Re w + \Re(z-w)
%\le q+\Re w + q
\le 2(q+\Re w)$. Thus $(1/8)\s_q(r)\le \s_{p+q}(r)\le 8\s_q(r)$ for any $q\ge p>0$ and $r\in\R$, implying that the finiteness and local integrability of $\s_p$ do not dependent on $p\in(0,\infty)$.
\end{rem}

For L\'evy processes with bounded jumps, Theorem~\ref{thm:Vigon} below was established in~\cite[Thm~1.5]{VigonConjecture}. We extend this result to all infinite activity L\'evy processes. Our proof follows the same strategy as the one in~\cite{VigonConjecture} but is  shorter and has the advantage of being almost completely elementary, requiring only basic facts about Fourier inversion and Brownian motion. The key step in~\cite{VigonConjecture}, relying heavily on the fluctuation theory of L\'evy processes, is replaced by a simple Gaussian perturbation of the L\'evy process. Moreover, almost no potential theory is used in our proof. More specifically, we apply~\cite[Thm~2.6]{MR889378}  
only once to show that $\lim_{q\to\infty}\s_q(r)=0$ whenever $\s_p(r)<\infty$ for some $p>0$.

\begin{thm}
\label{thm:Vigon}
Suppose $X$ has infinite activity. Then for any $p\in(0,\infty)$ and $-\infty\le a<b\le\infty$, we have
\begin{equation}
\label{eq:Vigon}
\int_a^b\s_p(r)\D r = \mu_p((a,b)).
\end{equation}
\end{thm}

The equivalence~\eqref{eq:vigon_identity} is immediate from Theorem~\ref{thm:Vigon}. The proof of Proposition~\ref{prop:Vigon} below is elementary, requiring no knowledge of potential theory for L\'evy processes. Theorem~\ref{thm:Vigon} follows easily from Proposition~\ref{prop:Vigon} and~\cite{MR889378} as we will see below. Note that the isolated application of~\cite[Thm~2.6]{MR889378} is the only time potential theory is used in the proof of Theorem~\ref{thm:Vigon}. Moreover, fluctuation identities are not used in the proof of Theorem~\ref{thm:Vigon}, which is what one would expect since identity~\ref{eq:Vigon} involves only the marginal laws of $X$.

\begin{prop}
\label{prop:Vigon}
Suppose $X$ has infinite activity. For $p>0$ and $a,b\in\R$ with $a<b$ we have:\\
\nf{(a)} if $\int_a^b\s_p(r)\D r<\infty$, then for any $q\in[p,\infty)$ we have
\begin{equation}
\label{eq:Vigon2}
\int_a^b (\s_p(r)-\s_q(r))\D r
=\mu_{p,q}((a,b));
\end{equation}
\nf{(b)} if $\mu_p((a,b))<\infty$, then $\int_{a+\ve}^{b-\ve}\s_p(r)\D r<\infty$ for every $\ve\in(0,(b-a)/2)$.
\end{prop}

Note that~\eqref{eq:Vigon2}, applied to every open subinterval of $(a,b)$, implies $\s_p\ge\s_q$ a.e. on the interval $(a,b)$ for any $q\in[p,\infty)$. We now show that Proposition~\ref{prop:Vigon} implies Theorem~\ref{thm:Vigon}.

\begin{proof}[Proof of Theorem~\ref{thm:Vigon}]
First assume $a,b\in\R$ and $\int_a^b\s_p(r)\D r<\infty$. Then $\s_p$ is finite a.e. on $(a,b)$. If $\s_p(r)<\infty$, then, by~\cite[Thm~2.6]{MR889378}, we have $\s_q(r)\to 0$ as $q\to\infty$ since $\s_q(r)$ is the reciprocal of the $q$-capacity $c^q_r$ of the set $\{0\}$ (see~\cite[Def.~43.1]{MR3185174} for definition). 
Since, by~\eqref{eq:Vigon2} we have
$\s_q\le\s_p$ a.e. on the interval $(a,b)$ for any $q\in[p,\infty)$,
the monotone convergence theorem (along a countable sub-sequence)
implies 
$\int_a^b (\s_p(r)-\s_q(r))\D r\ua\int_a^b\s_p(r)\D r$ as $q\to\infty$.
Again, by monotone convergence, we have $\mu_{p,q}((a,b))\ua \mu_{p}((a,b))$  as $q\to\infty$, implying the identity $\mu_p((a,b))=\int_a^b\s_p(r)\D r$ by~\eqref{eq:Vigon2}. 

Next we show that, for $a,b\in\R$, $\int_a^b\s_p(r)\D r=\infty$ implies $\mu_p((a,b))=\infty$. 
%We prove the equivalent statement: if $\mu_p((a,b))<\infty$, then $\int_a^b\s_p(r)\D %r<\infty$. 
Suppose $\mu_p((a,b))<\infty$. Then 
$\int_{a+1/n}^{b-1/n}\s_p(r)\D r<\infty$ for all sufficiently large $n\in\N$ by Proposition~\ref{prop:Vigon}(b). Hence, \eqref{eq:Vigon} holds over every interval $(a+1/n,b-1/n)$. Since $\s_p\ge 0$, taking $n\to\infty$ and applying the monotone convergence theorem gives~\eqref{eq:Vigon} over the interval $(a,b)$, completing the proof for $a,b\in\R$.

Take any real sequences $a_n\da a$ and $b_n\ua b$ as $n\to\infty$. Since~\eqref{eq:Vigon} holds over the intervals $(a_n,b_n)$, the monotone convergence theorem implies~\eqref{eq:Vigon} over the possibly infinite interval $(a,b)$.
\end{proof}

The following result can be deduced from the results in~\cite[Sec.~42]{MR3185174} on the potential theory of L\'evy processes. Since Lemma~\ref{lem:potential_density} is key in the proof of Proposition~\ref{prop:Vigon}, we include an elementary short proof for completeness.

\begin{lem}
\label{lem:potential_density}
Suppose the L\'evy process $X$ is of infinite activity and $\s_p(r)<\infty$ for some $p>0$, $r\in\R$. Then, for any $\ve>0$ we have
\begin{equation}
\label{eq:U_p}
\frac{1}{2\pi}\int_\R\frac{\sin(u\ve)}{u\ve}
    \Re\frac{1}{p+iur-\psi(u)}\D u
=\frac{1}{2\ve}
    \int_0^\infty 
        \p(X_t-rt\in(-\ve,\ve))e^{-pt}\D t.
\end{equation}
In particular, the following limit holds
\begin{equation}
\label{eq:potential_density}
\s_p(r)
=\lim_{\ve\da 0}\frac{1}{2\ve}
    \int_0^\infty 
        \p(X_t-rt\in(-\ve,\ve))e^{-pt}\D t,
\qquad p>0.
\end{equation}
\end{lem}

Note from~\eqref{eq:potential_density} that $p\mapsto\s_p(r)$ is a non-increasing function for each $r\in\R$.

\begin{proof}[Proof of Lemma~\ref{lem:potential_density}]
Since $(X_t-rt)_{t\geq0}$ is of infinite activity, we may assume without loss of generality that $r=0$. Define the measure $U_p(\D x):=\int_0^\infty\p(X_t\in \D x)e^{-pt}\D t$ on $\R$.
Note that 
$U_p(\D x)$ is diffuse (by Doeblin's lemma~\cite[Lem.~15.22]{MR1876169}) 
%finite and diffuse (by Doeblin's lemma~\cite[Lem.~15.22]{MR1876169}) 
%given by $U_p(A)=\int_0^\infty\p(X_t\in A)e^{-pt}\D t$ for measurable $A\subset\R$. %We aim to show that $U_p$ is absolutely continuous with density at $0$ given by $\s_p(0)$. 
and, by Fubini's theorem, the Fourier transform of $U_p$ equals
$\int_\R e^{iux}U_p(\D x)
=\int_0^\infty e^{-(p-\psi(u))t}\D t
=1/(p-\psi(u))$. Fourier inversion formula~\cite[Thm~26.2]{MR1324786} and Fubini's theorem yield
\[
U_p((-\ve,\ve))
=\lim_{c\to\infty}\frac{1}{2\pi}
    \int_{-c}^c\Re\bigg[\frac{e^{iu\ve}-e^{-iu\ve}}{iu}
    \frac{1}{p-\psi(u)}\bigg]\D u
=\lim_{c\to\infty}\frac{1}{\pi}
    \int_{-c}^c\frac{\sin(u\ve)}{u}
    \Re\frac{1}{p-\psi(u)}\D u.
\]
Since $0<\Re(1/(p-\psi(u)))^2\leq \Re(1/(p-\psi(u)))/p$ for all $u$ and $\s_p(r)<\infty$, the function is $u\mapsto \Re(1/(p-\psi(u)))$ square-integrable. Since $u\mapsto\sin(u\ve)/u$ is also square-integrable, their product is integrable by Cauchy--Schwarz. Thus, $U_p((-\ve,\ve))
=\pi^{-1}\int_\R(\sin(u\ve)/u)\Re(1/(p-\psi(u)))\D u$ for any $\ve>0$, implying~\eqref{eq:U_p}. Since $|\sin(x)/x|\le 1$ and $\Re(1/(p-\psi(u)))\ge 0$ is integrable, taking $\ve\da 0$ in~\eqref{eq:U_p} and applying the dominated convergence theorem gives~\eqref{eq:potential_density}.
\end{proof}

\begin{proof}[Proof of Proposition~\ref{prop:Vigon}]
(a). Since $\s_p$ is integrable on $(a,b)$, it is finite a.e. on $(a,b)$. By Remark~\ref{rem:Vigon}(ii), for each $r\in(a,b)$ with $\s_p(r)<\infty$  we have $\s_q(r)<\infty$ fir all $q>0$. Hence, by Lemma~\ref{lem:potential_density}, 
%\[
$\s_{q}(r)
\ge \int_0^\infty 
        \p(X_t-rt\in(-\ve,\ve))e^{-qt}\D t\to \s_q(r)$ as $\ve\da 0$.
%\]
Thus, the dominated convergence theorem and Fubini's theorem give 
\begin{align*}
\int_a^b\s_{q}(r)\D r
&=\int_a^b\lim_{\ve\da 0}\frac{1}{2\ve}
    \int_0^\infty 
        \p(X_t-rt\in(-\ve,\ve))e^{-qt}\D t\D r\\
&=\lim_{\ve\da 0}\int_0^\infty\frac{1}{2\ve}\e\bigg[\int_a^b 
        \1_{((X_t-\ve)/t,
            (X_t+\ve)/t)}(r)\D r\bigg]e^{-qt}\D t.
\end{align*}
The random variable $\tfrac{1}{2}(t/\ve)\int_a^b \1_{((X_t-\ve)/t,(X_t+\ve)/t)}(r)\D r$ is bounded by $1$ and converges to $\1_{(a,b)}(X_t/t)+\tfrac{1}{2}\1_{\{a,b\}}(X_t/t)$ as $\ve\da 0$, which equals $\1_{(a,b)}(X_t/t)$ a.s. since $X$ has infinite activity. Since the function $t\mapsto (e^{-pt}-e^{-qt})/t$ is integrable on $(0,\infty)$ for any $q>p$, the dominated convergence theorem implies  
\begin{multline*}
\int_a^b(\s_p(r)-\s_{q}(r))\D r
=\lim_{\ve\da 0}\int_0^\infty\e\bigg[\frac{t}{2\ve}\int_a^b 
        \1_{(X_t-\ve,
            X_t+\ve)}(rt)\D r\bigg](e^{-pt}-e^{-qt})\frac{\D t}{t}\\
=\int_0^\infty\e\big[\1_{(a,b)}(X_t/t)
        \big](e^{-pt}-e^{-qt})\frac{\D t}{t}
=\int_0^\infty\p(X_t/t\in(a,b))(e^{-pt}-e^{-qt})
    \frac{\D t}{t},
\end{multline*}
establishing~\eqref{eq:Vigon2}.

(b). Assume $\mu_p((a,b))<\infty$ and that there exists some $\ve>0$ with $\int_{a+\ve}^{b-\ve}\s_p(r)\D r=\infty$. We now show that these assumptions lead to a contradiction. Let Brownian motion $B$ be independent of $X$. The characteristic exponent of $X+\vs B$ equals $u\mapsto\psi(u)-\vs^2u^2/2$ for any $\vs>0$. For any $q>0$ let
\[
\s_q^{\vs}(r)\coloneqq\frac{1}{2\pi}\int_\R\Re\frac{1}{q+iur-\psi(u)+\vs^2u^2/2}\D u
%\qquad
%\z_p^{\vs}(r)\coloneqq\int_\R\frac{1}{|p+iur-\psi(u)+\vs^2u^2/2|}\D u.
\quad\text{and note}\quad
0<\int_a^b\s_q^{\vs}(r)\D r\leq \frac{b-a}{2\pi}\int_{\R}\frac{\D u}{q+\vs^2u^2/2}<\infty.
\]
 Thus~\eqref{eq:Vigon2} holds for the process $X+\vs B$, the interval $(a+\ve,b-\ve)$ and any $q>p$. Since, by the monotone convergence theorem, the upper bound in the last display tends to zero as $q\to\infty$, the monotone convergence theorem applied to the right-hand side of~\eqref{eq:Vigon2} yields
%Note that $\z_p^{\vs}\ge 2\pi\s_p^{\vs}>0$ and $\z_p^{\vs}$ is bounded on $\R$. %Moreover, the integrand of in the definition of $\z_p^\vs$ is decreasing in $p$. %Thus, for all $r\in\R$ we have $2\pi\s_p^{\vs}(r)\le\z_p^{\vs}(r)\da 0$ as %$p\to\infty$ by monotone convergence. Hence, taking $q\to\infty$ %in~\eqref{eq:Vigon2} (for the process $X+\vs B$ and the interval $(a+\ve,b-\ve)$) %and using the monotone convergence theorem yields 
\begin{equation}
\label{eq:Vigon_eBM}
\int_{a+\ve}^{b-\ve}\frac{1}{2\pi}\int_\R
    \Re\frac{1}{p+iur-\psi(u)+\vs^2u^2/2}\D u\D r
=\int_0^\infty\p((X_t+\vs B_t)/t\in (a+\ve,b-\ve))
    e^{-pt}\frac{\D t}{t}.
\end{equation}

Since we assumed $\int_{a+\ve}^{b-\ve}\s_p(r)\D r=\infty$, then for every $M>0$ there exist some $K>0$ such that 
$(2\pi)^{-1}\int_{a+\ve}^{b-\ve}
\int_{-K}^K\Re(1/(p+iur-\psi(u)))\D u\D r\ge 2M$. The bound $\Re(1/(p+iur-\psi(u)+\vs^2u^2/2))\le 1/p$ and the dominated convergence theorem (applied as $\vs\da 0$) give 
\[
\frac{1}{2\pi}\int_{a+\ve}^{b-\ve}
\int_{-K}^K\Re\frac{1}{p+iur-\psi(u)+\vs^2u^2/2}\D u\D r\ge M,
\]
for all sufficiently small $\vs>0$. Since $M>0$ is arbitrary, this implies that the integral on the right side of~\eqref{eq:Vigon_eBM} diverges as $\vs\downarrow0$. 

To complete the proof, we show that the assumption $\mu_p((a,b))<\infty$ implies that the integral on the right side of~\eqref{eq:Vigon_eBM} is bounded as $\vs\downarrow0$. We will first bound the integral on $[\vs^2,\infty)$. Note that 
\[
\p((X_t+\vs B_t)/t\in (a+\ve,b-\ve))
\le\p(|\vs B_t/t|\ge\ve)
    +\p(X_t/t\in (a,b)).
\]
By assumption, the integral $\int_{\vs^2}^\infty\p(X_t/t\in (a,b))e^{-pt}t^{-1}\D t$ is finite and converges to $\mu_p((a,b))<\infty$ as $\vs\da0$. The elementary bound $\p(|B_1|\ge x)\le e^{-x^2/2}/(\sqrt{2\pi}x)$ implies that, 
\begin{align*}
\int_{\vs^2}^\infty \p(|\vs B_t/t|\ge\ve)e^{-pt}\frac{\D t}{t}
%&=\int_{\vs^2}^\infty \p(|B_1|\ge\ve\sqrt{t}/\vs)e^{-pt}\frac{\D t}{t}
=\int_{1}^\infty \p(|B_1|\ge\ve\sqrt{t})e^{-p\vs^2t}\frac{\D t}{t}%\\
%&\le\int_1^\infty \frac{e^{-(2\vs^2p+\ve^2)t/2}}{\ve\sqrt{2\pi}}\frac{\D t}{t^{3/2}}
\le\int_{1}^\infty \frac{e^{-\ve^2t/2}}{\ve\sqrt{2\pi}}\frac{\D t}{t^{3/2}}<\infty.
\end{align*}

It remains to bound the integral on the right side of~\eqref{eq:Vigon_eBM} over the interval $(0,\vs^2)$ as $\vs\da 0$. To do this, note that, for any $c<d$, 
\[
\sup_{x\in\R}\p(\vs B_t/t+x\in(c,d))
=\p\big(|B_1|\le (d-c)\sqrt{t}/(2\vs)\big)
\le(d-c)\sqrt{t}/\big(\vs\sqrt{2\pi}\big).
\] 
Thus, elementary inequalities yield
\begin{align*}
\int_0^{\vs^2}\p((X_t+\vs B_t)/t\in (a+\ve,b-\ve))\frac{\D t}{t}
&=\int_0^{\vs^2}\int_\R\p(\vs B_t/t+x\in (a+\ve,b-\ve))\p(X_t/t\in\D x)\frac{\D t}{t}\\
&\le\int_0^{\vs^2}\frac{(b-a-2\ve)\sqrt{t}}{\vs\sqrt{2\pi}}\frac{\D t}{t}
\le\int_0^{1}\frac{(b-a-2\ve)}{\sqrt{2\pi}}\frac{\D t}{\sqrt{t}}<\infty.
\end{align*}
Hence, the right side of~\eqref{eq:Vigon_eBM} is bounded as $\vs\downarrow0$, completing the proof.
\end{proof}

\subsection{Characterisation of infinite variation} The following lemma is proved in \cite[Prop.~1.5.3]{vigon:tel-00567466}. The basic idea for its proof is already present in~\cite{MR0368175}, see the first display on page 34 of~\cite{MR0368175}. As this lemma is important for the examples in this paper, we give a proof below. 

\begin{lem}[{\cite[Prop.~1.5.3]{vigon:tel-00567466}}]
\label{lem:VigonThesis}
Let $\psi$ be the characteristic exponent of a L\'evy process $X$. Then the following equivalence holds:  $\int_1^\infty u^{-2}|\Re\psi(u)|\D u=\infty$ if and only if $X$ has paths of infinite variation.
\end{lem}

\begin{proof}
If the Gaussian component $\sigma^2>0$, the integral in the lemma is infinite and $X$ is of infinite variation. We thus assume $\sigma^2=0$.
Since the compound Poisson process composed of the jumps of $X$ of magnitude at least $1$ has a bounded characteristic function, we may assume that the L\'evy measure $\nu$ of $X$ is supported in the interval $(-1,1)$. Recall that $\ov\nu(x)=\nu((-1,1)\setminus(-x,x))$. Define $\wt\nu(x)\coloneqq\int_x^1\ov\nu(y)\D y$ for $x\in[0,1)$ and $0$ otherwise. By Fubini's theorem we get $\wt\nu(0)=\int_{(-1,1)}|x|\nu(\D x)$. Moreover, for any twice differentiable function $f:[0,1)\to[0,\infty)$ satisfying $f(0)=f'(0+)=0$, Fubini's theorem implies
\begin{align*}
\int_{(-1,1)}f(|x|)\nu(\D x)
=\int_{(-1,1)}\int_{0}^{|x|} f'(y)\D y\nu(\D x)
=\int_0^1 f'(y)\ov\nu(y)\D y
%\\
%&=\int_0^1\int_0^y
%    f''(z)\ov\nu(y)\D z\D y
%=\int_0^1\int_z^1
%    f''(z)\ov\nu(y)\D y\D z
=\int_0^1f''(z)\wt\nu(z)\D z.
\end{align*}
The choice $f(x)=x^2$ yields $2\int_0^1\wt\nu(x)\D x=\int_{(-1,1)}x^2\nu(\D x)$, implying that $\wt\nu$ is integrable. Similarly, the choice $f(x)=(1-\cos(ux))/u^2$ gives $u^{-2}|\Re\psi(u)|=\int_{(-1,1)}u^{-2}(1-\cos(ux))\nu(\D x)=\int_0^1\wt\nu(z)\cos(uz)\D z$.
Fix $\lambda\in(0,\infty)$, integrate the last identity on $(0,\lambda)$ and apply Fubini's theorem again to obtain
\[
\int_0^\lambda\frac{|\Re\psi(u)|}{u^2}\D u
=\int_0^1\frac{\sin(\lambda x)}{x}\wt\nu(x)\D x.
\]
Note that the integrand is integrable since $|\sin(\lambda x)/x|\leq \lambda$ for all $x\in\R$ and  $\wt\nu$ is integrable. Recall that $\wt\nu(x)=0$ for $x\in\R\setminus[0,1)$. Hence, Fourier's single-integral formula~\cite[Thm~12, p.~25]{MR942661} shows that 
$\int_{(0,\infty)}u^{-2}|\Re\psi(u)|\D u
=\lim_{\lambda\to\infty}\int_\R x^{-1}\sin(\lambda x)\wt\nu(x)\D x=(\pi/2)\wt\nu(0)=(\pi/2)\int_{(-1,1)}|x|\nu(\D x)$. This quantity is infinite if and only if $X$ is of infinite variation, completing the proof.
\end{proof}

\section*{Acknowledgements}
JGC and AM are supported by EPSRC grant EP/V009478/1 and The Alan Turing Institute under the EPSRC grant EP/N510129/1; 
AM was supported by the Turing Fellowship funded by the Programme on Data-Centric Engineering of Lloyd's Register Foundation; DB is funded by the CDT in Mathematics and Statistics at The University of Warwick.

\printbibliography

\end{document}